\documentclass[11pt,oneside,reqno]{article}
\usepackage[margin=1.3in]{geometry}
\usepackage{amsmath,amsthm,amssymb,color,bm}
\usepackage[authoryear]{natbib}
\usepackage{booktabs}

\newcommand{\tstar}{
    \begin{tikzpicture}[scale=0.5]
        \draw (0,0) -- (0,0.4);
        \draw (-0.25,0) -- (0,0.4);
        \draw (0.25,0) -- (0,0.4);
    \end{tikzpicture}
}

\newcommand{\td}{\widetilde}

\newcommand{\ptdH}{\partial_{ij}\bar{H}}

\def\d{\,\mathrm{d}}
\usepackage{tikz}
\RequirePackage{amsthm,amsmath,amsfonts,amssymb,mathrsfs,dsfont,mathtools,thmtools}
\RequirePackage[colorlinks,citecolor=blue,urlcolor=blue,linkcolor=blue]{hyperref}
\usepackage[capitalize]{cleveref}
\RequirePackage{graphicx}

\crefname{equation}{}{}
\crefformat{equation}{\textup{(#2#1#3)}}
\crefrangeformat{equation}{\textup{(#3#1#4)--(#5#2#6)}}
\crefdefaultlabelformat{#2\textup{#1}#3}
\crefname{lemma}{Lemma}{Lemmas}
\crefname{page}{p.}{pp.}
\usepackage[normalem]{ulem}
\usepackage{bbm}
\usepackage{mathrsfs}
\usepackage{graphicx}
\usepackage{tikz}
\usepackage{enumitem}
\usetikzlibrary{arrows, automata}
\usetikzlibrary{calc}
\usetikzlibrary{positioning}

\numberwithin{equation}{section}
\allowdisplaybreaks[4]

\theoremstyle{plain}
\newtheorem{theorem}{Theorem}[section]
\newtheorem{proposition}{Proposition}[section]
\newtheorem{lemma}{Lemma}[section]

\theoremstyle{definition}

\newtheorem{remark}{Remark}[section]

\newtheorem{conj}{Conjecture}[section]

\makeatletter

\newcount\minute
\newcount\hour
\newcount\hourMins
\def\now{%
\minute=\time%
\hour=\time \divide \hour by 60%
\hourMins=\hour \multiply\hourMins by 60%
\advance\minute by -\hourMins%
\zeroPadTwo{\the\hour}:\zeroPadTwo{\the\minute}%
}
\def\zeroPadTwo#1{\ifnum #1<10 0\fi#1}

\renewcommand{\cite}{\citet}

\def\^#1{\ifmmode {\mathaccent"705E #1} \else {\accent94 #1} \fi}
\def\~#1{\ifmmode {\mathaccent"707E #1} \else {\accent"7E #1} \fi}

\def\*#1{#1^\ast}
\edef\-#1{\noexpand\ifmmode {\noexpand\bar{#1}} \noexpand\else \-#1\noexpand\fi}
\def\>#1{\vec{#1}}
\def\.#1{\dot{#1}}

\def\atop{\@@atop}
\def\*#1{\mathscr{#1}}

\renewcommand{\leq}{\leqslant}
\renewcommand{\geq}{\geqslant}

\newcommand{\eq}{\eqref}

\def\t#1{\widetilde{#1}}

\newcommand{\IE}{\mathbbm{E}}
\newcommand{\IP}{\mathbbm{P}}
\newcommand{\Var}{\mathop{\mathrm{Var}}\nolimits}

\newcommand{\IR}{\mathbb{R}}

\def\be#1{\begin{equation*}#1\end{equation*}}
\def\ben#1{\begin{equation}#1\end{equation}}
\def\bes#1{\begin{equation*}\begin{split}#1\end{split}\end{equation*}}
\def\besn#1{\begin{equation}\begin{split}#1\end{split}\end{equation}}

\def\mid{\vert}

\def\beqn#1\eeqn{\begin{align}#1\end{align}}
\def\beq#1\eeq{\begin{align*}#1\end{align*}}

\usepackage{graphicx}
\usepackage{latexsym}
\usepackage{amsmath,amsthm,amssymb,amscd}
\usepackage{epsf,amsmath}

\def\E{{\IE}}
\def\P{{\IP}}

\def\d{\,\mathrm{d}}

\usepackage{datetime}

\renewcommand\section{\@startsection {section}{1}{\z@}%
{-3.5ex \@plus -1ex \@minus -.2ex}%
{1.3ex \@plus.2ex}%
{\center\small\sc\mathversion{bold}}}
\def\subsection#1{\@startsection {subsection}{2}{0pt}%
{-3.5ex \@plus -1ex \@minus -.2ex}%
{1ex \@plus.2ex}%
{\bf\mathversion{bold}}{#1}}

\def\subsubsection#1{\@startsection{subsubsection}{3}{0pt}%
{\medskipamount}%
{-10pt}%
{\normalsize\itshape}{\kern-2.2ex. #1.}}

\def\blfootnote{\xdef\@thefnmark{}\@footnotetext}

\makeatother

\begin{document}

\title{Conditional central limit theorems for exponential random graphs}
\author{Xiao Fang$^1$, Song-Hao Liu$^2$, Zhonggen Su$^3$, Xiaolin Wang$^3$}
\date{\it $^1$The Chinese University of Hong Kong, $^2$Southern University of Science and Technology, $^3$Zhejiang University \\[2ex]  }
\maketitle

\noindent{\bf Abstract:} 
In this paper, we study the Exponential Random Graph Models (ERGMs) conditioning on the number of edges. 
In subcritical region of model parameters, we prove a conditional Central Limit Theorem (CLT) with explicit mean and variance for the number of two stars. This generalizes the corresponding result in the literature for the Erd\H{o}s--R\'enyi random graph. To prove our main result, we develop a new conditional CLT via exchangeable pairs based on the ideas of Dey and Terlov. Our key technical contributions in the application to ERGMs include establishing a linearity condition for an exchangeable pair involving two star counts, a local CLT for edge counts, as well as new higher-order concentration inequalities. Our approach also works for general subgraph counts, and we give a conjectured form of their conditional CLT.

\medskip

%60F05, 60F10, 62E17


%Central limit theorem, Cram\'er-type moderate deviations, multivariate normal approximation, $p$-Wasserstein distance, Stein's method
%\tcb{re-ordered alphabetically}
%\begin{keyword}[class=AMS]
%\kwd[Primary ]{60F05,62E20} \kwd[; secondary ]{62L10}
%\end{keyword}

%{
%  \hypersetup{linkcolor=black}
%  \tableofcontents
%}

%We have
%\be{
%\sum_{i,j,k,l}\E \widetilde X_{ij} \t X_{jk} \t X_{kl}=\sum_{j,k} \E \t X_{jk} \t d_{j} \t d_{k},
%}
%where $\t d_j$ denotes the centered degree.
%Except for a probability $\sim e^{-\sqrt{n}}$, $\t d_j, \t d_{\rm{K}}= O(\sqrt{n})$. Given such $d_j$ and $d_{\rm{K}}$, we should have $\E [\t X_{jk} \mid d_j, d_{\rm{K}}]=o(1)$? and hence $\sum_{i,j,k,l}\E \widetilde X_{ij} \t X_{jk} \t X_{kl}=o(n^3)$.

\section{Introduction}

%\paragraph{MODEL DESCRIPTION:} 
Exponential random graph models (ERGMs) are frequently used as parametric statistical models in network analysis, especially in the sociology community.
They were suggested for directed networks by \cite{holland1981exponential} and for undirected networks by \cite{frank1986markov}.
A general development of the models is presented in \cite{wasserman1994social}. 
In this paper, we consider undirected dense ERGMs, where the resulting graphs generally contain $\Theta_p(n^2)$ edges, with
$n$ representing the number of vertices.
We do not consider sparse ERGMs (models that have $\Theta_p(n)$ edges), which are harder to study than their dense counterparts. For important discoveries about sparse ERGMs, see, for example, \cite{chatterjeeDembo2016} and \cite{Nicholas2024}.
%(see also \cite{chakraborty2024,Nicholas2024} for more details). more reference?
Alternative model specifications are also studied in the literature. For example, \cite{snijders2006new} proposed curved exponential random graph models which include statistics such as geometrically weighted degree counts. These more complicated models do not fall under our current investigation.

%\paragraph{MOTIVATION:} 
We study the limiting distribution of subgraph counts in the ERGM conditioning on the number of edges. This problem is motivated by the similar problem in the literature on the Erd\H{o}s--R\'enyi random graph; see, for example, \cite{janson1990,bresler2018optimal,dey2023}. The current problem is much more involved because of the non-trivial dependency between the edges in ERGM. In fact, even the central limit theorem (CLT) for the total number of edges was only established recently in \cite{fang2024}. The result established in the paper may be used to statistically test an ERGM given an observed number of edges (see \cref{rem:test}). 
%\cite{Bauerschmidt2024}.
Our study is a meaningful special case of conditional limit theorems of degenerate generalized $U$-statistics of weakly dependent random variables.
Our study is also remotely related to the study of which observable is causing the phase transition in a statistical physics model (see \cref{rem:phase}).

We refer to
\cite{bhamidi2011mixing} and \cite{chatterjee2013estimating} for the following formulation of the model.
Let $\mathcal{G}_n$ be the space of all simple graphs\footnote{In this paper, all the graphs considered are simple graphs  which are undirected and without self-loops or multiple edges.} on $n$ labeled vertices. 
Let $m\geq 1$ be a positive integer. Let $\beta=(\beta_1,\dots, \beta_m)$ be a vector of real parameters, and let $H_1,\dots, H_m$ be (typically small) simple graphs without isolated vertices. 
%For any graph $G\in \mcl{G}_n$, 
We use $\mathcal{V}(G)$ and $\mathcal{E}(G)$ to denote the vertex set and edge set of a graph $G$, respectively. For each graph $H_i$, let $\operatorname{Hom}(H_i,G)$ denote the number of homomorphisms of $H_i$ into $G$.
A homomorphism is defined as an injective mapping from the vertex set $\mathcal{V}(H_i)$ of $H_i$ to the vertex set $\mathcal{V}(G)$ of $G$, such that each edge in $H_i$ is mapped to an edge in $G$.
%, that is,  injective maps $V(H_i)\to V(G)$, where $V(H_i)$ and $V(G)$ are the vertex sets of $H_i$ and $G$ respectively, such that each edge in $H_i$ gets mapped to an edge in $G$. 
For instance, if $H_i$ is an edge, then $\text{Hom}(H_i,G)$ is equal to twice the number of edges in $G$. Similarly, if $H_i$ is a triangle, then $\text{Hom}(H_i,G)$ is equal to six times the number of triangles in $G$. Given $\beta=(\beta_1,\dots, \beta_m)$ and $H_1,\dots, H_m$, the ERGM assigns probability
\begin{equation}\label{eq:ERGM}
p_{\beta}(G)=\frac{1}{Z(\beta)}\exp\left\{ n^2 \sum_{i=1}^m \beta_i t(H_i, G)\right\}
\end{equation}
to each $G\in \mathcal{G}_n$, where 
\be{
t(H_i, G):=\frac{\text{Hom}(H_i,G)}{n^{|\mathcal{V}(H_i)|}}
}
denotes the homomorphism density, 
$|\cdot|$ denotes cardinality when applied to a set, and $Z(\beta)$ is a normalizing constant. The $n^2$ and $n^{|\mathcal{V}(H_i)|}$ factors in \cref{eq:ERGM} ensure a nontrivial large $n$ limit. In this paper, we always take $H_1$ to be an edge by convention ($H_2,\dots, H_m$ are graphs with at least two edges) and assume $\beta_2,\dots, \beta_m$ are positive ($\beta_1$ can be negative).
Note that if $m=1$, then \cref{eq:ERGM} is the Erd\H{o}s--R\'enyi model $G(n,p)$ where every edge is present with probability $p=p(\beta)=e^{2\beta_1}(1+e^{2\beta_1})^{-1}$, independent of each other. If $m\geq 2$, \cref{eq:ERGM} ``encourages" the presence of the corresponding subgraphs.

Let $v_i$ be the number of vertices in $H_i$ and $e_i$ be the number of edges in $H_i$, and
\begin{align}\label{varph1}
    \Phi_{\beta}(a):=\sum_{i=1}^m\beta_ie_ia^{e_i-1},\quad \varphi_{\beta}(a):=\frac{\exp(2\Phi_{\beta}(a))}{1+\exp(2\Phi_{\beta}(a))}.
\end{align}
The so-called subcritical region contains all the parameters $\beta=(\beta_1,\dots, \beta_m)$ such that 
\begin{align}\label{Sub}
\text{there is a unique solution $p := p_{\beta}$ to the equation} \     \varphi_{\beta}(a)=a
\end{align} in $(0, 1)$ and $\varphi_{\beta}^{\prime}(p)<1.$ 
We always use $p$ to denote the unique solution in the rest of the paper.
Furthermore, Dobrushin’s uniqueness region, which is a subset of the subcritical region, contains all the parameters such that 
\ben{\label{eq:Dob}
\Phi_{\beta}^{\prime}(1)<2.
} 

%\paragraph{MAIN RESULT:} 
For the ERGM \cref{eq:ERGM} in the subcritical region \cref{Sub}, 
a CLT for the total number of edges was recently proved by \cite{fang2024}. 
\cite{Sam2020} proved that a similar decomposition as the Hoeffding decomposition for multilinear forms of independent random variables holds for the ERGM in Dobrushin's uniqueness region.
Combining the previous two results, we infer that CLT holds for any subgraph counts under \cref{eq:Dob};
%because the number of edges is the leading terms in the Hoeffding decomposition of any subgraph counts; 
see \cite[Corollary~3.3]{Sam2020} for the case of triangle counts. 
%under \cref{eq:Dob} and \cite[Corollay~3.1]{fang2024} for the general case under \cref{Sub}. 

The situation is less clear if we consider subgraph counts in ERGMs \emph{conditioning} on the number of edges. As a natural first step, we consider two star (the graph with three vertices and two edges connecting them) counts.
We obtain the following result. Its proof is contained in \cref{sec:main}.
We use $d_{\mathcal{W}}(W,Z)$ to denote the Wassertein distance between the distributions of two random variables $W$ and $Z$, that is,
\be{
d_{\mathcal{W}}(W, Z):=\sup _{h:\left\|h^{\prime}\right\|_\infty \leqslant 1}|\IE h(W)-\IE h(Z)|,
}
where the sup is taken over all absolutely continuous functions $h: \IR\to \IR$ and $\|\cdot\|_\infty$ denotes the $L^\infty$ norm.

Let $B$ be a compact subset of the subcritical region \cref{Sub}. 
%such that $2-\Phi'_\beta(1)\geq \delta>0$. 
Let $n\geq 1$ be an integer. 
Let $\beta_{2n},\dots, \beta_{mn}$ be positive parameters and let $\widetilde p_n\in (0,1)$.
Let $\E_{(\beta_{1n},\beta_{2n},\dots, \beta_{mn})}Y_{ij}=\t p_n$, where $Y_{ij}$ is the edge indicator of a random graph with vertex set $\{1,\dots, n\}$ and the expectation is with respect to the ERGM \cref{eq:ERGM} with parameters $(\beta_{1n},\beta_{2n},\dots, \beta_{mn})$.

%Let $\beta_2,\dots, \beta_k$ be fixed positive parameters. Let $\{(\beta_1, \beta_2,\dots, \beta_k): \beta_1\in B_1\}$ be a closed subset of Dobrushin's uniqueness region. Let $\widetilde p_n\in (0,1)$, $n\geq 1$, be a sequence of numbers. Assume that for all sufficiently large $n$, there exists $\beta_{1n}\in B_1$ with $\E_{(\beta_{1n},\beta_2,\dots, \beta_k)}Y_{ij}=\t p_n$, where $Y_{ij}$ is the edge indicator of a random graph with vertex set $\{1,\dots, n\}$ following the ERGM \cref{eq:ERGM} with parameters $(\beta_{1n},\beta_2,\dots, \beta_k)$.

Let $G_n$ be a random graph on $n$ vertices following the ERGM \cref{eq:ERGM} with any value of $\beta_1$ that may be unknown and $\beta_2=\beta_{2n},\dots, \beta_m=\beta_{mn}$. Let $E_n$ and $V_n$ be the number of edges and two stars in $G_n$, namely, 
\be{
E_n=\sum_{1\leq i<j\leq n} Y_{ij}
}
and
\be{
V_n=\sum_{1\leq j<k\leq n}\sum_{i\ne j,k} Y_{ij}Y_{ik}.
}
%with vertex set $\{1,\cdots,n\}$ and edge set $(i,j)_{1\leq i\neq j\leq n}$. Assume that the parameters of the ERGM are in a closed subset of Dobrushin's uniqueness region. 
%and $p_n$ is the unique solution to \cref{Sub}. 
%Denote  be the
%Let $Y_{i,j}$ be the indicator of edge $(i,j)$, $E_n$ and $V_n$ be the number of edges and two stars in $G_n$, i.e. $E_n=\sum_{i<j}Y_{i,j}$ and $V_n=\sum_{i<j,k\neq i,j}Y_{i,k}Y_{j,k}$. 

\begin{theorem}\label{thm:twostar}
Let $N=n(n-1)/2$. 
Then, conditioning on $E_n/N=\widetilde p_n$, we have,
\ben{\label{eq:thm1}
d_{\mathcal{W}}\left(\frac{V_n -\mu_{ V_n}}{\sigma_{ V_n}},Z\right)\leq C_\varepsilon n^{-\frac{1}{2}+\varepsilon}
}
for any $\varepsilon>0$,
where $C_\varepsilon$ is a constant depending only on $\varepsilon$ and $B$, $Z\sim N(0,1)$ is a standard normal random variable,
\begin{align}
    \mu_{V_n}&=N(n-2)\widetilde p_n^2+\frac{2N\t q_n^2\sum_{l=2}^m\beta_{ln}s_l\t p_n^{e_l}}{1-2\t q_n\sum_{l=2}^m\beta_{ln} s_l \t p_n^{e_l-1}},\quad \t q_n=1-\t p_n, \label{eq:muv} \\ \sigma_{ V_n}^2&=\frac{Nn\t p_n^2\t q_n^2}{\left(1-2\t q_n\sum_{l=2}^m\beta_{ln}s_l\t  p_n^{e_l-1}\right)^2}, \label{eq:sigv}
\end{align}
and $e_l$ and $s_l$ are the numbers of edges and two stars, respectively, contained in $H_l$, $2\leq l\leq m$, in the definition of the ERGM \cref{eq:ERGM}.
\end{theorem}
\begin{remark}
Note that the centering and normalization of $V_n$ in \cref{thm:twostar} are functions of the given value $\t p_n$ of $E_n/N$ and can be explicitly computed. This is useful if we do not know the asymptotic value of $\t p_n$. The explicit expressions in \cref{eq:muv,eq:sigv} come as by-products of the proof (see \cref{IEV} and \cref{m2}). The error bound \cref{eq:thm1} generalizes the corresponding result for the Erd\H{o}s--R\'enyi random graph $G(n,p)$, except for the arbitrarily small $\varepsilon$ error rate loss. Our assumption that the parameters of the ERGM lie in $B$ reduces to the corresponding assumption that $p$ is bounded away from 0 and 1 in $G(n,p)$.
See \cite[Lemma~4.4]{dey2023} and \cite[Theorem~12 with $d=2$]{bresler2018optimal}.

%Therefore, the condition of \cref{thm:twostar} can also be rephrased as: for the parameters $\beta_2, \dots, \beta_k$ in the ERGM \cref{eq:ERGM} and the asymptotic value $0<p<1$ of $E_n/N$, there exists some real number $\beta_1'$ such that $\beta_1', \beta_2, \dots, \beta_k$ are in Dobrushin's uniqueness region and $p$ is the solution to \cref{Sub} with $\beta=(\beta_1', \beta_2,\dots, \beta_k)$. 

%The condition $|\t p_n-p|=O(1/n)$ is imposed because of two reasons in our proof: 1. This is the typical range of fluctuation of $E_n$ and a local CLT implies that $r_k$ in \cref{P1} is bounded away from zero and infinity; 2. This condition ensures $|p_n-p|=O(1/n)$ ($p_n$ is defined in \cref{sec:main}) to be able to control various error terms in the proof of \cref{thm:twostar}. The optimality of this condition deserves further investigation.
\end{remark}

\begin{remark}\label{rem:test}
\cite{bresler2018optimal} showed that given an observed number of edges, two star counts can be used to test the Erd\H{o}s--R\'enyi random graph against an ERGM. \cref{thm:twostar} provides the asymptotical power of the test in subcritical region. Moreover, \cref{thm:twostar} can also be used to distinguish two ERGMs.
\end{remark}

\begin{remark}
For the special case of two star ERGM, where $m=2, \beta_2>0$ and $H_2$ is a two star in \cref{eq:ERGM}, \cite[Theorem~1.6]{mukherjee2023statistics} proved a CLT for edge-corrected two star counts.
Their result is valid for any parameters, including the critical case. However, their method only applies to the two star ERGM and, moreover, one can not directly infer a conditional CLT from their result.
\end{remark}

%Let $W$ be the number of edges in exponential random graphs and $Z\sim N(\mu,\sigma^2)$. Let $\widetilde{p}=N^{-1}\IEW$. Such defined $\widetilde{p}$ is close to $p$, which is the solution of \eqref{Sub}. Existing literature shows $|p-\widetilde{p}|=O(n^{-1/2})$. However, this accuracy is still too weak for the conditional CLT in this article. We modify the details of their proofs to obtain $|p-\widetilde{p}|=O(n^{-1})$ in the subcritical region.

%In this paper, we aim to obtain a conditional central limit theorem for two stars in exponential random graphs, based on \cite{dey2023}. To this end, we need to control the order of various quantities in that paper. In the process, we verify the local central limit theorem of W, give more precise error estimates of $p$ and $\widetilde{p}$, derive some linear conditions for two stars, and as a by-product, we give the expectation and variance of two stars.

To prove \cref{thm:twostar}, we use the approach of \cite{dey2023} who developed Stein's method of exchangeable pairs for proving conditional CLTs. An alternative approach to prove \cref{thm:twostar} would be to study the conditional graph directly. However, one difficulty of this alternative approach is that we need conditional higher-order concentration inequalities, which has not been developed yet.

In \cref{P1}, we provide a streamlined and modified version of the main result of \cite{dey2023} in dimension one. The main ingredient in applying the result to ERGMs is obtaining an approximate linearity condition (as a by-product, we obtain the asymptotic mean and variance of conditional two star counts). This involves higher-order concentration inequalities (the connection between higher-order concentration inequalities and the CLT was realized by \cite{fang2024}) and a bound on $\widetilde p - p$ (see \cref{L2}). 
Another ingredient is a local central limit theorem (LLT). The latter two are of independent interest and we state them as \cref{L2,lem:lclt}.

%In this paper, unless otherwise specified, we use $C$ to denote positive constants depending only on the closed subset $B$ of Dobrushin's uniqueness region given in \cref{thm:twostar}.The value of $C$ may differ from line to line. 

%Moreover, the Big $O$ notation may involve constants depending on the parameters of the ERGM.

\begin{proposition}\label{L2} 
Recall $E_n$ is the number of edges in the ERGM \cref{eq:ERGM} and $N=n(n-1)/2$.
    Let $\widetilde p:=\E(E_n)/N$ be the average edge density. In any compact subset $B$ of the subcritical region \cref{Sub},
    %such that $1-\varphi_\beta'(p)\geq \delta>0$, 
    we have
    \begin{equation}\label{eqmu}
        |p-\t p|\leq C n^{-1},
        \end{equation}
where $C$ is a constant depending only on $B$.
\end{proposition}

\begin{remark}
We can infer the bound 
\ben{\label{eq:slowrate}
|p-\widetilde{p}|\leq C/\sqrt{n}
}
from the main result of \cite{reinert2019} in the subcritical region. However, this accuracy is too weak for proving the conditional CLT in this paper. We apply Stein's method in a different way from \cite{reinert2019} to obtain \cref{L2}. See \cref{sec:lem1}. 
\cref{L2} addresses \cite[Conjecture~2.1]{winstein2025} in the subcritical region.
We also remark that as another application of \cref{L2}, we can remove the triangle-free condition in \cite[Theorem~3.1]{ding2024second}.
\end{remark}
\begin{remark}\label{rem:cstar}
In fact, we can obtain the exact constant of $\t p-p $ at the $n^{-1}$ order in the subcritical region, which gives the asymptotic mean in the CLT for the number of edges proved in \cite{fang2024}. Surprisingly, it requires the asymptotic expectations of the numbers of two stars and triangles we derive in this paper. In \cref{append}, we compute $c^*$ explicitly
%which only depends on $\beta$ and $\{H_1,\ldots,H_l\}$ 
such that %denoting the edge set of $H_l$ by $s_{l_1},\cdots,s_{l e_l}$, $1\leq l\leq m$, in Dobrushin’s uniqueness region, we have
     \begin{align*}
        |\widetilde{p}-p-c^* n^{-1}|\leq C n^{-3/2}.
    \end{align*}
%where $p$ is the solution to \cref{Sub}. 
 %   $v(H_l\backslash s_{lj})$ and $t(H_l\backslash s_{lj})$ are the numbers of two stars and triangles in graph $H_l\backslash s_{lj}$, respectively.
\end{remark}

To state the next proposition, 
recall $E_n$ is the number of edges in the ERGM \cref{eq:ERGM}. Let 
\ben{\label{eq:mun}
\mu_n=\E (E_n),\quad \sigma_n^2=\frac{Np(1-p)}{1-\sum_{i=2}^m \beta_i e_i(e_i-1) 2p^{e_i-1}(1-p)}.
}
Let $Z_{\mu_n, \sigma_n^2}\sim N(\mu_n, \sigma_n^2)$ be a normal random variable with mean $\mu_n$ and variance $\sigma_n^2$. Let $Z_{\mu_n, \sigma_n^2}^{(d)}$ be a discretized normal random variable supported on integers and, for any integer $k$,
\begin{align}\label{Zd}
    \P(Z_{\mu_n, \sigma_n^2}^{(d)}=k)=\P(k-\frac{1}{2}<Z_{\mu_n, \sigma_n^2} \leq k+\frac{1}{2}).
\end{align}
Let $d_{\rm{K}}$ and $d_{\rm{loc}}$ denote the Kolmogorov (maximum absolute difference between distribution functions) and local (maximum absolute difference between probability mass functions) metric, respectively. 
\cite{fang2024} proved that in the subcritical region, we have $\sigma_n^2\asymp n^2$ and
\ben{\label{Kolo}
d_{\rm{K}} (E_n, Z_{\mu_n, \sigma_n^2}^{(d)})\leq C n^{-1/4},
}
where $C$ is a constant depending only on the parameters of the ERGM.
We strengthen the result to be

\begin{proposition}\label{lem:lclt}
In any compact subset $B$ of the subcritical region \cref{Sub},
%such that $1-\varphi_\beta'(p)\geq \delta>0$, 
we have
\ben{\label{lclt}
d_{\rm{loc}}(E_n, Z_{\mu_n, \sigma_n^2}^{(d)})\leq C n^{-9/8},
}
where $C$ is a constant depending only on $B$.
\end{proposition}

Our approach works for conditional CLT for general subgraph counts. However, the case of two star is already very technical. Therefore, for the general case, we only provide a conjecture based on our preliminary computations sketched in \cref{sec:multi}. 

In this paper, we use $K$ to denote  simple graphs without isolated vertices and for any $K$, we use $\t K$ to denote the number of (not necessarily induced) copies of $K$ in $G_n$ except that each edge
indicator $Y_s$, is replaced by $\t Y_s:=Y_s-\t p_n$.
For example, if $K$ is an edge, then $\t K=\sum_{1\leq i<j\leq n}\t Y_{ij}$; if $K$ is a two star, then $\t K=\sum_{1\leq j<k\leq n}\sum_{i\ne j,k} \t Y_{ij}\t Y_{ik}$.
We also denote by $v(K)$ and $e(K)$ the number of vertices and edges of $K$, respectively.
\begin{conj}\label{conj1}
Let $H$ be a fixed %,(not need?) connected
 graph containing $t$ triangles and $s$ two stars. Under the same conditions as in \cref{thm:twostar}, with $F_n$ denoting the number of copies of $H$ in $G_n$, we have
\be{
\frac{F_n- \mu_F}{\sigma_F}\to N(0,1)\quad \text{in distribution},
}
where 
\begin{align}\label{eq:muf}
    \mu_F&=\frac{(n)_{(v(H))} }{\operatorname{Aut}(H)}\t p^{e(H)}+ \left(\frac{(n-3)_{(v(H)-3)}}{\operatorname{Aut}(H)}\right)\big(2s\t p^{e(H)-2}\mu_{\t V}+6t\t p^{e(H)-3}\mu_{\t \triangle}\big),\\\label{eq:sigmaf}
\sigma_F^2&=\left(\frac{(n-3)_{(v(H)-3)}}{\operatorname{Aut}(H)}\right)^2\left(4s^2\t p^{2e(H)-4}\sigma^{2}_{V} +36t^2\t p^{2e(H)-6}\sigma^{2}_{\t \triangle }\right),
\end{align}
 $(n)_{(k)}=n(n-1)\cdots (n-k+1),$ $\operatorname{Aut}(K)$ denotes the number of automorphisms of $K$
to itself and and $\t \triangle :=\sum_{i<j<k}\t Y_{ik}\t Y_{jk}\t Y_{ij}$. The values of $\mu_{\t V}, \sigma_{V}, \mu_{\t\triangle}, \sigma_{\t \triangle}$ are in \cref{muv,sigmav,IET',sigmat}, respectively.
\end{conj}

%We believe our approach also works beyond the subcritical region. The missing part towards a complete proof involves higher-order concentration inequalities, which have not been fully developed yet beyond the subcritical region.
%In fact, if \cref{LH,lemmaB3,LI} (even with some slower rates) can be extended to the subcritical region, then the conditional CLT can be extended to the subcritical region. We leave a systematic study of higher-order concentration inequalities in the subcritical region for future work.
%Nevertheless, we make the following conjecture, which is consistent with \cite[Theorem~1.6]{mukherjee2023statistics} for edge-corrected two star counts in the two star ERGM.
%\begin{conj}\label{conj2}
%\cref{thm:twostar} and \cref{conj1} also hold in  the supercritical region, and, in this case, $\t p_n$ is assumed to be sufficiently close to any solution $p^*$ to \cref{Sub} such that $\varphi'(p^*)<1$. 
%\end{conj}

%Again suggested by \cite[Theorem~1.6]{mukherjee2023statistics},
%in the critical case, $\mu_{V_n}$ in \cref{eq:muv} and $\mu_F$ in \cref{eq:muf} should be changed. However, we do not have a concrete conjectured form due to lack of understanding of the critical case.

\begin{remark}\label{rem:phase}
Note that $\sigma_n^2$ in \cref{eq:mun} explodes as the parameters approach criticality, indicating a phase transition of the number of edges in the ERGM. However, $\sigma_{ V_n}^2$ in \cref{eq:sigv} is always positive for any solution $p$ to \cref{Sub} with $1-\varphi_\beta'(p)\geq 0$. This seems to suggest that the phase transition of ERGMs when fixing the edge number, if exists, occurs beyond the uniqueness threshold.
See a somewhat related result for the Ising model in \cite{Bauerschmidt2024}. 
We do not know if the conditional CLT continues to hold beyond the subcritical region.
\end{remark}

We organize the paper as follows.  
In \cref{sec:2}, we prove \cref{thm:twostar}. In \cref{sec:lem}, we prove \cref{L2,lem:lclt}. In \cref{sec:prooflemma}, we prove some lemmas stated in \cref{sec:2}. In \cref{sec:multi}, we sketch the computations leading to \cref{conj1}. The computation of $c^*$ can be found in \cref{append}. In \cref{appendc}, we provide a higher-order concentration inequality using the Poincar\'e inequality.

%In Section 2 we give a proof of the local central limit theorem, in Section 3 we modify the theorem of Dey et al. \cite{dey2023}, in Section 4 we calculate the linearity condition in detail, and in Section 5 we present our conjecture about the central limit theorem for general graphs in ERGM. Finally, in the appendix we give two technical proofs.

\section{Proof of main result}\label{sec:2}

In this section, we first state and prove a new conditional CLT via Stein's method of exchangeable pairs. 
The result (\cref{P1}) and its proof resemble \cite[Section~2]{dey2023}, except that our result may be easier to apply as we do not make change of variables and do not impose the condition that $\Delta Y$ can only take three values ($\Delta Y$ is to be defined soon). 
Then, we use \cref{P1} to prove \cref{thm:twostar} in \cref{sec:main}.

\subsection{Modified conditional CLT of Dey and Terlov}
The general problem we consider is that we have two random variables $X$ and $Y$ and we wish to establish a CLT for the conditional distribution of $W=(X-\mu_k)/\sigma_k$ given $Y=k$, where $\mu_k, \sigma_k$ are suitable normalizing constants depending on $k$ such that $$\E(W|Y=k)\approx 0,\quad \Var(W|Y=k)\approx 1.$$

\cite{dey2023} successfully adapted Stein's method of exchangeable pairs (\cite{stein1986approximate}) to the above setting by assuming $W=W(G), Y=Y(G)$ are functions of an underlying random variable $G$ ($G$ takes values in a general space $\mathcal{G}$) and considering an exchangeable pair $(G, G')$ (meaning  the following joint laws are equal) such that
\ben{\label{eq:exchWY}
\mathcal{L}(G, G')=\mathcal{L}(G', G).
}
For any function $u: \mathcal{G}\to \mathbb{R}$ and random variable $U=u(G)$, we denote $U'=u(G')$ and $\Delta$ to be the following difference operator
\ben{\label{Delta def}
\Delta U=U'-U.
}

Under suitable conditions on the exchangeable pair, Dey and Terlov obtained a CLT with an error bound on the Wasserstein distance for $\mathcal{L}(W\mid Y=k)$. See the details in their Section~2.

In the following, we provide a streamlined and modified version of their result. Compared to their result, we do not make change of variables and we allow  $\Delta Y$ to take more than three values.
%We denote $\Delta W:=W'-W$, $\Delta Y:=Y'-Y$. In fact, $\Delta$ is a difference operator, which means that for any function $T=T(X,Y)$ of $(X,Y)$, 
%\begin{equation}
    %\label{Delta def}
    %\Delta(T)=T(X',Y')-T(X,Y).
%\end{equation} 
Suppose that $\Delta Y$ is integer-valued.
We denote (we use $\pm$ to denote two equations with the corresponding sign on the right-hand side respectively)
\begin{align}\label{eq:Mlpm}
    M_{l,\pm}({W},Y)=
       \IE\left(({W}^{\prime}-{W})^l\mathbbm{1}_{\Delta Y =\pm 1}\mid W, Y\right), &\quad l=0,1,2,
\end{align}
where $\mathbbm{1}_\cdot$ denotes the indicator function.
Suppose for some constants $Q>0$, $\lambda>0$ and ${a_+}+{a_-}=1$, we have
\begin{align}
   \label{ua1} M_{0,\pm}({W},Y)=&Q+R_{0,\pm};\\ 
    \label{ua2}M_{1,\pm}({W},Y)=&-\lambda({a_\pm}{W}+{R_{1,\pm}});\\\label{ua3}
    M_{2,\pm}({W},Y)=&\lambda(1+{R_{2,\pm}}).
\end{align}
 
\begin{proposition}\label{P1}Suppose $W$ and $Y$ are random variables satisfying the above assumptions \eqref{ua1}-\eqref{ua3}. For any $k$ such that $p_k:=\mathbb{P}(Y=k)>0$ and $p_{k-1}=\mathbb{P}(Y=k-1)>0$, we have, using $(W|Y=k)$ to denote a random variable having the conditional distribution of $W$ given $Y=k$,
\begin{align}
 \label{p211}
     d_{\mathcal{W}}((W\mid Y=k),Z)\leq& C(1+r_k)(\mathscr{A}_k+\mathscr{B}_k+\lambda^{-1} \mathscr{C}_k+Q^{-1}(\mathscr{D}_k+\mathscr{E}_k+\mathscr{F}_k)),
\end{align}
 where $d_{\mathcal{W}}$ denotes the Wasserstein distance, $C$ is a universal constant, $Z\sim N(0,1)$ is a standard normal random variable, $r_k=p_{k-1}/p_k$,
 \begin{align*}
     &\mathscr{A}_k=\IE(|R_{1,+}+R_{1,-}|\mid Y=k);\\ &\mathscr{B}_k=\IE(|R_{2,+}+R_{2,-}|\mid Y=k);\notag\\
      &\mathscr{C}_k=\IE(|\Delta W|^3\mid Y\in\{k-1,k\});\\
     &\mathscr{D}_k=\IE((1+|a_+|+|\IE R_{1,+}|)(|\Delta W|+|R_{0,+}|+|R_{0,-}|)\mid Y\in\{k-1,k\}),\\
     &\t R_{1,+}=R_{1,+}-\IE R_{1,+},\\
     &\mathscr{E}_k=\IE(|\t R_{1,+}\Delta W|+|\Delta (\t R_{1,+})|  \mid Y \in \{k-1,k\})+ \IE(|\t R_{1,+}|(|R_{0,+}|
        + |R_{0,-}|)\mid Y \in \{k-1,k\}),\\
      &\mathscr{F}_k=  \IE(| R_{2,+}\Delta W|+|\Delta ( R_{2,+})|  \mid Y \in \{k-1,k\})
        + \IE(| R_{2,+}|(|R_{0,+}|
        + |R_{0,-}|)\mid Y \in \{k-1,k\}),
 \end{align*}
 and $\Delta$ is the difference operator defined in \cref{Delta def}.
 \end{proposition}
Before the proof, we need a lemma, which is a modified version of Lemma 5.1 in \cite{dey2023}, to bound  $\IE\left(U( \mathbbm{1}_{Y=k}-\mathbbm{1}_{Y=k-1})\right)$.
\begin{lemma}\label{le:d}
Let $U=u(W,Y)$ be a function of $W$ and $Y$, and let $\Delta U=u(W', Y')-u(W, Y)$ be defined as in \cref{Delta def}.
   %Assume  random variables $(U^{\prime},Y^{\prime})$ are exchangeable with $(U,Y)$ and $U$ is $(X,Y)$-measurable.  
   Then, for any $k$ such
that $\mathbb{P}(Y=k)>0$ and $\mathbb{P}(Y=k-1)>0$, we have
\begin{align}
 & \left| \IE (U (\mathbbm{1}_{Y=k} - \mathbbm{1}_{Y=k-1}) )\right|\notag\\\label{p212}
\leq &\frac{1}{Q} \IE(|\Delta U| \mathbbm{1}_{Y \in \{k-1,k\}}) + \frac{1}{Q} \IE(|U|(|R_{0,+}| + |R_{0,-}|)\mathbbm{1}_{Y \in \{k-1,k\}}).
\end{align}
\end{lemma}
\begin{proof}
  Assumption \cref{ua1} implies that
\[
Q = \IE(\mathbbm{1}_{\Delta Y = -1} | W, Y) - R_{0,-}.
\]
Hence,
\begin{align*}
   \IE(U\mathbbm{1}_{Y=k}) = &\frac{1}{Q} \IE(U\mathbbm{1}_{Y=k}\mathbbm{1}_{\Delta Y = -1}) - \frac{1}{Q} \IE(U\mathbbm{1}_{Y=k}R_{0,-})\\
= &\frac{1}{Q} \IE(U\mathbbm{1}_{Y=k}\mathbbm{1}_{Y'=k-1}) - \frac{1}{Q} \IE(UR_{0,-}\mathbbm{1}_{Y=k}). 
\end{align*}
Similarly, using the fact that $Q = \IE(\mathbbm{1}_{\Delta Y = 1} | W, Y) - R_{0,+}$, we get that
\[
\IE(U\mathbbm{1}_{Y=k-1}) = \frac{1}{Q} \IE(U\mathbbm{1}_{Y'=k}\mathbbm{1}_{Y=k-1}) - \frac{1}{Q} \IE(UR_{0,+}\mathbbm{1}_{Y=k-1}).
\]
Moreover, exchangeability of $(U, Y)$ and $(U', Y')$ implies
\[
\IE(U\mathbbm{1}_{Y=k-1}\mathbbm{1}_{Y'=k}) = \IE(U' \mathbbm{1}_{Y'=k-1}\mathbbm{1}_{Y=k}).
\]
It follows that
\[
|\IE(U\mathbbm{1}_{Y=k}) - \IE(U\mathbbm{1}_{Y=k-1})| \leq \frac{1}{Q} \IE(|\Delta U| \cdot \mathbbm{1}_{Y \in \{k-1,k\}})
\]
\[
+ \frac{1}{Q} \IE(|U| (|R_{0,+}| + |R_{0,-}|)\mathbbm{1}_{Y \in \{k-1,k\}}),
\]
which completes the proof.
\end{proof}

\begin{proof}[Proof of Proposition \ref{P1}]
%Without loss of generality, we assume $\sigma_X=1$. 
By the exchangeability of $(W,Y)$ and $(W^{\prime},Y^{\prime})$, 
\begin{align*}
    \Theta_f(W,Y):= \IE[(F(W^{\prime})-F(W))\cdot \psi(Y^{\prime},Y)\mid W,Y]
\end{align*}
is a zero-mean variable for any symmetric function $\psi$ and $F^{\prime}=f$ such that the above expectation exists.
Let $f$ be the bounded solution to the Stein equation
\begin{align*}
    xf (x)-f' (x)=h(x)-\E h(Z)
\end{align*}
for a $1-$Lipschitz function $h$. From \cite[Eq.(2.13)]{chen2010normal}, we have
\begin{equation}\label{10001}
    \|f\|_{\infty}\leq 1,\quad
\|f'\|_{\infty}\leq \sqrt{2/\pi}\ \text{ and }\  \|f''\|_{\infty}\leq 2.
\end{equation}
Let $\psi(y^{\prime},y)=\mathbbm{1}_{y^{\prime}=k}\mathbbm{1}_{y^{\prime}-y=1}+\mathbbm{1}_{y=k}\mathbbm{1}_{y^{\prime}-y=-1}$. Using Taylor's expansion and the notation \cref{eq:Mlpm}, we have
\begin{align*}
    \Theta_f(W,Y)=&\left(f(W)M_{1,-}(W,Y)+\frac{1}{2}f^{\prime}(W)M_{2,-}(W,Y)\right)\mathbbm{1}_{Y=k}\\
    &+\left(f(W)M_{1,+}(W,Y)+\frac{1}{2}f^{\prime}(W)M_{2,+}(W,Y)\right)\mathbbm{1}_{Y=k-1}+ \operatorname{Err},
\end{align*}
where
\begin{align}\label{10007}
    |\operatorname{Err}|\leq \frac{1}{3}|\Delta W|^3 \mathbbm{1}_{Y\in \{k-1, k\}} .
\end{align}
Using assumptions \eqref{ua2}, \eqref{ua3} and $a_++a_-=1$, we have
\begin{align*}
\Theta_f(W,Y)=&-\lambda  Wf(W)\mathbbm{1}_{Y=k}-\lambda a_+ f(W)(\mathbbm{1}_{Y=k-1}-\mathbbm{1}_{Y=k})
\\
&-\lambda (R_{1,+}+R_{1,-})f(W)\mathbbm{1}_{Y=k}- \lambda \t R_{1,+} f(W)(\mathbbm{1}_{Y=k-1}-\mathbbm{1}_{Y=k})\\
&-\lambda f(W)(\mathbbm{1}_{Y=k-1}-\mathbbm{1}_{Y=k}) \IE R_{1,+}\\
&+\lambda  f^{\prime}(W)\mathbbm{1}_{Y=k}+\frac{1}{2}\lambda  
 f^{\prime}(W)(\mathbbm{1}_{Y=k-1}-\mathbbm{1}_{Y=k})\\
&+\frac{1}{2}\lambda(R_{2,+}+R_{2,-}) f^{\prime}(W)\mathbbm{1}_{Y=k}+\frac{1}{2}\lambda R_{2,+} f^{\prime}(W)(\mathbbm{1}_{Y=k-1}-\mathbbm{1}_{Y=k})\\
&+\operatorname{Err}.
\end{align*}
Hence, by $\IE\Theta_f(W,Y)=0$, 
\begin{equation}\label{10002}
\begin{aligned}
    \lambda \IE\big[(Wf(W)-f^{\prime}(W))1_{Y=k}\big]=I+II+III+IV,
\end{aligned}
\end{equation}
where
\begin{align*}
    I=&-\lambda  a_+ \IE\big[ f(W)( \mathbbm{1}_{Y=k-1}- \mathbbm{1}_{Y=k})\big]-\lambda \IE R_{1,+} \IE\big[ f(W)(\mathbbm{1}_{Y=k-1}-\mathbbm{1}_{Y=k})\big]\\&+\frac{1}{2}\lambda  \IE\big[
 f^{\prime}(W)(\mathbbm{1}_{Y=k-1}-\mathbbm{1}_{Y=k})\big],\\
 II=&-\lambda \IE\big[(R_{1,+}+R_{1,-})f(W)\mathbbm{1}_{Y=k}\big]- \lambda \IE\big[ \t R_{1,+} f(W)(\mathbbm{1}_{Y=k-1}-\mathbbm{1}_{Y=k})\big],\\
 III=&\frac{1}{2}\IE \big[\lambda(R_{2,+}+R_{2,-}) f^{\prime}(W)\mathbbm{1}_{Y=k}\big]+\frac{1}{2}\IE\big[\lambda R_{2,+} f^{\prime}(W)(\mathbbm{1}_{Y=k-1}-\mathbbm{1}_{Y=k})\big],\\
 IV=&\E [\operatorname{Err}].
\end{align*}
Noting that $$ \IE\big[(Wf(W)-f^{\prime}(W))\mathbbm{1}_{Y=k}\big]=\IE \big[(h(W)-\E h(Z))\mathbbm{1}_{Y=k}\big]=p_k\{\E[h(W)|Y=k]-\E h(Z) \}$$
and using \cref{10001}, we have
\begin{equation}\label{10009}
   d_{\mathcal{W}}((W\mid Y=k),Z)\leq \frac{1}{\lambda p_k}\sup_{f\ \text{in}\ \cref{10001}} \left| I+II+III+IV \right|. 
\end{equation}
For $I$, recalling that $r_k=p_{k-1}/p_k\in (0,\infty)$, by \cref{le:d}, we obtain
\begin{align}\label{10003}
     &\frac{1}{\lambda  p_k}\sup_{f\ \text{in}\ \cref{10001}}|I|\notag\\\leq& \frac{1+r_k}{Q}\Big(\big(\sqrt{2/\pi}(|a_+|+|\IE R_{1,+}|)+1\big) \IE(|\Delta W|\mid Y\in \{k-1,k\})\notag\\
        &\quad\quad\quad\quad\quad+(|a_+|  +|\IE R_{1,+}|+1/\sqrt{2\pi})\IE(|R_{0,+}| + |R_{0,-}|\mid Y\in \{k-1,k\})\Big).
\end{align}
Noting that $R_{1,\pm}$ and $R_{2,\pm}$ are $(W,Y)$ measurable, we have
$$|\Delta ( \t R_{1,+} f(W))|\leq \|f'\|_{\infty}|\t R_{1,+}||\Delta W|+\|f\|_{\infty}|\Delta (\t R_{1,+})|\leq \sqrt{2/\pi}|\t R_{1,+}\Delta W|+|\Delta (\t R_{1,+})|$$
and
$$|\Delta (  R_{2,+} f'(W))|\leq \|f''\|_{\infty}| R_{2,+}||\Delta W|+\|f'\|_{\infty}|\Delta ( R_{2,+})|\leq | R_{2,+}\Delta W|+\sqrt{2/\pi}|\Delta ( R_{2,+})|.$$

We now consider $II$, by \cref{le:d} again,
\begin{align}
       & \frac{1}{\lambda  p_k}\sup_{f\ \text{in}\ \cref{10001}} |II|\notag\\\leq& \frac{1+r_k}{ Q}\Big(\frac{Q}{(1+r_k)}\IE\big(|R_{1,+}+R_{1,-}|\mid Y=k)+\IE\Big(\sqrt{\frac{2}{\pi}}|\t R_{1,+}\Delta W|+|\Delta (\t R_{1,+})| \big \mid Y \in \{k-1,k\}\Big)\notag\\
        &+ \IE(|\t R_{1,+}|(|R_{0,+}|\label{10004}
        + |R_{0,-}|)\mid Y \in \{k-1,k\})\Big).
\end{align}
Similarly, for $III$,
\begin{align}
       & \frac{1}{\lambda  p_k}\sup_{f\ \text{in}\ \cref{10001}}|III|\notag\\\leq& \frac{1+r_k}{ Q\sqrt{2\pi}}\Big(\frac{Q}{1+r_k}\IE(|R_{2,+}+R_{2,-}|\mid Y=k)+\IE(\sqrt{2\pi}|R_{2,+}\Delta W|+|\Delta  (R_{2,+})| \big \mid Y \in \{k-1,k\})\notag\\
        &+ \IE\big(| R_{2,+}|(|R_{0,+}|
        + |R_{0,-}|)\mid Y \in \{k-1,k\}\big)\Big).\label{10005}
\end{align}
For $IV$, by \cref{10007},
\begin{equation}
    \label{10008}
    \frac{1}{\lambda p_k}|IV|\leq\frac{1+r_k}{3\lambda }\IE\left(|\Delta W|^3\mid Y\in\{k-1,k\}\right).
\end{equation}
Combining \cref{10009,10003,10004,10005,10008}, we complete the proof.
 % From the assumption ,   and the bound of the first three derivatives of solutions to Stein equations,  dividing by $p_k$, now applying Lemma \ref{le:d} yields
%     \begin{align*}
%         &d_{\mathcal{W}}((X\mid Y=k),Z)\notag\\
%         \leq &\frac{1+r_k}{\psi Q}\Big(\frac{Q}{(1+r_k)}\IE(|R_{1,+}+R_{1,-}|\mid Y=k)+\IE(|\t R_{1,+}\Delta X|+|\Delta (\t R_{1,+})|  \mid Y \in \{k-1,k\})\\
%         &+ \IE(|\t R_{1,+}|(|R_{0,+}|
%         + |R_{0,-}|)\mid Y \in \{k-1,k\})\Big)\notag\\
%         &+\frac{1+r_k}{\psi Q\sqrt{2\pi}}\Big(\frac{Q}{1+r_k}\IE(|R_{2,+}+R_{2,-}|\mid Y=k)+\IE(|R_{2,+}\Delta X|+|\Delta  R_{2,+})|  \mid Y \in \{k-1,k\})\\
%         &+ \IE(| R_{2,+}|(|R_{0,+}|
%         + |R_{0,-}|)\mid Y \in \{k-1,k\})\Big)\\
%         &+\\
%         &+\frac{1+r_k}{3\lambda \psi}\IE\left(|\Delta X|^3\mid Y\in\{k-1,k\}\right),
%     \end{align*}
% which is the
% desired result \eqref{p211}.

\end{proof}

\subsection{Proof of \cref{thm:twostar}}\label{sec:main}

We first argue that the conditional probability distribution of the ERGM \cref{eq:ERGM} given the total number of edges is independent of the parameter $\beta_1$. In fact, for any $G\in \mathcal{G}_n$ with $n$ vertices and $k$ edges, 
\begin{align*}
    \mathbb{P}_{\beta}(G\mid E_n=k)=&\frac{\exp\left\{ n^2 \sum_{i=1}^m \beta_i t(H_i, G)\right\}}{\sum_{\substack{F\in \mathcal{G}_n\\ e(F)=k}}\exp\left\{ n^2 \sum_{i=1}^m\beta_i t(H_i, F)\right\}}\\
    =&\frac{\exp\left\{2\beta_1 k+ n^2 \sum_{i=2}^m \beta_i t(H_i, G)\right\}}{\sum_{\substack{F\in \mathcal{G}_n\\ e(F)=k}}\exp\left\{2\beta_1 k+ n^2 \sum_{i=2}^m \beta_i t(H_i, F)\right\}}\\
    =&\frac{\exp\left\{ n^2 \sum_{i=2}^m \beta_i t(H_i, G)\right\}}{\sum_{\substack{K\in \mathcal{G}_n\\ e(F)=k}}\exp\left\{ n^2 \sum_{i=2}^m \beta_i t(H_i, F)\right\}}.
\end{align*}
In statistical terms, this is because $E_n$ is a sufficient statistic for the parameter $\beta_1$.
Therefore, we assume, for the purpose of approriate centering, that the parameters in the ERGM \cref{eq:ERGM} are $(\beta_{1n},\beta_{2n},\dots, \beta_{mn})\in B$ as in the first paragraph of \cref{thm:twostar}. We denote $\E(\cdot):=\E_{(\beta_{1n}, \beta_{2n},\dots, \beta_{mn})}(\cdot)$ to be the expectation with respect to the ERGM \cref{eq:ERGM}.

We label the vertex set of the random graph $G_n$ by $\{1,\dots, n\}$ and let $Y_{ij}$, $1\leq i<j\leq n$ be the edge indicators which equals 1 or 0 depending on whether the edge $(i,j)$ is present or not in $G_n$. Let $Y_{ji}=Y_{ij}$. By the assumption of \cref{thm:twostar}, we have
\be{
\E Y_{ij}=\t p_n.
}

In the following, for convenience of notation, we drop the subindex $n$. For example, we write $G:=G_n$, $\beta_1:=\beta_{1n},\dots, \beta_m:=\beta_{mn}$, $E:=E_n$, $V:=V_n$, $\widetilde p:=\widetilde p_n$. Let $p:=p_n$ be the unique solution to \eq{Sub}.

We remark that all the estimates in \cite{bhamidi2011mixing}, \cite{ganguly2019} and \cite{Sam2020} are uniform in compact subsets of the subcritical region.
%if $1-\varphi'_\beta(p)$ ($2-\Phi'_\beta(1)$ resp.) is bounded away from 0. 
%\red{Because $\tau_{mix}$ and the constant in the Poincar\'e inequality are uniform in any compact subset?}\blue{(right)}
For example, see \cite[Proof of Lemma~14]{bhamidi2011mixing}, \cite[Eq.(61)]{ganguly2019} and \cite[Theorem~4.1 and Proof of Proposition~4.2]{Sam2020}.
Therefore, all the constants $C$, including those implicitly involved in the Big $O$ notation $O(\cdot)$, and the $L^r$ norms appearing in the error bounds in this proof are uniformly bounded in the subset $B$ of the subcritical region considered in \cref{thm:twostar}.

We introduce the following notation:
\be{
\t Y_{ij}:=Y_{ij}-\t p,
}
\be{
\t E:=\sum_{1\leq i<j\leq n} \t Y_{ij},
}
\be{
\t V:=\sum_{1\leq j<k\leq n}\sum_{i\ne j,k} \t Y_{ij} \t Y_{ik}.
}
Note that given $E=N\t p$, we have
\be{
V-N(n-2)\widetilde p^2 =\t V. 
}
Therefore, \cref{eq:thm1} is equivalent to (note that we have dropped the subindex $n$)
\ben{\label{eq:thm11}
d_{\mathcal{W}}\left((\frac{\t V -\mu_{ \t V}}{\sigma_{  V}}\mid \t E=0),\ Z\right)\leq C_\varepsilon n^{-\frac{1}{2}+\varepsilon},
}
where
\begin{align}\label{muv}
    \mu_{\t V}&=\frac{2N\t q^2\sum_{l=2}^m\beta_{l}s_l\t p^{e_l}}{1-2\t q\sum_{l=2}^m\beta_{l} s_l \t p^{e_l-1}},\quad \t q=1-\t p,  \\\label{sigmav} \sigma_{ V}^2&=\frac{Nn\t p^2\t q^2}{\left(1-2\t q\sum_{l=2}^m\beta_{l}s_l\t  p^{e_l-1}\right)^2}.
\end{align}

To prove \cref{eq:thm11}, we will apply \cref{P1} with $k=0$ and
\begin{equation}\label{20001}
    W=\frac{\t V-\mu_{\t V}}{\sigma_V},\quad Y=\t E,
\end{equation}
both viewed as functions of the random graph $G$ or its edge indicators $\{Y_{ij}, 1\leq i<j\leq n\}$. We construct the exchangeable pair by randomly choosing an edge index $e$ uniformly from $\{(i,j): 1\leq i<j\leq n\}$ and resampling the edge indicator to be $Y'_e$ according to its conditional distribution under the ERGM \cref{eq:ERGM} given the value of all the other edge indicators. Let $G'$ be the resulting random graph which differs from $G$ by at most one edge. We verify the conditions of \cref{P1} by establishing the following proposition.

\begin{proposition}\label{P2}
    For $W$ and $Y=\t E$ in \cref{20001}, 
    %W=(\widetilde{V}-\mu_{\t V})/\sigma_V$, where 
    %\begin{align}\label{Ev}
    %   \IE\t V=\mu_{\t V}+O(n), 
    %\end{align} $\sigma_V$ was defined in \cref{eq:sigv}, and  $Y=\widetilde{E}$, 
    we have (recall \cref{eq:Mlpm})
    \begin{align}
       M_{0,\pm}(W,\t E)&=\t p\,\t q +R_{0,\pm},\label{eq:M0W}\\
     M_{1,+}(W,\t E)&=-\lambda \left( \alpha_+ W+R_{1,+}\right),\notag\\
      M_{1,-}(W,\t E)&=-\lambda \left( \alpha_- W+R_{1,-}\right),\notag\\
      M_{2,\pm}(W,\t E)&=\lambda(1+R_{2,\pm}),\notag
    \end{align}
where $\lambda=(2(1-2\sum_{l=2}^m\beta_ls_l\t p^{e_l-1}\t q))/N$, $\Delta W=O_r(n^{-1}),$
\begin{align*}
     a_+=&\frac{\t p-2\t q^2\sum_{l=2}^m\beta_l s_l\t p^{e_l-1}}{1-2\t q\sum_{l=2}^m\beta_ls_l\t p^{e_l-1}}, \quad a_{-}=\frac{\t q-2\t q\sum_{l=2}^m\beta_l s_l\t p^{e_l}}{1-2\t q\sum_{l=2}^m\beta_ls_l\t p^{e_l-1}},
     \\
     R_{0,\pm}=&O_{r}(n^{-1}),\quad R_{1,+}=O_{r}(n^{\frac{1}{2}}),\quad R_{1,-}=O_{r}(n^{\frac{1}{2}}),\quad
     R_{1,+}+R_{1,-}=O_{r}(n^{-\frac{1}{2}}),\\
    \t R_{1,+}=&R_{1,+}-\E R_{1,+}=O(n^{-\frac{1}{2}})\t E+O_{r}(n^{-\frac{1}{2}}),\\
    \t R_{1,-}=&R_{1,-}-\E R_{1,-}=O(n^{-\frac{1}{2}}) \t E+O_{r}(n^{-\frac{1}{2}}),\\
     R_{2,+}=&O_{r}(n^{-\frac{1}{2}}),\quad R_{2,-}=O_{r}(n^{-\frac{1}{2}}).
\end{align*}
Hereafter, for convenience, we use $O_{r}(h)$ in an unconventional way to denote some random variable $X$ such that the $L^r$ norm $\|X/h\|_r\leq C_r$ for any given $r\geq 1$ and a constant $C_r$ depending only on $r$ and the subset $B$ of the subcritical region considered in \cref{thm:twostar}.
\end{proposition}

We first use \cref{P1}, \cref{P2} and \cref{lem:lclt} to prove \cref{thm:twostar}.

\begin{proof}[Proof of \cref{thm:twostar}]
%\brown{We first prove a conditional CLT for $W$ defined in \cref{20001}.} 
It suffices to prove that the right-hand side of \cref{p211} is bounded by $C_\varepsilon n^{-1/2+\varepsilon}$.
Note that
 for any random variable $X$ and any $i$ such that $\P(\t E=i)>0$, we have
  \begin{align}\label{eq:cond1}
      \IE(X\mid \widetilde{E}=i)=\frac{\IE(X\mathbbm{1}_{\widetilde{E}=i})}{\mathbb{P}(\widetilde{E}=i)}\leq \frac{(\IE|X|^r)^{1/r} \mathbb{P}(\widetilde{E}=i)^{1-1/r}}{\mathbb{P}(\widetilde{E}=i)}=(\IE|X|^r)^{\frac{1}{r}}\mathbb{P}(\widetilde{E}=i)^{-\frac{1}{r}}.
  \end{align}

By \cref{lem:lclt}, we have 
\be{
\P(\widetilde{E}=0) = 1/\sqrt{2\pi \sigma_n^2}+O(n^{-\frac{9}{8}}),\quad \sigma_n^2\asymp N.
}
For any $\varepsilon>0$, we can choose a large enough $r$ such that
\ben{\label{eq:cond2}
\mathbb{P}(\widetilde{E}=i)^{-\frac{1}{r}}=O(n^{\varepsilon}).
}
Then, for $\mathscr{A}_0$ and $\mathscr{B}_0$, from
$R_{1,+}+R_{1,-}=O_{r}(n^{-1/2})$ and $R_{2,+}+R_{2,-}=O_{r}(n^{-1/2})$ proved in \cref{P2}, we obtain
\begin{align}\label{DeltaW}
    \IE(|R_{1,+}+R_{1,-}|\mid Y=0)\leq C_\varepsilon n^{-\frac{1}{2}+\varepsilon},\quad \IE(|R_{2,+}+R_{2,-}|\mid Y=0)\leq C_\varepsilon n^{-\frac{1}{2}+\varepsilon}.
\end{align}
From  \cref{LH},
\begin{align}\label{ha}
    \sum_{k\neq i,j}(\t Y_{ik}+\t Y_{jk})=O_{r}(n^{\frac{1}{2}}).
\end{align}
Together with the fact that $\sigma_V\asymp n^{3/2}$, we obtain
\begin{align*}
    \Delta W=N^{-1}\sigma_{V}^{-1}\sum_{i<j} \sum_{k\neq i,j}(\t Y_{ik}+\t Y_{jk})(\t Y_{ij}'-\t Y_{ij})=O_{r}(n^{-1}).
\end{align*}
Moreover, it implies $\lambda^{-1}\mathscr{C}_0\leq C_\varepsilon n^{-1+\varepsilon}.$
For $\mathscr{D}_0$,  noting that $a_+=O(1)$, $\IE R_{1,+}=O(n^{1/2}), |\Delta W|=O_{r}(n^{-1})$ and $R_{0,\pm}=O_{r}(n^{-1})$,  we obtain
\begin{align}\label{D0}
\mathscr{D}_0\leq C_\varepsilon n^{-\frac{1}{2}+\varepsilon}.
\end{align}
For $\mathscr{E}_0$ and $\mathscr{F}_0$, taking $\Delta \t R_{1,+}=O_{r}(n^{-1/2})$ and $\Delta  R_{2,+}=O_{r}(n^{-1/2})$ into consideration and by similar arguments as for \cref{D0}, we obtain
\begin{align*}
    \mathscr{E}_0\leq C_\varepsilon n^{-\frac{1}{2})+\varepsilon},\quad \mathscr{F}_0\leq C_\varepsilon n^{-\frac{1}{2}+\varepsilon}.
\end{align*}
By \cref{lem:lclt}, we have 
\begin{equation}\label{r0}
    r_0=1+O(n^{-\frac{1}{8}}).
\end{equation}
 Now by \cref{P1}, we conclude that 
\begin{align*}
    d_{\mathcal{W}}((W\mid \widetilde{E}=0),Z)\leq C_{\varepsilon}n^{-\frac{1}{2}+\varepsilon}.
\end{align*}
% With the above CCLT, we can now complete the proof of the theorem.
%And from \cref{muv,sigmav,Ev},
%\begin{align*}
%    W=\frac{\t V-\IE \t V}{\sigma_{\t V}}=\frac{\t V-\mu_{\t V}}{\sigma_{\t V}}+O(n^{-1/2}),
%\end{align*}
%which implies 
%\begin{align*}
%    d_{\mathcal{W}}\left(\left(\frac{\t V-\mu_{\t V}}{\sigma_{\t V}}\mid \widetilde{E}=0\right),Z\right)\leq C_{\varepsilon}n^{-1/2+\varepsilon}.
%\end{align*}
%Moreover, it is easy to obtain
%    \begin{align*}
%        \mathcal{L}(V\mid E=n\t p)=\mathcal{L}\left(\sum_{i<j,k\neq i,j}(\t Y_{i,k}+\t p)(\t Y_{j,k}+\t p)\mid \t E=0 \right)= \mathcal{L}(\t V+ N(n-2)\t p^2 \mid \t E=0 ),
%    \end{align*}
%    where $\mathcal{L}(\cdot|\cdot)$ denotes the conditional law. Hence 
%we complete the proof of \eqref{eq:thm1}. 
\end{proof}

In the remainder of this section, we prove \cref{P2}. Its proof requires \cref{L2} and some technical lemmas (\cref{LH,LI,lemmaB3}). Their proofs are given in \cref{sec:lem1} and \cref{sec:prooflemma}.

We identify the random graph $G$ following \cref{eq:ERGM} on $n$ vertices $\{1,\dots, n\}$ with its edge indicators 
$Y=\{Y_{e}, e\in\{ (i,j), 1\leq i<j\leq n\}\}$ where $Y_e = 1$ if the
edge $e$ is present in $G$ and $Y_e = 0 $.
The exchangeable pair of random graphs $(G, G')$ defined above \cref{P2} 
is related to the discrete-time Glauber dynamics associated with the ERGM as follows.
%In all technical proofs, we need to use Heat-bath Glauber dynamics (GD), so we first introduce GD on ERGM. In order to use the above lemma, we define a natural (discrete time) Heat-bath Glauber dynamics (GD) associated with the ERGM.  
%Let $K_n$ be a complete graph with vertex set $\{1,2,\dots, n\}$. 
%We identify a random graph $G\in \mathcal{G}_n$ with $Y = (Y_e)_{e\in \mathcal{E}(Kn)}$ with $Y_e = 1$ if the (undirected) edge $e$ is present in $G$ and $Y_e = 0 $, otherwise.  
Given the current state $Y$, $e$ is uniformly chosen from $\{ (i,j), 1\leq i<j\leq n\}$ and $Y_e$ is resampled to $Y_e'$ according to its conditional distribution under the ERGM \cref{eq:ERGM} given $(Y_f)_{f\neq e}$. The resulting random graph which differs at most by one edge from $G$ is denoted by $G'$. 
%Then $(G, G')$ is an exchangeable pair.

From \cref{eq:ERGM}, it can be verified that given $e$ is chosen to be updated, one has the following transition
probabilities:
\begin{align}\label{tran}
    \P(y,y_{e+})=\frac{\exp(\partial_e H(y))}{1+\exp(\partial_e H(y))},\quad  \P(y,y_{e-})=\frac{1}{1+\exp(\partial_e H(y))},
\end{align} 
where $y_{e+}$ ($y_{e-}$ resp.) is the same as $y$ except that $y_e=1$ ($y_e=0$ resp.),
%$H(x)=\sum_{l=1}^m \beta_in^{2-v_l}\operatorname{Hom}(H_l, x),$
  \begin{align}\label{r}
      \partial_e H(y)=&\sum_{l=1}^m \beta_ln^{2-v_l}\operatorname{Hom}(H_l, y, e)
    \end{align}
and $\operatorname{Hom}(H_l, Y, e)=\operatorname{Hom}(H_l, G, e)$ (recall we identify $G$ with $Y$) denotes the number of homomorphisms of $H_l$ into $G$ but requiring
that an edge of $H_l$ must be mapped to $e$ (no matter the edge $e$ is present in $G$ or not).
% Let $Y_{e_1}, \ldots, Y_{e_k}$ be the set of $k\geq 2$ distinct edge indicators. 
% Then we have that 
% \begin{equation}
%     \E \prod_{i=1}^k Y_{e_i} - \widetilde{p}^k= \sum_{r=1}^k \sum_{\{I_{1},\ldots i_r\}\subset\{1, \ldots, k\} } \widetilde{p}^{k-r} \E \prod_{s=2}^r \widetilde{Y}_{e_{i_s}}
% \end{equation}
% From \cite{ganguly2019} we know that

For $e=(i,j)$, we do an approximate centering and let (recall \cref{varph1})
\ben{\label{eq:pijbH}
\partial_{ij} \bar{H}= \partial_{ij} \bar{H}(Y)=\partial_e H(Y)-2\Phi_\beta(p).
}

We first argue that $\E [\partial_{ij} \bar{H}]\approx 0$.
From \cite[Eqs.(34)\&(84)]{ganguly2019}, for any fixed $ k>1$ and distinct edges $l_1, \cdots, l_k$, we have, in the subcritical region,
\begin{align}\label{GN}
    \left|\IE(Y_{l_{1}}|Y_{l_{2}},\ldots,Y_{l_{k}}) - \IE Y_{l_{1}} \right| \leq Cn^{-1}
\end{align}
and
\begin{align}\label{eq:GN84}
    \left|\IE(Y_{l_{1}}\cdots Y_{l_{k}}) - (\IE Y_{l_{1}})\cdots (\E Y_{l_k}) \right| \leq Cn^{-1}.
\end{align}
For any $H_l$ with $e_l$ edges, counting the number of homomorphisms and using \cref{eq:GN84}, we obtain
%\be{
%\operatorname{Hom}(H_l, G, (i,j))=2\sum_{\substack{e_1,\cdots,e_{e_l-1}:\\e_1\cup\cdots\cup e_{e_l-1}\cup (i,j)\cong H_l}}Y_{e_1}\cdots Y_{e_{e_l-1}}.
%}
%where $\{Y_e, e\in \mathcal{E}(K_n)\}$ is the edge indicator of $G$. 
%From (84) in \cite{ganguly2019}, we get for any $1\leq l\leq m$,
    \begin{align}\label{meandiff}
        |\IE\operatorname{Hom}(H_l, G, (i,j)))-2n^{v_l-2}e_l\widetilde{p}^{e^l-1}|
        %=&2\sum_{\substack{e_1,\cdots,e_{e_l-1}\\e_1\cup\cdot\cup e_{e_l-1}\cup e_{ij}\cong H_l}}\left(\IE Y_{e_1}\cdots Y_{e_{e_l-1}}-\widetilde{p}^{e_l-1}\right)\notag\\
        \leq Cn^{v_l-2}n^{-1},
    \end{align}
    and therefore (recall \cref{varph1})
\begin{align}\label{meandiffbar}
    {\IE(\partial_{ij} \bar{H})}
    =&\sum_{l=1}^m \beta_ln^{2-v_l}\bigg[(\IE\operatorname{Hom}(H_l, G, (i,j))-2n^{v_l-2}e_l\widetilde{p}^{e^l-1})\notag\\&\qquad \qquad +(2n^{v_l-2}e_l\widetilde{p}^{e^l-1})-(2n^{v_l-2}e_lp^{e^l-1})\bigg]\notag\\
    =&O(n^{-1}),
\end{align}
where we used \cref{meandiff,eqmu} in the last equality.
We conclude that
\begin{align*}
    \E[\partial_{ij} \bar{H}]=O(n^{-1}).
\end{align*}

Applying the Hoeffding decomposition under weak dependence (\cite{gotze2019higher,Sam2020}), we write (see details in \cref{sec:prooflemma})
\begin{align}\label{eq:partialABC}
    \partial_{ij} \bar{H}=\partial_{ij}\bar{H}_A+\partial_{ij}\bar{H}_B+\partial_{ij}\bar{H}_C+O(n^{-1}),
\end{align}
where 
\begin{align}\label{Ha}
\partial_{ij}\bar{H}_A&=\sum_{s\neq i,j}\frac{2\sum_{l=2}^m \beta_ls_l\widetilde{p}^{e_l-2}(\widetilde{Y}_{is}+\widetilde{Y}_{js})+6\sum_{l=2}^m\beta_lt_l\widetilde{p}^{e_l-3}(\widetilde{Y}_{is}\widetilde{Y}_{sj}-\textit{mean})}{n},\\\label{Hb}
    \partial_{ij}\bar{H}_B&=\frac{\sum_{\mathcal{K}\subset \mathcal{I}:\substack{v(\mathcal{K}\cup (i,j))= 4\\(i,j)\notin \mathcal{K}}}O(1)(\prod_{s\in \mathcal{K}}\t Y_{s}-\textit{mean})}{n^{v(\mathcal{K}\cup (i,j))-2}}, \;\mathcal{I}=\{(k,l): 1\leq k<l \leq n\},
\end{align}
$s_l$ and $t_l$ are the number of two stars and triangles, respectively, in $H_l$, the sum in \cref{Hb} is over all graphs $K$ without isolated vertices which do not contain the edge $(i,j)$, 
and we always use mean to denote the expectation of the random variable in front of it.
In the following, we also regard any edge set $\mathcal{K}$ as a subgraph in $G$.
In fact, there is a  fixed upper bound on $v(\mathcal{K})$ in \cref{Hb} depending on $H_1,\dots, H_m$ in the definition of the ERGM in \cref{eq:ERGM}. %Therefore, the sum $\sum_{\mathcal{K}}$ is over a fixed number of terms. 
We omit the upper bound for notational convenience. This applies to all the $\sum_{\mathcal{K}}$ below.

%Below are some crucial lemmas  with the proofs deferred to the \cref{sec:prooflemma}.
%In the other hand, from the higher-order concentration inequalities in Sambame and Sinulis \cite{Sam2020}, we obtain

We will prove the following lemmas in \cref{sec:prooflemma}, which involve tricky arguments using the Poincar\'e inequality and higher-order concentration inequalities for ERGMs.
\begin{lemma}\label{LH}
     In the subset $B$ of the subcritical region considered in \cref{thm:twostar}, we have
\begin{align}\label{LH1}
\partial_{ij}\bar{H}_A =O_{r} (n^{-\frac{1}{2}}),\quad  \partial_{ij}\bar{H}_B= O_{r} (n^{-1}),\quad \partial_{ij}\bar{H}_C=O_{r} (n^{-\frac{3}{2}}).
     \end{align} 
%\red{$\sum_{l=2}^m$; $s\ne i,j$ conflicts with $s_l$? $O_{r}(n^{-1.5})$ comes from the 2nd term on the right-hand side of \cref{hof2}?}
 \end{lemma}

Recall from above \cref{conj1} that we use $K$ to denote simple graphs without isolated vertices and for any $K$, we use $\t K$ to denote the number of (not necessarily induced) copies of $K$ in $G$ except that each edge
indicator $Y_s$, is replaced by $\t Y_s$. We also denote by $v(K)$ the number of vertices of $K$.

 \begin{lemma}\label{lemmaB3}
 Consider the ERGM \cref{eq:ERGM} with parameters in the subset $B$ of the subcritical region considered in \cref{thm:twostar}.
   For any graph $K$, if $v(K)=3$, we have
 \begin{align*}
     \IE\widetilde{K}=O(n^2).
 \end{align*}
 Moreover, if $v(K)\geq 4$, we have
\begin{align*}
     \IE\widetilde{K}=O(n^{v-2}).
 \end{align*}
 \end{lemma}

  \begin{lemma}\label{LI}
  Under the same assumption as in \cref{lemmaB3},
  for a graph $K$ with $v(K)=v\geq 2$, we have 
     \begin{align}\label{l22}
        \widetilde{K}-\IE\widetilde{K}=O_{r}(n^{v/2}+n^{v-2}).
    \end{align}
 \end{lemma}
 
%Turn to the proof for Theorem \ref{thm:twostar}.  We first prove the central limit theorem for the centered two-star $\t V_n=\sum_{i<j,k\neq i,j}\t Y_{ik} \t Y_{jk}$ by applying Proposition \ref{P1}. Set $X=\widetilde{V}_n-\IE \widetilde{V}_n $ and $Y=\widetilde{E}=\sum_{1\leq i<j\leq n}\t Y_{i,j}$. When there is no ambiguity, we omit the subscripts $n$ for $\widetilde{p}_n$, $V_n$ and $E_n$. From Lemma \ref{L2}, we get $|\td{p}-p|=O(n^{-1})$, then we have

Now we are ready to prove \cref{P2}.
\begin{proof}[Proof of \cref{P2}]
In this proof, we use $C$ to denote constants depending only on the subset $B$ of the subcritical region considered in \cref{thm:twostar}, whose value may differ from line to line, and use $O(1)$ to denote a non-random constant bounded by $C$.

Recall from \cref{20001} and \cref{eq:sigv} that
\ben{\label{eq:WtoV}
W=\frac{\t V-\mu_{\t V}}{\sigma_V},\quad \sigma_V\asymp n^{3/2}.
}
Define as in \cref{eq:Mlpm} that
\begin{align*}
    M_{l,\pm}(\t V,\t E)=
       \IE\left(({\t V}^{\prime}-{\t V})^l\mathbbm{1}_{\Delta \t E =\pm 1}\mid \t V, \t E\right), &\quad l=0,1,2.
\end{align*}

\textbf{Step 1.} 
We first prove \cref{eq:M0W}.  Note that if the edge $(i,j)$ is chosen to be updated in the Glauber dynamics, then $\mathbbm{1}_{\Delta \t E = 1}=\mathbbm{1}_{\t E'-\t E=1}=(1-Y_{ij})Y_{ij}'=(\t q-\t Y_{ij})Y_{ij}'$. Using Taylor's expansion for $e^x/(1+e^x)$ around $x=2\Phi_\beta(p)$ and recalling $p=\exp(2\Phi_\beta(p))/(1+\exp(2\Phi_\beta(p)))$ from \cref{Sub}, we obtain
 \besn{\label{eq:M0pV}
   M_{0,+}(\widetilde{V},\widetilde{E})=&\frac{1}{N}\IE\bigg[\sum_{i<j}(\widetilde{q}-\td{Y}_{ij})Y_{ij}^\prime\mid\td{V},\td{E}\bigg]\\
    =&\frac{1}{N}\IE\bigg[\sum_{i<j}(\widetilde{q}-\td{Y}_{ij})\frac{e^{\partial_{ij}H(G)}}{1+e^{\partial_{ij}H(G)}}\mid \td{V},\td{E}\bigg]\\
    =&\frac{1}{N}\IE\bigg[\sum_{i<j}(\widetilde{q}-\td{Y}_{ij})\left(p+p(1-p)\partial_{ij} \bar{H}+O(\partial_{ij} \bar{H})^2)\right)\mid \td{V},\td{E}\bigg]\\
=:& \t p\,\t q+I_{01}+I_{02}+I_{03}+O(n^{-1}).
}
where we used $p-\t p=O(n^{-1})$ in the last equality and
\begin{align*}
    I_{01}=&-\frac{p\widetilde{E}}{N},\\
    I_{02}=&\frac{p(1-p)}{N}\IE\bigg[\sum_{i<j}(\widetilde{q}-\td{Y}_{ij})\partial_{ij} \bar{H}\mid \td{V},\td{E}\bigg],\\
    I_{03}=&\frac{1}{N}\IE\bigg[\sum_{i<j}(\widetilde{q}-\td{Y}_{ij})O\left(\partial_{ij} \bar{H}\right)^2\mid \td{V},\td{E}\bigg].
\end{align*}
%Hence $M_{0,+}(\widetilde{V},\widetilde{E})=\t p\,\t q+I_{01}+I_{02}+I_{03}+O(n^{-1})$. 
Similarly,   we obtain
\begin{align*}
    M_{0,-}(\widetilde{V},\widetilde{E})=&\frac{1}{N}\IE\Big[\sum_{i<j}(\widetilde{p}+\td{Y}_{ij})(1-Y_{ij}^\prime)\big\mid\td{V},\td{E}\Big]\\=:&\t p\,\t q+J_{01}+J_{02}+J_{03}+O(n^{-1}),
\end{align*} where
\begin{align*}
    J_{01}=&\frac{q\widetilde{E}}{N};\\
    J_{02}=&-\frac{p(1-p)}{N}\IE\bigg[\sum_{i<j}(\widetilde{p}+\td{Y}_{ij})\partial_{ij} \bar{H}\mid \td{V},\td{E}\bigg];\\
    J_{03}=&\frac{1}{N}\IE\bigg[\sum_{i<j}(\widetilde{p}+\td{Y}_{ij})O(\partial_{ij} \bar{H})^2\mid \td{V},\td{E}\bigg].
\end{align*}
Recalling assumption \eqref{ua1}, we set
\begin{align}\label{lc1}
     Q=\t p\,\t q.
\end{align} 
Observe that $I_{01}=O_{r}(n^{-1})$, $J_{01}=O_{r}(n^{-1})$ since $\t E=O_{r}(n)$ from Lemma \ref{LI}. 
%And for  each  integers $s,t>0$, using (5.18) in \cite{fang2024},
%\begin{align*}
%\IE\left[\IE\left[\left|\partial_{ij} \bar{H}\right|^s\mid \td{V},\td{E}\right]^t\right]\leq &\IE\left[\IE\left[\left|\partial_{ij} \bar{H}\right|^{st}\mid \td{V},\td{E}\right]\right]\\=&\IE\left|\partial_{ij} \bar{H}\right|^{st}=O(n^{-st/2}).
%\end{align*}
From \cref{LH}, we have 
$I_{03}=O_{r}(n^{-1}), J_{03}=O_{r}(n^{-1})$.
Moreover, from  \cref{LH,lemmaB3,LI}, we have
\begin{align*}
    I_{02}
    =&\frac{p(1-p)}{N}\E\bigg(\sum_{K:v(K)\geq 2}\frac{O(1)\t K}{n^{v(K)-2}}\Big\mid \t V,\t E\bigg){+ \frac{p(1-p)}{N}\sum_{i<j}(\t q-\t Y_{ij})O(n^{-1})}\\
    =&O_{r}(n^{-1}),\\
    \text{and}\quad\quad\quad\quad&\\
     J_{02}
    =&\frac{p(1-p)}{N}\E\bigg(\sum_{K:v(K)\geq 2}\frac{O(1)\t K}{n^{v(K)-2}}\Big\mid \t V,\t E\bigg){+ \frac{p(1-p)}{N}\sum_{i<j}(\t p+\t Y_{ij})O(n^{-1})}\\
    =&O_{r}(n^{-1}).
\end{align*}
Therefore, we have $R_{0,+}=O_{r}(n^{-1})$ and $R_{0,-}=O_{r}(n^{-1})$.
Recalling \cref{eq:WtoV}, we have obtained in this step that
 \begin{align*}
    & M_{0,\pm}(W,\t E)=M_{0,\pm}(\t V,\t E)=\t p\,\t q+O_{r}(n^{-1}).
 \end{align*}

\textbf{Step 2.} Turning to $M_{1,\pm}(\widetilde{V},\widetilde{E})$, we have
\begin{align}\label{eq:M1plus}
      M_{1,+}=&(\widetilde{V},\widetilde{E})\notag\\=&\frac{1}{N}\IE\Big[\sum_{i<j}(\widetilde{q}-\td{Y}_{ij})Y_{ij}^\prime\sum_{k\neq i,j}(\td{Y}_{ik}+\td{Y}_{jk})\big\mid\td{V},\td{E}\Big]\notag\\
    =&\frac{1}{N}\IE\bigg[\sum_{i<j}(\widetilde{q}-\td{Y}_{ij})\sum_{k\neq i,j}(\td{Y}_{ik}+\td{Y}_{jk})\frac{e^{\partial_{ij}H(G)}}{1+e^{\partial_{ij}H(G)}}\Big\mid \td{V},\td{E}\bigg]\notag\\
    =&\frac{1}{N}\IE\Big[\sum_{i<j}(\widetilde{q}-\td{Y}_{ij})\sum_{k\neq i,j}(\td{Y}_{ik}+\td{Y}_{jk})\big(\t p+p(1-p)\partial_{ij} \bar{H}+O(1)(\partial_{ij} \bar{H})^2\notag\\
    &\quad\quad+O(\partial_{ij} \bar{H})^3+O(n^{-1})\big)\big\mid \td{V},\td{E}\Big]\notag\\
=:&I_{11}+I_{12}+I_{13}+I_{14},
\end{align}
where 
\begin{align}\label{m1p}
    I_{11}=&\frac{\t p}{N}\IE\Big[\sum_{i<j}(\widetilde{q}-\td{Y}_{ij})\sum_{k\neq i,j}(\td{Y}_{ik}+\td{Y}_{jk})\big\mid \td{V},\td{E}\Big]=\frac{2(n-2)\t p\td{q}\td{E}}{N}-\frac{2\t p\td{V}}{N},\\
    I_{12}=&\frac{p(1-p)}{N}\IE\Big[\sum_{i<j}(\widetilde{q}-\td{Y}_{ij})\sum_{k\neq i,j}(\td{Y}_{ik}+\td{Y}_{jk})\partial_{ij} \bar{H}\big\mid \td{V},\td{E}\Big],\notag\\
     I_{13}=&\frac{O(1)}{N}\IE\Big[\sum_{i<j}(\widetilde{q}-\td{Y}_{ij})\sum_{k\neq i,j}(\td{Y}_{ik}+\td{Y}_{jk})\left((\partial_{ij} \bar{H})^2+O(n^{-1})\right)\big\mid \td{V},\td{E}\Big],\notag\\
    I_{14}=&\frac{1}{N}\IE\Big[\sum_{i<j}(\widetilde{q}-\td{Y}_{ij})\sum_{k\neq i,j}(\td{Y}_{ik}+\td{Y}_{jk})O(\partial_{ij} \bar{H})^3\big\mid \td{V},\td{E}\Big].\notag
\end{align}
Similarly, we have
\begin{align*}
    M_{1,-}(\widetilde{V},\widetilde{E})=&-\frac{1}{N}\IE\Big[\sum_{i<j}(\widetilde{p}+\td{Y}_{ij})\big(1-Y_{ij}^\prime\big)\sum_{k\neq i,j}(\td{Y}_{ik}+\td{Y}_{jk})\big\mid\td{V},\td{E}\Big]\\
    =:&J_{11}+J_{12}+J_{13}+J_{14},
\end{align*}
where 
\begin{align}\label{m1m}
    J_{11}=&-\frac{2(n-2)\t p\,\t q\widetilde{E}}{N}-\frac{2\t q\widetilde{V}}{N},\\
    J_{12}=&\frac{p(1-p)}{N}\IE\Big[\sum_{i<j}(\widetilde{p}+\td{Y}_{ij})\sum_{k\neq i,j}(\td{Y}_{ik}+\td{Y}_{jk})\partial_{ij} \bar{H}\bigm| \td{V},\td{E}\Big];\notag\\
    J_{13}=&\frac{O(1)}{N}\IE\Big[\sum_{i<j}(\widetilde{p}+\td{Y}_{ij})\sum_{k\neq i,j}(\td{Y}_{ik}+\td{Y}_{jk})\big((\partial_{ij} \bar{H})^2+O(n^{-1})\big)\big\mid \td{V},\td{E}\Big]\notag\\
     J_{14}=&\frac{1}{N}\IE\Big[\sum_{i<j}(\widetilde{p}+\td{Y}_{ij})\sum_{k\neq i,j}(\td{Y}_{ik}+\td{Y}_{jk})O(\partial_{ij} \bar{H})^3\big\mid \td{V},\td{E}\Big],\notag
\end{align}
and the constant $O(1)$ in $J_{13}$ is the same as that in $I_{13}$. 

In the following, we will first calculate $M_{1,+}(\t V,\t E)+M_{1,-}(\t V,\t E)$ (partly for obtaining $\E \t V$), and then calculate $M_{1,+}(\t V,\t E)$ and $M_{1,-}(\t V,\t E)$, respectively. 

From \cref{eq:partialABC}, we have
\begin{align}\label{IJ12}
    I_{12}+J_{12}=&\frac{p(1-p)}{N}\IE\Big[\sum_{i<j}\sum_{k\neq i,j}(\td{Y}_{ik}+\td{Y}_{jk})(\partial_{ij}\bar{H}_A+\partial_{ij}\bar{H}_B+\partial_{ij}\bar{H}_C)\big\mid \td{V},\td{E}\Big]\notag\\
    &+\frac{O(1)}{Nn}\IE\Big[\sum_{i<j}\sum_{k\neq i,j}(\td{Y}_{ik}+\td{Y}_{jk})\big\mid \td{V},\td{E}\Big].
\end{align}
From \cref{LH}, we have 
\begin{align*}
    \E\big[\sum_{i<j}\sum_{k\neq i,j}(\td{Y}_{ik}+\td{Y}_{jk})\partial_{ij}\bar{H}_C\big]=N \cdot O_r(n^{1/2})\cdot O_r(n^{-3/2})=O_r(n).
\end{align*}
%$\sum_{k\neq i,j}(\td{Y}_{ik}+\td{Y}_{jk})=O_{r}(n^{1/2})$ and
and
\begin{align}\label{edge}
    \frac{1}{n}\IE\Big[\sum_{i<j}\sum_{k\neq i,j}(\td{Y}_{ik}+\td{Y}_{jk})\big\mid \td{V},\td{E}\Big]=\frac{1}{n}\big(2n+O(1)\big)\t E=O_{r}(n).
\end{align} 

%and applying Lemmas \ref{lemmaB3} and \ref{LI}, we have
%\begin{align}\label{IJ121}
%    \|\IE\bigg[\sum_{i<j}\sum_{k\neq i,j}(\td{Y}_{ik}+\td{Y}_{jk})\mid \td{V},\td{E}\bigg]\|_p=\|O(1)\t V\|_p= C_p N.
%\end{align}
%Therefore, the last term on the right-hand side of \cref{IJ12} is of the correct order \red{(i.e. belongs to $R_{1,+}+R_{1,-}$ in \cref{P2})} since it will be divided by $\sigma_V\asymp n^{3/2}$ in the end.  
We argue (up to \cref{eq:I12J12}) that the main contribution to $I_{12}+J_{12}$ comes from the term involving $\partial_{ij}\bar{H}_A$. Recall that each $H_l$ has $s_l$ two stars and $t_l$ triangles. From \cref{Ha}, we have
\begin{align*}
    &\IE\Big[\sum_{i<j}\sum_{k\neq i,j}(\td{Y}_{ik}+\td{Y}_{jk})\partial_{ij}\bar{H}_A\big\mid \td{V},\td{E}\Big]\notag\\
    =&\sum_{l=2}^m\IE\Big[\sum_{i<j}\sum_{k\neq i,j}(\td{Y}_{ik}+\td{Y}_{jk})n^{-1}\big(\sum_{s\neq i,j}\big(2\beta_ls_l\widetilde{p}^{e_l-2}(\widetilde{Y}_{is}+\widetilde{Y}_{js})\\
    &\quad\quad\quad+6\beta_lt_l\widetilde{p}^{e_l-3}(\widetilde{Y}_{is}\widetilde{Y}_{sj}-\textit{mean})\big)\big)\big|\td{V},\td{E}\Big]\notag\\
    =&\sum_{i<j}\IE\Big[\frac{2\sum_{l=2}^m\beta_ls_l\widetilde{p}^{e_l-2}}{n}\sum_{k\neq i,j}\sum_{s\neq i,j}(\t Y_{ik}+\t Y_{jk})(\t Y_{is}+\t Y_{js})\\
    &\qquad \qquad +\frac{6\sum_{l=2}^m\beta_lt_l\widetilde{p}^{e_l-3}}{n}\sum_{k\neq i,j}\sum_{s\neq i,j}(\t Y_{ik}+\t Y_{jk})(\t Y_{is}\t Y_{js})\notag\\
    &\qquad \qquad +O(n^{-1})\sum_{k\neq i,j}(\t Y_{ik}+\t Y_{jk})\Big|\td{V},\td{E}\big]\notag\\
    &=:I+II+III,\notag
\end{align*}
where we used $\textit{mean}=O(1/n)$ (cf. \cref{GN}) in the second equality.\\
For $I$, expanding the square, and using  
\begin{align}\label{star}
    \t Y_{ik}^2=(1-2\t p)Y_{ik}+\t p^2=(1-2\t p)\t Y_{ik}+\t p\,\t q,
\end{align} we obtain
\begin{align*}
I=&\IE\Big[\sum_{i<j}\sum_{l=2}^m\frac{2\beta_ls_l\widetilde{p}^{e_l-2}}{n}\sum_{\substack{k\neq s\\ k\neq i,j\\s\neq i,j}}(\t Y_{ik}\t Y_{is}+\t Y_{ik}\t Y_{js}+\t Y_{jk}\t Y_{is}+\t Y_{jk}\t Y_{js})\big|\td{V},\td{E}\Big]\\
    &+\IE\Big[\sum_{i<j}\sum_{l=2}^m\frac{2\beta_ls_l\widetilde{p}^{e_l-2}}{n}\sum_{k\neq i,j}((1-2\t p)(\t Y_{ik}+\t Y_{jk})+2\t p\,\t q+2\t Y_{ik}\t Y_{jk})\big|\td{V},\td{E}\Big]\\
    =&\sum_{l=2}^m 2\beta_ls_l\widetilde{p}^{e_l-2}\big[2\t V+4\t E^2/n+2(1-2\t p)\t E+ 2N\t p\,\t q\big]+O(n^{-1})\t V+O(n^{-1})\t E+O(n).
\end{align*}
For $II$, grouping terms according to $k$ equal $s$ or not, and using \eqref{star} again, we obtain
\begin{align*}
    II=&6n^{-1}\sum_{l=2}^m\beta_lt_l\widetilde{p}^{e_l-3}\IE \Big[\sum_{i<j}\sum_{k\neq i,j}\sum_{s\neq i,j}(\t Y_{ik}+\t Y_{jk})(\t Y_{is}\t Y_{js})\Big|\td{V},\td{E}\Big]\\
    =&\IE\Big[O(n^{-1})\sum_{i<j}\sum_{\substack{k\neq s:\\ k\neq i,j\\s\neq i,j}}(\t Y_{ik}+\t Y_{jk})(\t Y_{is}\t Y_{js})+O(n^{-1})\sum_{i<j}\sum_{k\neq i,j}(\t Y_{ik}+\t Y_{jk})(\t Y_{ik}\t Y_{jk})\big|\td{V},\td{E}\Big]\\
    =&\IE\big(O(n^{-1})\t \sqcap+O(n^{-1})\t V|\td{V},\td{E} \big)+O(1)\t E.
\end{align*}
where $  \t \sqcap:=\sum_{i<j}\sum_{\substack{k\ne l:\\ k\neq i,j\\l\ne i,j}}\t Y_{ik}\t Y_{kl}\t Y_{lj}$. \\
For $III$, from \eqref{edge}, we have
\begin{align*}
    III=O(1)\t E.
\end{align*}
Combining the expressions of $I$, $II$ and $III$, we obtain
\begin{align}\label{A2}
&\IE\Big[\sum_{i<j}\sum_{k\neq i,j}(\td{Y}_{ik}+\td{Y}_{jk})\partial_{ij}\bar{H}_A\big\mid \td{V},\td{E}\Big]\notag\\=&4N\t q\sum_{l=2}^m\beta_ls_l\widetilde{p}^{e_l-1}+4\sum_{l=2}^m\beta_ls_l\widetilde{p}^{e_l-2}\t V\notag\\
&+\mathbb{E}(O(n^{-1})\t E^2+ O(n^{-1})\t \sqcap+O(n^{-1})\t V|\t V,\t E)+O(1)\t E+O(n) \notag\\
=&4N\t q\sum_{l=2}^m\beta_ls_l\widetilde{p}^{e_l-1}+4\sum_{l=2}^m\beta_ls_l\widetilde{p}^{e_l-2}\t V+O_{r}(n),
\end{align}
where we use  Lemmas  \ref{lemmaB3} and \ref{LI}  in the last equality.

Now we turn to the term containing $\partial_{ij}\bar{H}_B$ on the right-hand side of \cref{IJ12}.  Multiplying  \cref{Hb} by $\t Y_{ik}$, grouping terms according to $(i,k)\notin \mathcal{K}$ or $(i,k)\in \mathcal{K}$, and in the latter case using \eqref{star}, we obtain
\begin{align*}
    &\IE\bigg[\sum_{i<j}\sum_{k\neq i,j}\td{Y}_{ik}\partial_{ij}\bar{H}_B\mid \td{V},\td{E}\bigg]\\
    =&\IE\Big[\sum_{i<j}\sum_{k\neq i,j}\td{Y}_{ik}\Big(\sum_{\mathcal{K}\subset\mathcal{I}:\substack{v(\mathcal{K}\cup (i,j))= 4\\(i,j)\notin \mathcal{K}}}O(n^{-v(\mathcal{K}\cup (i,j))+2})(\prod_{s: s\in \mathcal{K}}\t Y_{s}-\textit{mean})\Big)\big\mid\td{V},\td{E}\Big]\\
    =&\IE\Big[\sum_{i<j}\sum_{k\neq i,j}\Big(\sum_{\mathcal{K}\subset\mathcal{I}:\substack{v(\mathcal{K}\cup (i,j)\cup (i,k))=5\\(i,j)\notin \mathcal{K},k\notin \mathcal{V}(\mathcal{K})}}O(n^{-v(\mathcal{K}\cup (i,j)+2})\prod_{s: s\in \mathcal{K}\cup (i,k)}\t Y_{s}\Big)\big|\td{V},\td{E}\Big]\\
    &+\IE\bigg[\sum_{i<j}n^{-2}\bigg(\sum_{\mathcal{K}\subset\mathcal{I}:\substack{v(\mathcal{K}\cup (i,j))= 4\\(i,j)\notin \mathcal{K}}}O(1)\Big(\prod_{s: s\in \mathcal{K}}\t Y_{s}+\sum_{k\in \mathcal{V}(\mathcal{K})\backslash \{i,j\}}n^{v_{\rm{iso}}}\prod_{s: s\in \mathcal{K}\backslash (i,k)}\t Y_{s}\Big)\bigg)\Big|\td{V},\td{E}\bigg]\\
    &+\IE\Big[O(n^{-1})\sum_{i<j}\sum_{k\neq i,j}\td{Y}_{ik} \Big|\td{V},\td{E}\Big],
\end{align*}
where $v_{\rm{iso}}$ denotes the number of isolated vertices in $\mathcal{K}\backslash (i,k)$ without counting in the possibly isolated vertex $\{i\}$ (because of the summation over $i$ in front). 
%$v_{\rm{iso}}(H)$ denote the number of isolate vertices in $H$. 
From \cref{lemmaB3,LI}, we obtain that the order of these terms are $O_{r}(n)$  except that we need some additional arguments for the term (which results from the constant term in the expression of $\t Y_{ik}^2$ in \cref{star})
\begin{align}\label{except}
\IE\big[\sum_{i<j}\sum_{\mathcal{K}\subset\mathcal{I}:\substack{v(\mathcal{K}\cup (i,j))= 4\\(i,j)\notin \mathcal{K}}}\sum_{k\in \mathcal{V}(\mathcal{K})\backslash \{i,j\}}n^{v_{\rm{iso}}-2}\prod_{s: s\in \mathcal{K}\backslash (i,k)}\t Y_{s}\mid \t V,\t E \big].
\end{align} 
If both $i$ and $k$ are isolated in $\mathcal{K}\backslash (i,k)$, then $\mathcal{K}\backslash (i,k)$ is an edge (recall $v(\mathcal{K}\cup (i,j))=4$) connecting to $j$ and \cref{except} is of order $O(n^2/n^2)\t E=O_{r}(n)$. If only one of the vertices $i$ and $k$ is isolated in $\mathcal{K}\backslash (i,k)$, then \cref{except} is of order $O(n/n^2)O_{r}(n^2)=O_{r}(n)$.

%Note that $v_{\rm{iso}}(K\backslash (i,k))\leq 2$.
%and $i$ must be a isolate vertex of $K\backslash (i,k)$ since $(i,j)\notin K$. If $v_{\rm{iso}}(K\backslash (i,k))=1$, then the vertices of the graph consisted of $\{s_l:s_l\in \mathcal{E}(K)\backslash (i,k)\}$ is either 2 or $3$. By \cref{lemmaB3,LI}, the order of \eqref{except} is $O_{r}(n)$. If $v_{\rm{iso}}(K\backslash (i,k))=2$, then the vertices of the graph consisted of $\{s_l:s_l\in \mathcal{E}(K)\backslash (i,k)\}$ is $2$ and the order of \eqref{except} is also $O_{r}(n)$. 

Together with the same estimate for $\IE\big[\sum_{i<j}\sum_{k\neq i,j} 
 \td{Y}_{jk}\partial_{ij}\bar{H}_B\mid \td{V},\td{E}\big]$, we obtain
\begin{align}\label{IJ123}
    \IE\big[\sum_{i<j}\sum_{k\neq i,j}(\td{Y}_{ik}+\td{Y}_{jk})\partial_{ij}\bar{H}_B\mid \td{V},\td{E}\big]=O_{r}(n).
\end{align}
Substituting  \eqref{edge}, \eqref{A2} and \eqref{IJ123} into \eqref{IJ12}, we obtain
\begin{align}\label{eq:I12J12}
    I_{12}+J_{12}=&\frac{pq}{N}\big(I+II+III+O_{r}(n)\big)\notag\\
    =&4\t q^2\sum_{l=2}^m\beta_ls_l\widetilde{p}^{e_l}+\frac{4\t q\t V}{N}\sum_{l=2}^m\beta_ls_l\widetilde{p}^{e_l-1}+O_{r}(n).
\end{align}
Turn to $I_{13}$. From \cref{LH,lemmaB3,LI}, \cref{ha}, \cref{eq:partialABC} and \cref{edge}, we have
\begin{align*}
        I_{13}=&\frac{O(1)}{N}\IE\Big[\sum_{i<j}\sum_{k\neq i,j}(\t q-\t Y_{ij})(\td{Y}_{ik}+\td{Y}_{jk})\big(\partial_{ij}\bar{H}_A+\partial_{ij}\bar{H}_B+\partial_{ij}\bar{H}_C+O(n^{-1})\big)^2\big\mid \td{V},\td{E}\Big]\\
        &+O(n^{-1})\frac{p(1-p)}{N}\IE\Big[\sum_{i<j}\sum_{k\neq i,j}(\t q-\t Y_{ij})(\td{Y}_{ik}+\td{Y}_{jk})\big\mid \td{V},\td{E}\Big]\\
        =&\frac{O(1)}{N}\IE\Big[\sum_{i<j}\sum_{k\neq i,j}(\t q-\t Y_{ij})(\td{Y}_{ik}+\td{Y}_{jk})(\partial_{ij}\bar{H}_A)^2\mid \td{V},\td{E}\Big]+O_{r}(n^{-1}).
\end{align*}
Using \cref{Ha}, expanding the products of summations, considering all the subgraph counts they form, and applying \cref{LH,lemmaB3,LI} and \cref{star}, we obtain
\begin{align}\label{R1P3'}
   &\IE\bigg[\sum_{i<j}\sum_{k\neq i,j}(\t q-\t Y_{ij})(\td{Y}_{ik}+\td{Y}_{jk})(\partial_{ij}\bar{H}_A)^2\mid \td{V},\td{E}\bigg]\notag\\
   =&\IE\bigg[\sum_{i<j}\sum_{k\neq i,j}(\t q-\t Y_{ij})(\td{Y}_{ik}+\td{Y}_{jk})\sum_{s\neq i,j}\Big(O(n^{-1})(\td{Y}_{is}+\td{Y}_{js})\notag\\
   &\qquad \qquad\qquad \qquad \qquad \qquad\qquad \qquad +O(n^{-1})\big(\td{Y}_{is}\td{Y}_{js}+O(n^{-1})\big)\Big)^2\mid \td{V},\td{E}\bigg]\notag\\
   =& \frac{1}{N}\IE\bigg[\sum_{i<j}(\t q-\t Y_{ij})\Big(O(1)\big(\sum_{k\neq i,j}(\td{Y}_{ik}+\td{Y}_{jk})\big)^{3\ \text{or}\ 2\ \text{or}\ 1}+O(1)\big(\sum_{k\neq i,j}(\td{Y}_{ik}+\td{Y}_{jk})\big)^{2\ \text{or}\  1}\notag\\
   &\qquad \qquad \qquad \qquad \times \big(\sum_{s\neq i,j}\td{Y}_{is}\td{Y}_{js}\big)+O(1)\sum_{k\neq i,j}(\td{Y}_{ik}+\td{Y}_{jk})\big(\sum_{s\neq i,j}\td{Y}_{is}\td{Y}_{js}\big)^2\Big)\mid \td{V},\td{E}\bigg]\notag\\
   =&\frac{1}{N}\E\big(\sum_{K:4\leq v(K)\leq 5} O(1)\t K+O(n) \t \triangle \mid \td{V},\td{E}\big)+O(n^{-1})\t V+O(n^{-1})\t E\notag+O(1)\\
   =&O_{r}(n).
\end{align}
Therefore, 
\begin{align*}
    I_{13}=O_{r}(n^{-1}).
\end{align*} 
Similarly, we have $J_{13}=O_{r}(n^{-1})$.
Moreover, from \cref{ha,LH1}, we have
\begin{align*}
I_{14}, J_{14}=O_{r}(N^{-1}\cdot N\cdot \sqrt{n}\cdot n^{-3/2})=O_{r}(n^{-1}).
%\|I_{14}\|_p+\|J_{14}\|_p\leq C_p \left(N^{-p}[N^p\cdot \sqrt{n}^{p}\cdot n^{-1.5p}]\right)^{1/p}=C_pn^{-1} .
\end{align*}
From \cref{m1p}, \cref{m1m}, \cref{eq:I12J12} and the above estimates of $I_{13},J_{13},I_{14},J_{14}$,
we can write
\begin{align}\label{A1}
    &M_{1,+}(\t V,\widetilde{E})+M_{1,-}(\t V,\widetilde{E})\notag\\
    =&-\frac{2}{N}\Big[(1-2\t q\sum_{l=2}^m\beta_ls_l\t p^{e_l-1})\t V+2\t q^2\sum_{l=2}^m\beta_ls_lN\widetilde{p}^{e_l}+O_{r}(n)\Big].
\end{align}
%From the above analysis, the ``$\textit{remains}$'' is $O_{r}(n)$. 
From the fact that $\IE(M_{1,+}(\t V,\t E)+M_{1,-}(\t V,\t E))=0$ (expectation of an anti-symmetric function of an exchangeable pair is 0), we obtain 
\begin{align}\label{IEV}
\IE\widetilde{V}=\frac{2N\t q^2\sum_{l=2}^m\beta_ls_l\t p^{e_l}}{1-2\t q\sum_{l=2}^m\beta_ls_l\t p^{e_l-1}}+O_{r}(n).
%\quad R_{1,+}(\t V,\t E)+R_{1,-}(\t V,\t E)=O_{r}(n).
\end{align}
Next, we calculate $M_{1,+}(\t V,\t E)$. From \cref{eq:M1plus} and discarding some small order terms as above, we obtain 
\begin{align}\label{R1P}
    M_{1,+}(\t V,\t E)=&-\frac{2\t p\td{V}}{N}+\frac{2(n-2)\t p\,\td{q}\td{E}}{N}\notag\\
    &+\frac{p(1-p)}{N}\IE\bigg[\sum_{i<j}(\t q-\t Y_{ij})\sum_{k\neq i,j}(\td{Y}_{ik}+\td{Y}_{jk})(\partial_{ij}\bar{H}_A+\partial_{ij}\bar{H}_B)\mid \td{V},\td{E}\bigg]\notag\\
    &+\frac{O(1)}{N}\IE\bigg[\sum_{i<j}(\t q-\t Y_{ij})\sum_{k\neq i,j}(\td{Y}_{ik}+\td{Y}_{jk})(\partial_{ij}\bar{H}_A)^2\mid \td{V},\td{E}\bigg]+O_{r}(n^{-1}).
    \end{align}
    From Lemma \ref{lemmaB3} and following the calculations below \cref{star}, we obtain
     \begin{align}\label{R1P1}
&\IE\Big[\sum_{i<j}(\t q-\t Y_{ij})\sum_{k\neq i,j}(\td{Y}_{ik}+\td{Y}_{jk})\partial_{ij}\bar{H}_A\big\mid \td{V},\td{E}\Big]\notag\\
=& \t q \E\Big[\sum_{i<j}\sum_{k\neq i,j}(\td{Y}_{ik}+\td{Y}_{jk})\partial_{ij}\bar{H}_A\big\mid \td{V},\td{E}\Big] \notag\\
&+\sum_{l=2}^m\IE\Big[\sum_{i<j}\sum_{k\neq i,j}\t Y_{ij}(\td{Y}_{ik}+\td{Y}_{jk})n^{-1}\Big(\sum_{s\neq i,j}\big(2\beta_ls_l\widetilde{p}^{e_l-2}(\widetilde{Y}_{is}+\widetilde{Y}_{js})\notag\\&\qquad\qquad\qquad\qquad\qquad\qquad\qquad\qquad+6\beta_lt_l\widetilde{p}^{e_l-3}(\widetilde{Y}_{is}\widetilde{Y}_{sj}-\textit{mean})\big)\Big)\Big|\td{V},\td{E}\Big]\notag\\
=& 4\t q\sum_{l=2}^m\beta_l s_l\t p^{e_l-2}\t V+4N\sum_{l=2}^m\beta_l s_l{\t p}^{e_l-1}\t q^2+O_{r}(n).
    \end{align}

Following the same calculation leading to \cref{IJ123}, we obtain
\begin{align}\label{R1P2}
    \IE\bigg[\sum_{i<j}(\t q-\t Y_{ij})\sum_{k\neq i,j}(\td{Y}_{ik}+\td{Y}_{jk})\partial_{ij}\bar{H}_B\mid \td{V},\td{E}\bigg]=O_{r}(n).
\end{align}
From \eqref{R1P3'}, we have
\begin{align}\label{R1P3}
\IE\bigg[\sum_{i<j}(\t q-\t Y_{ij})\sum_{k\neq i,j}(\td{Y}_{ik}+\td{Y}_{jk})(\partial_{ij}\bar{H}_A)^2\mid \td{V},\td{E}\bigg]=O_{r}(n).
\end{align}
Substituting \eqref{R1P1}, \eqref{R1P2} and \eqref{R1P3} into \eqref{R1P} yields
  \begin{align}\label{r1}
 M_{1,+}(\t V,\t E)= &-\frac{2}{N}(\t p-2\sum_{l=2}^m\beta_l s_l\t p^{e_l-1}\t q^2)(\t V-\IE \t V)+\frac{2n\t p\td{q}\td{E}}{N}\notag\\
 &+O(1)+O_{r}(n^{-1}) .
\end{align}

Similarly, we have
\begin{align}\label{r2}
    M_{1,-}(\t V,\t E)= &-\frac{2}{N}(\t q-2\sum_{l=2}^m\beta_l s_l\t p^{e_l}\t q)(\t V-\IE \t V)-\frac{2n\t p\td{q}\td{E}}{N}\notag\\
 &+O(1)+O_{r}(n^{-1}) .
\end{align}
Recalling \eqref{ua2} and substituting \eqref{r1} and \eqref{r2}, we deduce that 
\begin{align}\label{m11}
\lambda=&\frac{2(1-2\t q\sum_{l=2}^m\beta_ls_l\t p^{e_l-1})}{N},\notag\\
    a_+=\frac{\t p-2\t q^2\sum_{l=2}^m\beta_l s_l\t p^{e_l-1}}{1-2\t q\sum_{l=2}^m\beta_ls_l\t p^{e_l-1}},&\quad \quad \quad\quad a_{-}=\frac{\t q-2\t q\sum_{l=2}^m\beta_l s_l\t p^{e_l}}{1-2\t q\sum_{l=2}^m\beta_ls_l\t p^{e_l-1}}.
\end{align} 

Recall \cref{eq:WtoV}. We have obtained in this step that
 \begin{align*}
    & M_{1,+}(W,\t E)+ M_{1,+}(W,\t E)=-\lambda \big(W+O_{r}(n^{-1/2})\big),\\
    & M_{1,+}(W,\t E)=-\lambda \left(a_+W+O(n^{1/2})+O_{r}(n^{-1/2})-\frac{2n\t p\,\t q}{N\lambda}\cdot \frac{\t E}{\sigma_{V}}\right),\\
      &M_{1,-}(W,\t E)=-\lambda \left(a_-W+O(n^{1/2})+O_{r}(n^{-1/2})+\frac{2n\t p\,\t q}{N\lambda}\cdot \frac{\t E}{\sigma_{V}}\right).
 \end{align*}

\textbf{Step 3.} Finally, we consider  $M_{2,\pm}(\t V, \t E)$. We have 
\begin{align*}
    M_{2,+}(\widetilde{V},\widetilde{E})=&\frac{1}{N}\IE\Big[\sum_{i<j}(\widetilde{q}-\td{Y}_{ij})Y_{ij}^\prime\Big[\sum_{k\neq i,j}(\td{Y}_{ik}+\td{Y}_{jk})\Big]^2\big\mid\td{V},\td{E}\Big]\\
        =&\frac{1}{N}\IE\Big[\sum_{i<j}(\widetilde{q}-\td{Y}_{ij})\Big[\sum_{k\neq i,j}(\td{Y}_{ik}+\td{Y}_{jk})\Big]^2\left(\t p+O\left(|\partial_{ij} \bar{H}|+n^{-1}\right)\right)\big\mid\td{V},\td{E}\Big]\\
    =:&I_{21}+I_{22},
    \end{align*}
where 
\be{
I_{21}=\frac{\t p}{N}\IE\Big[\sum_{i<j}(\t q-\t Y_{ij})\big(\sum_{k\neq i,j}(\t Y_{ik}+\t Y_{jk})\big)^2\mid \t V,\t E\Big].
}
We also have
\begin{align*}
  M_{2,-}(\widetilde{V},\widetilde{E})=&\frac{1}{N}\IE\Big[\sum_{i<j}(\widetilde{p}+\td{Y}_{ij})(1-Y_{ij}^\prime)\big[\sum_{k\neq i,j}(\td{Y}_{ik}+\td{Y}_{jk})\big]^2\big\mid\td{V},\td{E}\Big]\\
    =&\frac{1}{N}\IE\Big[\sum_{i<j}(\widetilde{p}+\td{Y}_{ij})\big[\sum_{k\neq i,j}(\td{Y}_{ik}+\td{Y}_{jk})\big]^2\left(\t q+O\left(|\partial_{ij} \bar{H}|+n^{-1}\right)\right)\big\mid\td{V},\td{E}\Big]\\
    =:&J_{21}+J_{22},
    \end{align*}   
where 
\be{
J_{21}=\frac{\t q}{N}\IE\Big[\sum_{i<j}(\t p+\t Y_{ij})\big(\sum_{k\neq i,j}(\t Y_{ik}+\t Y_{jk})\big)^2\mid \t V,\t E\Big].
}
For $I_{22}$ or $J_{22}$, due to Lemma \ref{LH} and \cref{ha}, 
\begin{align*}
    I_{22}, J_{22}= O_{r}(N^{-1}\cdot N\cdot n\cdot n^{-1/2})=O_{r} (n^{1/2}).
\end{align*}
From \cref{LH,lemmaB3,LI} and \cref{star}, we have
\begin{align*}
   I_{21}=&\frac{\t p}{N}\IE\Big[\sum_{i<j}(\t q-\t Y_{ij})\big(\sum_{k\neq i,j}(\t Y_{ik}+\t Y_{jk})\big)^2\mid \t V,\t E\Big]\\
    =&\frac{\t p}{N}\IE\Big[\sum_{i<j}(\t q-\t Y_{ij})\Big[\sum_{\substack{k\neq s:\\ k\neq i,j\\s\neq i,j}}(\t Y_{ik}\t Y_{is}+\t Y_{ik}\t Y_{js}+\t Y_{jk}\t Y_{is}+\t Y_{jk}\t Y_{js})\\
&\quad\quad+\sum_{k\neq i,j}((1-2\t p)(\t Y_{ik}+\t Y_{jk})+2\t p\,\t q+2\t Y_{ik}\t Y_{jk})\Big]\big\mid \t V,\t E\Big]\\
   =&\frac{2\widetilde{p}\t q}{N}\IE\left[n^3\t p\,\t q/2+O(n^2)+n \widetilde{V}+O(1) \t \triangle+O(1)\t \sqcap+O(1)\t\tstar + O(n)\t E \mid \t V,\t E\right]
   \\
  =&\frac{2\t p\,\t q}{N}\IE\left[n^3\t p\,\t q/2+n\IE\widetilde{V}+O_{r}(n^{5/2})\right].
  \end{align*}
%where   (change) $\widetilde{\tstar}=\sum_{i<j<k<l}\widetilde{Y}_{ij}\widetilde{Y}_{ik}\widetilde{Y}_{il}.$  
Similarly, we also have 
\begin{align*}
    J_{21}=&\frac{2\t p\,\t q}{N}\IE\left[n^3\t p\,\t q/2+n\IE\widetilde{V}+O_{r}(n^{5/2})\right].
\end{align*}
%From the expression of $\sigma_V^2$ below \cref{eq:thm1}, 
We find that assumption \cref{ua3} is satisfied if we set
\begin{align}\label{m2}
\sigma^2_{V}=\frac{Nn\t p^2\,\t q^2+n\t p\,\t q\IE\widetilde{V}}{1-2\t q\sum_{l=2}^m\beta_ls_l\t p^{e_l-1}}=\frac{Nn\t p^2\,\t q^2}{(1-2\t q\sum_{l=2}^m\beta_ls_l\t p^{e_l-1})^2},
\end{align}
which is the same as \cref{eq:sigv}.

Recall \cref{eq:WtoV}. We have obtained in this step that
 \begin{align*}
     & M_{2,\pm}(W,\t E)=\lambda(1+O_{r}(n^{-1/2})).
 \end{align*}

Combining the three steps, we have finished the proof of \cref{P2}.
 \end{proof}
 
%Now we turn to the proof of Theorem \ref{thm:twostar}, which apply Propositions \ref{P1} and \ref{P2} with $k=0$.

\section{Proof of \cref{L2,lem:lclt}}\label{sec:lem}

\subsection{Proof of \cref{L2}}\label{sec:lem1}

We prove \cref{L2} by adapting the approach of \cite{reinert2019}. However, one difference is that to obtain faster convergence rate for $|p-\t p|$, we apply Stein's method using the Stein operator for $G(n,p)$, while they used the Stein operator for the ERGM.

\begin{proof}[Proof of \cref{L2}]
We identify a simple graph on $n$ labeled vertices $\{1,\dots, n\}$ with an element $x=(x_{ij})_{1\leq i<j\leq n}\in \{0,1\}^{\mathcal{I}}$, where ${\mathcal{I}}:={\mathcal{I}_n}:=\{(i,j):1\leq i<j\leq n\}$ and $x_{ij}=1$ if and only if there is an edge between vertices $i$ and $j$. In this way, the ERGM \cref{eq:ERGM} induces a random element $Y\in \{0,1\}^{\mathcal{I}}$. Similarly, $G(n,p)$ induces a random element $X$ in $\{0,1\}^{\mathcal{I}}$.

   For any $x\in \{0,1\}^{\mathcal{I}}$, define $h(x)=\sum_{s\in \mathcal{I}} x_s$.
   It suffices to prove that 
   \begin{align}\label{eq:diffn}
       |\IE h(Y)-\IE h(X)|\leq Cn.  
    \end{align}
    %for $Y\sim ERGM(\beta)$ and $X\sim G(n,p)$. 
    Define the continuous-time Glauber dynamics $\{Y_t, t\geq 0\}$ for $Y$ similar to the discrete-time case above \cref{tran}, except with
exponential rate $1$ holding times.
As in \cite{reinert2019}, the continuous-time process has the generator $\mathcal{A}_Y$ given by
\begin{align*}
    \mathcal{A}_Y f(x)=\frac{1}{N}\sum_{s\in \mathcal{I}}\left[q_Y\big(x_{s+}|x\big)\Delta_s f(x)+\left(f(x_{s-})-f(x)\right)\right],
\end{align*}
where  $x_{s+}$ has a $1$ in the $s$th coordinate and is otherwise the same as $x$,
$x_{s-}$ has a $0$ in the $s$th coordinate and is otherwise the same as $x$, 
\be{
\Delta_s f(x):=f(x_{s+})-f(x_{s-})
}
and
\begin{align}\label{eq:qX}
    q_{Y }(x_{s+}|x) :=\P(Y_s = 1|(Y_u)_{u\neq s} = (x_u)_{u\neq s}).
\end{align}
We define $\{X_t, t\geq 0\}$, $\mathcal{A}_X$ and $q_X$ similarly for $G(n,p)$.

Now we apply the generator approach of Stein's method to prove \cref{eq:diffn}. We set up the Stein equation
\begin{equation}
    \label{30001}
  \mathcal{A}_Xf_h(x)=h(x)-\IE h(X).
\end{equation}
It is known that it has the solution
\begin{align*}
    f_h(x):=-\int_{0}^{\infty} \IE[h(X_t)-h(X)\mid X_0=x]\d t.
\end{align*} 
 We also know that the generator satisfies $\IE \mathcal{A}_Yf_h(Y)=0$.  Hence,  by \cref{30001,eq:qX}, we obtain
\begin{align}\label{SO}
    &\IE h(Y)-\IE h(X)\notag\\
    =&\IE\mathcal{A}_Xf_h(Y)\notag\\
   =&\IE\mathcal{A}_Xf_h(Y)-\IE\mathcal{A}_Yf_h(Y)\notag\\
   =&\frac{1}{N}\sum_{s\in \mathcal{I}}\IE\left[\big(q_Y(Y^{(s,1)}\mid Y)-q_X(Y^{(s,1)}\mid Y)\big)\Delta_s f_h(Y)\right].
\end{align}

We next prove that $\Delta_s f_h(x)=N.$ 
\begin{lemma}\label{L3}
For any $x\in \{0,1\}^{\mathcal{I}}$, we have $ \Delta_s f_h(x)=N.$
\end{lemma}
\begin{proof}[Proof of \cref{L3}]
Let $0=T_0 < T_1 < T_2 < \cdots$ be the jump times of the continuous-time Glauber dynamics Markov chain. The jump times are independent of the discrete skeleton.
Because $X$ is $G(n,p)$, we can couple the Markov chains starting from two different initial values $x_{s+}$ and $x_{s-}$ in such a way that their jump times are equal, they always choose the same edge to update and update that edge to the same value.

From the definition of $f_h(x)$, we have 
    \begin{align*}
        \Delta_s f_h(x)=&\int_0^{\infty}\IE\big[h(X_t)\mid X_0=x_{s+}\big]-{E}\big[h(X_t)\mid X_0=x_{s-}\big]\d t\\
        =&\sum_{m=0}^{\infty}\IE\left[\int_{T_m}^{T_{m+1}}\big[h(U^{[x,s]}(m))-h(V^{[x,s]}(m))\big]\d t\right]\\
        =&\sum_{m=0}^{\infty}\IE\big[h(U^{[x,s]}(m))-h(V^{[x,s]}(m))\big],   \end{align*}
    where $\mathcal{L}(U^{[x,s]}(m))=\mathcal{L}(X_t\mid X_0=x_{s+})$ for $T_m\leq t<T_{m+1}$, and $\mathcal{L}(V^{[x,s]}(m))=\mathcal{L}(X_m\mid X_0=x_{s-})$ for $T_m\leq t<T_{m+1}$.
    Note that $h(U^{[x,s]}(0))-h(V^{[x,s]}(0))=1$ because the initial values differ by one edge.
    Moreover, $\E[h(U^{[x,s]}(1))-h(V^{[x,s]}(1))]=1-N^{-1}$ because $h(U^{[x,s]}(1))-h(V^{[x,s]}(1))=0\  \text{or}\ 1$ and the probability of 0 is $N^{-1}$ (the probability that the edge $s$ is chosen to be updated).
    % , we consider the result when $m=1$, which is 
    % $1-N^{-1}.$ 
    %Using grand coupling, we can keep $V^{[x,s]}(m)\subseteq U^{[x,s]}(m)$ for every $m$.  
    By induction,
    \begin{align*}
        &\IE\big[h(U^{[x,s]}(m))-h(V^{[x,s]}(m))\big]=(1-N^{-1})^{m},
        %\\
        %=&\IE\left[\IE\big[h(U^{[x,s]}(m))-h(V^{[x,s]}(m))\mid U^{[x,s]}(m-1),V^{[x,s]}(m-1)\big]\right]\\
        %=&(1-N^{-1})\IE\big[ U^{[x,s]}(m-1)-V^{[x,s]}(m-1)\big]
        %\\&\quad\cdots\\
        %=&(1-N^{-1})^{m},
    \end{align*}
    which implies 
    \begin{align*}
        \Delta_s f_h(x)=\sum_{m=0}^{\infty}\IE\big[h(U^{[x,s]}(m))-h(V^{[x,s]}(m))\big]
        =\sum_{m=0}^{\infty}(1-N^{-1})^{m}=N.
    \end{align*}
    \end{proof}
    Returning to \eqref{SO},  we have, from \cref{L3},
    \begin{align}\label{eq:h-h}
    \IE h(Y)-\IE h(X)=&\sum_{s\in \mathcal{I}}\IE\left[\big(q_Y(Y^{(s,1)}\mid Y)-q_X(Y^{(s,1)}\mid Y)\big)\right] \nonumber \\
    =&\sum_{s\in \mathcal{I}}\IE\left[\frac{e^{\partial_s H(Y)}}{1+e^{\partial_s H(Y)}}-\frac{e^{2\Phi_{\beta}(p)}}{1+e^{2\Phi_{\beta}(p)}}\right],
\end{align}
where $\partial_s H(Y)$ was defined in \cref{r}.
Let $\phi(x)=e^x/(1+e^x)^2$. From the Taylor expansion and fact that $\phi(x)$ is 1-Lipschitz, we have
\begin{align}\label{E0}
     &\bigg|\IE h(Y)-\IE h(X)-\sum_{s\in \mathcal{I}}\IE\left[\big(\partial_s H(Y)-2\Phi_{\beta}(p)\big)\phi(2\Phi_{\beta}(p))\right]\bigg|\notag\\
    \leq& \sum_{s\in \mathcal{I}} \E (\partial_s H(x)-2\Phi_{\beta}(p))^2\leq C n,
\end{align}
where we used \cite[Eq.(5.29)]{fang2024}, \cref{eq:GN84} and \cref{eq:slowrate} in the last inequality.
Moreover, 
  \begin{align}\label{E1}
      &\sum_{s\in \mathcal{I}}\IE\left[\big(\partial_s H(Y)-2\Phi_{\beta}(p)\big)\phi(2\Phi_{\beta}(p))\right]\notag\\
      =&p(1-p) \sum_{s\in \mathcal{I}}\IE\big(\partial_s H(Y)-2\Phi_{\beta}(p)\big)\notag\\
      =&p(1-p)\sum_{s\in\mathcal{I}}\sum_{l=1}^m \beta_j\IE\Big[n^{2-v_l}\operatorname{Hom}(H_l, Y, s)-2e_lp^{e^l-1}\Big]\notag\\
      = &p(1-p)\sum_{s\in\mathcal{I}}\sum_{l=1}^m \beta_l\IE\Big[n^{2-v_l}\operatorname{Hom}(H_l, Y, s)-2e_l\widetilde{p}^{e^l-1}\Big]\\\label{E2}
      &+p(1-p)\sum_{s\in\mathcal{I}}\sum_{l=1}^m \beta_l[2e_l\widetilde{p}^{e^l-1}-2e_lp^{e^l-1}]
  \end{align}
 Recall \eqref{meandiff}:
    \begin{align*}
    |\IE\operatorname{Hom}(H_l, Y, s)-2n^{v_l-2}e_l\widetilde{p}^{e^l-1}|
        \leq Cn^{v_l-2}n^{-1},
    \end{align*}
    which implies that the order of \eqref{E1} is $O(n)$. For $\eqref{E2}$, we have
 \begin{align*}
     \widetilde{p}^{e_l-1}-p^{e_l-1}=( \widetilde{p}^{e_l-2}+\widetilde{p}^{e_l-3}p+\cdots+p^{e_l-2})(\widetilde{p}-p).
 \end{align*}
    For any two positive integers $a$, $b$ satisfying $a+b=e_{l}-2$, we have $$|\widetilde{p}^ap^b-p^{e_l-2}|\leq C|\widetilde{p}-p|.$$ 
    From \cref{eq:slowrate}, we have
    \begin{align}\label{E3}
         \Big|\widetilde{p}^{e_l-1}-p^{e_l-1}-(e_l-1) p^{e_l-2}(\widetilde{p}-p)\Big|\leq C|\widetilde{p}-p|^2\leq Cn^{-1}.
    \end{align}
    Hence we find that \eqref{E2} is equal to $2Np(1-p)(\widetilde{p}-p)\sum_{l=1}^m\beta_le_l(e_l-1)p^{e_l-2}+O(n).$
Since $\IE h(Y)-\IE h(X)= N(\widetilde{p}-p)$, from \eqref{E0}-\eqref{E3}, we find that
\begin{align*}
    \bigg|N(\widetilde{p}-p)-2Np(1-p)\Big(\sum_{l=1}^m\beta_le_l(e_l-1)p^{e_l-2}\Big)(\widetilde{p}-p)\bigg|\leq Cn.
\end{align*}
In the subcritical region, we have $\varphi^{\prime}_{\beta}(p)=2p(1-p)\Phi^{\prime}_{\beta}(p)=2p(1-p)\sum_{l=1}^me_l(e_l-1)p^{e_l-2}< 1$. Therefore, we conclude that $|\widetilde{p}-p|\leq Cn^{-1}$.
\end{proof}

\subsection{Proof of \cref{lem:lclt}}

For an integer-valued random variable $X$, define
\begin{align}\label{locdef}
    D(X)=\sum_{i=-\infty}^{\infty}|\mathbb{P}(X=i+2)-2\mathbb{P}(X=i+1)+\mathbb{P}(X=i)|.
\end{align}
Applying \cite[Theorem 2.2]{rollin2015} with $d_1=d_{\rm{loc}}, d_2=d_{\rm{K}}, m=1, l=2, \beta=1/2$, the local distance between two integer-valued random variables $U$ and $V$ can be bounded as
\begin{align}\label{loc}
    d_{\rm{loc}}(U,V)\leq c d_{\rm{K}}(U,V)^{1/2}[D(U)+D(V)]^{1/2},
\end{align} 
where $c$ denotes a universal constant.
From the definition of $Z_{\mu_n, \sigma_n^2}^{(d)}$ in \cref{Zd} and the fact that $\sigma_n^2\asymp n^2$, we have $D(Z_{\mu_n, \sigma_n^2}^{(d)})\leq c n^{-2}.$ Combining \eqref{Kolo} and \eqref{loc}, it now suffices to prove 
\ben{\label{eq:DEn}
D(E_n)\leq Cn^{-2}.
}
We need the following result in \cite{rollin2015}.
\begin{lemma} (\cite[Theorem~3.7]{rollin2015})\label{l1}
    Let $(X, X^{\prime}, X^{\prime\prime})$ be three consecutive steps of a reversible Markov chain in equilibrium. Let $M: = M(X)$, $M^{\prime} := M(X^{\prime})$, and $M^{\prime\prime} := M(X^{\prime\prime})$ take
values on the integers $\mathbb{Z}$. Define
\begin{align*}
    &Q_{1}(x)=\mathbb{P}\left(M^{\prime}=M+1 \left.\right| X=x\right),\quad q_{1}=\IE Q_{1}(X)=\mathbb{P}(M^{\prime}=M+1),\\
    &Q_{1, 1}(x)=\mathbb{P}\left(M^{\prime}=M+1, M^{\prime \prime}=M^{\prime}+1\left.\right| X=x\right).
\end{align*}
Define $Q_{-1}(x)$ and $Q_{-1,-1}(x)$ similarly by changing all the $+1$ to $-1$.
Then, we have 
    \begin{align*}
       D(X) \leq &\frac{1}{q_{1}^{2}} \Big[2 \operatorname{Var} Q_{1}(X)+\IE\left|Q_{1, 1}(X)-Q_{1}(X)^{2}\right|\\
        &+2 \operatorname{Var} Q_{-1}(X)+\IE\left|Q_{-1,-1}(X)-Q_{-1}(X)^{2}\right|\Big].
    \end{align*}
\end{lemma}
Let $G$ be a random graph following the ERGM \cref{eq:ERGM}. We identify $G$ with its edge indicators $\{Y_{ij}: 1\leq i<j\leq n\}$. 
 Let $G,G^{\prime},G^{\prime\prime}$ be three consecutive steps of the discrete-time Glauber dynamics defined above \cref{tran}, where in each step, one random edge indicator is chosen to be resampled according to the conditional distribution given all the other edge indicators. 
 %Let $E_0=E(G)$, $E_1=E(G^{\prime})$ and $E_2=E(G^{\prime\prime})$. We next calculate $D(E_n)$. 
We will apply \cref{l1} with $(X,X',X'')=(G,G',G'')$ and $M(G)=\text{number of edges in}\ G$
to obtain the following result.

\begin{lemma}\label{l2} In the subset $B$ of the subcritical region considered in \cref{lem:lclt}, we have
\begin{align*}
    q_1&=p(1-p)+O(n^{-1}),\\
    \operatorname{Var}Q_1(G)&=O(n^{-2}),\quad \operatorname{Var}Q_{-1}(G)=O(n^{-2}),\\
   \IE\left|Q_{1, 1}(G)-Q_{1}(G)^{2}\right|&=O(n^{-2}),\quad \IE\left|Q_{-1, -1}(G)-Q_{-1}(G)^{2}\right|=O(n^{-2}).
\end{align*}
\end{lemma}
\begin{proof}[Proof of \cref{l2}]
%Note that
  %\begin{align*}
  %Q_1(G)=N^{-1}\sum_{e\in \mathcal{E}(K_n)}(1-Y_e)\P(G,G_{e+}).
    %\end{align*}
From \eqref{tran}, we have \begin{align*}
     Q_1(G)=&N^{-1}\sum_{e\in \mathcal{I}}(1-Y_e)\frac{\exp(\partial_e H(G))}{1+\exp(\partial_e H(G))},
\end{align*}
where $\mathcal{I}=\{(i,j): 1\leq i<j\leq n\}$.\\
%where $K_n$ is the complete graph with these $n$ vertices.\\
%=&N^{-1}\sum_{e\in \mathcal{E}(K_n)}(1-G_e)p\\&+N^{-1}\sum_{e\in E(K_n)}(1-G_e)\left[\frac{\exp(\partial_e H(G))}{1+\exp(\partial_e H(G))}-\frac{\exp(2\Phi_{\beta}(p))}{1+\exp(2\Phi_{\beta}(p))}\right].

\medskip

\textbf{Estimate $q_1$:} 
From \cref{eq:M0W}, we have 
\begin{align*}
   q_1=&M_{0,+}(\t V,\t E)=\t p\,\t q+O(n^{-1})=pq+O(n^{-1}),
   \end{align*} 
   where we used $|p-\t p|=O(n^{-1})$ (see \cref{eqmu}) in the last equality. \\%that the function $$f(x)=\frac{e^x}{1+e^x}$$ is Lipschitz-1 and concave. In subcritical region, using \cref{L2} and \cref{LH} , we obtain
%\begin{align*}
   % |q_1-p(1-p)|=&N^{-1}\sum_{e\in E(K_n)}|\IE(1-G_e)\left[f(\partial_e H(G))-f(2\Phi_{\beta}(p))\right]|+O(n^{-1})\\
    %\leq &N^{-1}\sum_{e\in E(K_n)}|\IE\left|\partial_e H(G)-2\Phi_{\beta}(p)\right|+O(n^{-1})\\
  %  =&O(n^{-1/2}).
%\end{align*}
 %Now, $q_1=\IE Q_1(G)=p(1-p)+O(n^{-1/2}).$ \\
\medskip
 
 \textbf{Estimate $\operatorname{Var}Q_1(G)$ and $\operatorname{Var}Q_{-1}(G)$:} Denote $\frac{\exp(\partial_e H(G))}{1+\exp(\partial_e H(G))}$  by $g_{e,G}$. Then $\operatorname{Var}Q_1(G)$  can be written as
\begin{align}\label{varcov}
\operatorname{Var}Q_1(G)=&N^{-2}\operatorname{Cov}\left[\sum_{e\in \mathcal{I}}(1-Y_e)g_{e,G},\sum_{f\in \mathcal{I}}(1-Y_f)g_{f,G}\right]\notag\\
=&N^{-2}\sum_{e\in \mathcal{I}}\sum_{f\in \mathcal{I}}\operatorname{Cov}\big[(1-Y_e)g_{e,G},(1-Y_f)g_{f,G}\big].
\end{align}
  Note that $G:=(Y_e)_{e\in \mathcal{I}}$ are positively associated because $\beta_2,\dots, \beta_m>0$ in ERGM \cref{eq:ERGM}. From \cite{newman1980} (see also \cite[Lemma~3.2]{goldstein2018}), for every two real-valued differentiable functions $\phi$ and $\psi$ on $\mathbb{R}^{\mathcal{I}}$,
\begin{align}
   \operatorname{Cov}(\phi(G),\psi(G))\leq  C\sum_{g,h\in \mathcal{I}}\left\|\frac{\partial \phi}{\partial Y_g}\right\|_{\infty}\left\|\frac{\partial \psi}{\partial Y_h}\right\|_{\infty}\operatorname{Cov}(Y_g,Y_h).
\end{align}
Let $\phi_{e}(G)=(1-Y_e)g_{e,G}$ and $\phi_{f}(G)=(1-Y_f)g_{f,G}$. Next, we compute the  partial derivatives of $\phi_{e}$ and $\phi_{f}$.  %Like \cite{fang2024}, 
We have that
\begin{align*}
    \frac{\partial \phi_{e}}{\partial Y_{g}}=-\mathbbm{1}_{e=g}g_{e,G}+(1-Y_e)g_{e,G}\frac{\frac{\partial(\partial_eH(G))}{\partial Y_{g}}}{1+\exp(\partial_eH(G))}.
\end{align*}
which implies 
\begin{align}\label{or}
    \left\|\frac{\partial \phi_{e}}{\partial Y_g}\right\|_{\infty}\leq \mathbbm{1}_{e=g}+\left\|\frac{\partial(\partial_eH(G))}{\partial Y_{g}}\right\|_{\infty}.
\end{align}
Recall from \cref{r} that
\begin{align*}
      \partial_e H(y)=&\sum_{l=1}^m \beta_ln^{2-v_l}\operatorname{Hom}(H_l, y, e).
    \end{align*}
We label the vertices of $H_i$ by $l_1,\cdots,l_{v_i}$, let $[n]=\{1,\cdots,n\}$, and denote by $(e_1,e_2)$ the edge $e$ where $e_1$ and $e_2$ are the two vertices of $e$. Then $\operatorname{Hom}(H_i,G,(e_1,e_2))$ can be written as
\begin{align}\label{dr}
   \sum_{\substack{k_1,\cdots,k_{v_i}\in [n]\\k_1,\cdots,k_{v_i}\;\text{are distinct}}}\sum_{1\leq f\neq j\leq v_j}\mathbbm{1}_{e_1=k_f,e_2=k_j}\mathbbm{1}_{(l_f,l_j)\in \mathcal{E}(H_i)}\prod_{\{l_u,l_v\}\in \mathcal{E}(H_i)\setminus \{l_f,l_j\}}Y_{(k_u,k_v)}.
\end{align}
%We consider the partial derivative with respect of each $g\in \mathcal{E}(K_n)$. 
From \eqref{dr}, it is not difficult to find that if %\brown{\sout{$e$ and $g$  share both two vertices} 
$e=g$, the order of \eqref{or} is  $O(1)$, if $e$ and $g$  share a common vertex, the order of \eqref{or} is  $O(n^{-1})$, and if $e$ and $g$  do not share any vertex, the order of \eqref{or} is  $O(n^{-2})$. 
%\sout{The order of $\|\partial \phi_{f} / \partial Y_h\|_{\infty}$ for any $h\in \mathcal{E}(K_n)$ is similar.} 
Together with a similar argument for $\|\partial \phi_{f} / \partial Y_h\|_{\infty}$, we obtain 
\begin{align}\label{cov11}
    \left\|\frac{\partial \phi_{e}}{\partial Y_g}\right\|_{\infty}=O(n^{v(e\cap g)-2})\quad\textit{and}\quad \left\|\frac{\partial \phi_{f}}{\partial Y_h}\right\|_{\infty}=O(n^{v(f\cap h)-2}).
\end{align}
Furthermore, by \cite[Eqs.(86) and (87)]{ganguly2019}, we deduce that for  two edges $g$ and $h$ 
\begin{align}\label{cov12}
    \operatorname{Cov}(Y_g,Y_h)=O(n^{v(g\cap h)-2}).
\end{align}
Combing \cref{varcov,cov11,cov12}, we obtain 
\begin{align}\label{Varq1}
     \operatorname{Var}Q_1(G)\leq  & CN^{-2}\sum_{e,f,g,h\in \mathcal{I}}\left\|\frac{\partial \phi_{e}}{\partial Y_g}\right\|_{\infty}\left\|\frac{\partial \phi_{f}}{\partial Y_h}\right\|_{\infty}\operatorname{Cov}(Y_g,Y_h)\notag\\
     \leq &CN^{-2}\sum_{e,f,g,h\in \mathcal{I}} n^{v(e\cap g)-2}\cdot n^{v(f\cap h)-2}\cdot n^{v(g\cap h)-2}\notag\\
     \leq &\sup_{e,f,g,h\in \mathcal{I}}CN^{-2}\cdot n^{v(e\cup f\cup g\cup h)}\cdot n^{v(e\cap g)+v(f\cap h)+v(g\cap h)-6}.
\end{align}
According to the inclusion exclusion principle, we have
\begin{align*}
 &v(e\cup f\cup g\cup h)+v(e\cap g)+v(f\cap h)+v(g\cap h)\\
 =&v(e)+v(f)+v(g)+v(h)-v(e\cap f)-v(e\cap h)-v(f\cap g)  +v(e\cap f \cap g)\\
 &+v(e\cap f \cap h) +v(e\cap g \cap h)+v(f\cap g \cap h)-v(e\cap f \cap g\cap h)\\
=:&8-V_1+V_2,
\end{align*}
where
\begin{align*}
    V_1=&v(e\cap f)+v(e\cap h)+v(f\cap g)-(v(e\cap f \cap g)+v(e\cap f \cap h) +v(e\cap g \cap h)),\\
    V_2=&v(f\cap g \cap h)-v(e\cap f \cap g\cap h).
\end{align*}
Apparently, $V_1\geq 0$ and $0\leq V_2\leq 2$. We aim to prove 
\begin{align}\label{v1v2}
    V_2-V_1\leq 0.
\end{align} 
\begin{itemize}
    \item Firstly, if $V_2=0$, then $V_2-V_1\leq 0$ is trivial. 
    \item Secondly, if $V_2=2$, which means $f=g=h$ and $v(e\cap f)=0$, then $V_1=2$ and $V_2-V_1=0$.
    \item Thirdly, we consider the case $V_2=1$ and $v( f \cap g\cap h)=1$. On the one hand, if $f,g,h$ are distinct, then $v(e\cap f \cap g)+v(e\cap f \cap h) +v(e\cap g \cap h)=0$, and $v(e\cap f)+v(e\cap h)+v(f\cap g)\geq 1$, which means $V_2-V_1\leq 0$. On the other hand, we consider the case that two of $f,g,h$ are the same. If $f=g\neq h$, then $v(e\cap f\cap h)=v(e\cap f \cap g\cap h)=0$ and $v(f\cap g)=2$. From $v(e\cap f)+v(e\cap h)-(v(e\cap f \cap g) +v(e\cap g \cap h))\geq 0$, we have $V_2-V_1\leq 1-2=-1$. Similarly, if $f=h\neq g$ or $g=h\neq f$, we also have $V_2-V_1\leq 0$ .
    \item  Finally, if $V_2=1$ and $v( f \cap g\cap h)=2$, which means $f=g=h$ and  $v(e\cap f)=1$, then $V_2-V_1=0$. 
\end{itemize}

 Substituting \eqref{v1v2} into \eqref{Varq1} yields
\begin{align*}
    \operatorname{Var}Q_1(G)\leq CN^{-2}\cdot n^{8-V_1+V_2-6}\leq CN^{-1}.
\end{align*}
%&  \leq &C\sum_{g=h}\left|\frac{\partial \phi_{e}}{\partial Y_g}\right|_{\infty}\left|\frac{\partial \phi_{f}}{\partial Y_h}\right|_{\infty}\operatorname{Cov}(Y_g,Y_h)\\&+C\sum_{g\;\text{and}\;h\;\text{share a vertex}}\left|\frac{\partial \phi_{e}}{\partial Y_g}\right|_{\infty}\left|\frac{\partial \phi_{f}}{\partial Y_h}\right|_{\infty}\operatorname{Cov}(Y_g,Y_h)\\&+C\sum_{g\;\text{and}\;h\;\text{do not share a vertex}}\left|\frac{\partial \phi_{e}}{\partial Y_g}\right|_{\infty}\left|\frac{\partial \phi_{f}}{\partial Y_h}\right|_{\infty}\operatorname{Cov}(Y_g,Y_h)\\\leq&Cn^{4-v(g\cap h)}\cdot n^{v(e\cap g)-2}\cdot n^{v(f\cap h)-2}\cdot n^{v(g\cap h)-2}\\
Similarly, the order of $\operatorname{Var}Q_{-1}(G)$ is also $O(N^{-1}).$\\

\medskip
\textbf{Estimate $ \IE\left|Q_{1, 1}(G)-Q_{1}(G)^{2}\right|$ and $ \IE\left|Q_{-1, -1}(G)-Q_{-1}(G)^{2}\right|$: } Represent $Q_{1,1}(G)$ as
\begin{align*}
    N^{-2}\sum_{e\neq f\in \mathcal{I}}(1-Y_e)(1-Y_f)g_{e,G}g_{f,G\cup e}
\end{align*}
and $Q_{1}^2(G)$ as
\begin{align*}
     N^{-2}\sum_{e, f\in \mathcal{I}}(1-Y_e)(1-Y_f)g_{e,G}g_{f,G}.
\end{align*}
Therefore,
\bes{
Q_{1, 1}(G)-Q_{1}(G)^{2}=&N^{-2}\sum_{e\neq f\in \mathcal{I}}(1-Y_e)(1-Y_f)g_{e,G}(g_{f,G\cup e}-g_{f,G}) \\
&-N^{-2}\sum_{e}(1-Y_e)g^2_{e,G}.
}
If $e$ and $f$ do not share any vertex, then $|g_{f,G\cup e}-g_{f,G}|\leq C/N$. If $e$ and $f$ share one vertex, then $|g_{f,G\cup e}-g_{f,G}|\leq C/n$. In any case, 
\begin{align*}
    \left|Q_{1, 1}(G)-Q_{1}(G)^{2}\right|=O(N^{-1}).
\end{align*}
Similarly, we have 
\begin{align*}
    \left|Q_{-1,-1}(G)-Q_{-1}(G)^{2}\right|=O(N^{-1}).
\end{align*}
%\sout{The proof is complete.}
\end{proof}
Combining \cref{Kolo}, \cref{l1,l2}, we obtain the desired result \cref{eq:DEn} and conclude that
\begin{align*}
    d_{\rm{loc}}(W,Z^d)\leq Cn^{-9/8}.
\end{align*}
%which implies the local limit theorem for exponential random graphs in Dobrushin’s uniqueness region \sout{holds}.

\section{Proof of \cref{LH,LI,lemmaB3}}\label{sec:prooflemma}
\begin{proof}[Proof of Lemma \ref{LH}]

\textbf{Step 1: Higher-order concentration inequalities.} As in \cite{Sam2020}, as building blocks of the Hoeffding decomposition under the ERGM in Dobrushin's uniqueness region, 
    %we build  the Hoeffding decomposition of $\partial_{ij}H-\textit{mean}$ under the ERGM in Dobrushin’s uniqueness region. 
    we define centered random variables
     \begin{align}\label{hoff}
        f_{d,A}(Y) :=& \sum_{I \in \mathcal{I}^d} A_I g_I(Y), \notag\\
        g_I(Y):=&  \sum_{P \in \mathcal{P}(I)} (-1)^{M(P)} \left\{ M(P) 1_{\{N(P)=0\}} + \prod_{\substack{J \in P \\ |J|=1}} (Y_J - \widetilde{p}) 1_{\{N(P)>0\}} \right\} \notag\\&\quad\quad\quad\quad\quad\quad\quad\quad\quad\quad\quad\quad\times \prod_{\substack{J \in P \\ |J|>1}} \left\{ \IE \prod_{l \in J} (Y_l - \widetilde{p}) \right\},
    \end{align}
where \( d \geq 1 \) is an integer, \( \mathcal{I} = \{(i,j) : 1 \leq i < j \leq n\} \), \( A \) is a \( d \)-tensor with vanishing diagonal (i.e., \( A_I \neq 0 \) only if all the \( d \) elements in \( I \) are distinct),

$$\mathcal{P}(I) = \{S \subset 2^I : S \text{ is a partition of } I\},$$ and
\( N(P) \) (\( M(P) \), resp.) is the number of subsets with one element (more than one element, resp.) in the partition \( P \) and for a singleton set \( J = \{l\} \), \( Y_J := Y_l \). We will write \( f_{d,A} := f_{d,A}(Y) \) and \( g_I := g_I(Y) \) for simplicity of notation. 
%For our purpose, it is enough to consider those values of \( d \) bounded by the maximum number of edges of graphs \( H_1, \ldots, H_m \). 

We have the following result.
\begin{lemma}\label{appc} In the compact subset $B$ of the subcritical region, we have
\begin{align}\label{highorder}
    f_{d,A} =O_{r} (\|A\|_2),
\end{align}
where $\|A\|_2$ is the Euclidean norm of the tensor $A$ when viewed as a vector. 
\end{lemma}
In fact, in Dobrushin’s uniqueness region, \cref{highorder} was shown in \cite[Theorem~3.7]{Sam2020} using the modified logarithmic Sobolev inequality. We extend the result to the subcritical region using the Poincar\'e inequality by adapting the proofs in \cite{gotze_higher_2019} and \cite{Sam2020}. See \cref{appendc}.

For $k$ distinct indices
$s_1,\cdots, s_k \in \mathcal{I}$, by considering $Y_{s_i}=\t Y_{s_i}+\t p$, we have
\begin{align}\label{hof}
Y_{s_1}\cdots Y_{s_k} - \IE[Y_{s_1}\cdots Y_{s_k}] = \sum_{l=1}^k \t{p}^{k-l} \sum_{1\leq i_{1} < \cdots < i_l \leq k} [\widetilde{Y}_{s_{i_{1}}}\cdots \widetilde{Y}_{s_{i_l}} - \textit{mean}].
\end{align}
Completing each $\widetilde{Y}_{s_{i_{1}}}\cdots \widetilde{Y}_{s_{i_l}} - \textit{mean}$
to $g_{\{s_{i_{1}} ,\cdots,s_{i_{l}}\}}$ (recall \eqref{hoff})  and using \eqref{GN}, we obtain
%\begin{align}\label{hof2}
%    \sum_{1\leq i_{1} < \cdots < i_l \leq m} \widetilde{Y}_{s_{i_{1}}}\cdots \widetilde{Y}_{s_{i_k}} - \textit{mean}=g_{\{s_{i_{1}} ,\cdots,s_{i_{l}}\}}+\sum_{A:A\subset \{s_{i_{1}} ,\cdots,s_{i_{l}}\}}O(n^{-1})g_{A}
%\end{align}
\begin{align}\label{hof2}
    Y_{s_1}\cdots Y_{s_k} - \IE[Y_{s_1}\cdots Y_{s_k}]=\sum_{l=1}^k \left(\t p^{k-l}+O(\frac{1}{n}) \right)\sum_{1\leq i_1<\cdots<i_l\leq k}g_{\{s_{i_{1}} ,\cdots,s_{i_{l}}\}},
\end{align}
where $O(1/n)$ comes from the last line of \cref{hoff}.

\textbf{Step 2: Decomposition of $\partial_{ij}\bar{H}$.} Recall \cref{r,eq:pijbH}.
We observe that $\partial_{ij}\bar{H}$ is a linear combination of terms as on the left-hand side of \cref{hof2}. Equating those terms to the right-hand side of \cref{hof2}, collecting those terms with $l=2$ and $\{s_{i_1}, s_{i_2}, (i,j)\}$ forms a triangle, or $l=1$ and $s_{i_1}$ connects to either $i$ or $j$, and neglecting the $O(1/n)$ factor results in 
\be{
\partial_{ij}\bar{H}_A=\sum_{s\neq i,j}\frac{2\sum_{l=2}^m \beta_ls_l\widetilde{p}^{e_l-2}(\widetilde{Y}_{is}+\widetilde{Y}_{js})+6\sum_{l=2}^m\beta_lt_l\widetilde{p}^{e_l-3}(\widetilde{Y}_{is}\widetilde{Y}_{sj}-\textit{mean})}{n}
}
as defined in \cref{Ha}. See \cref{fig:enter-label} for an illustration of the terms we collected in $\partial_{ij}\bar{H}_A$. These terms involve only one additional vertex besides $i$ and $j$ and are the dominating part of $\partial_{ij}\bar{H}$.
\begin{figure}[htbp]
    \centering
\includegraphics[width=1.3\linewidth]{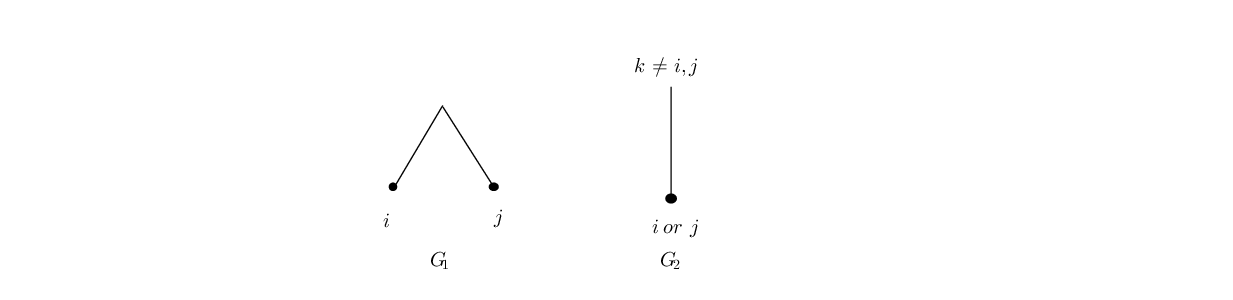}
    \caption{Two graphs in $\ptdH_A$}
    \label{fig:enter-label}
\end{figure}
Because there are $O(n)$ terms (the choice of the additional vertex $s$) in the numerator of $\partial_{ij}\bar{H}_A$, we obtain by \cref{highorder} that
\ben{\label{eq:partialAorder}
\partial_{ij}\bar{H}_A=O_{r}(\sqrt{n}/n)=O_{r}(n^{-1/2}).
}
Similarly, we collect those terms involving two additional vertices besides $i$ and $j$ to form $\partial_{ij}\bar{H}_B$ as in \cref{Hb} and its order is
\be{
\partial_{ij}\bar{H}_B=O_{r}(n/n^2)=O_{r}(n^{-1}).
}
The remaining terms are collected to be $\partial_{ij}\bar{H}_C$. The sum of those terms involving three or more additional vertices besides $i$ and $j$ is of order $O_{r}(n^{-3/2})$ again by \cref{highorder}. Moreover, the sum of those terms involving the $O(1/n)$ factor on the right-hand side of \cref{hof2} that are omitted in $\partial_{ij}\bar{H}_A$ and $\partial_{ij}\bar{H}_B$ are also of order $O(1/n)O_{r}(n^{-1/2})=O_{r}(n^{-3/2})$. Therefore, 
\be{
\partial_{ij}\bar{H}_C=O_{r}(n^{-3/2}).
}
This proves the lemma.
\end{proof}

To prove  \cref{lemmaB3}, we need the following Lemma.
\begin{lemma}\label{lemma:expectation of K}
Under the same assumption as in \cref{lemmaB3},
for any graph $K$ with $v\geq 4$ vertices, we
have 
\begin{equation}
   \label{eq:10001}
  \begin{aligned}
    \E \t K=O(n^{v-2}+n^{v/2+1/2}).
  \end{aligned}
\end{equation}

\end{lemma}
\begin{proof}[Proof of  \cref{lemma:expectation of K}]
 Let
$\mathcal{V}(K)=\{1,2, \ldots ,v\}$.  By the definition of $\t K$, 
  \begin{equation}
     \label{eq:10002}
    \begin{aligned}
      \t K=\frac{1}{\text{Aut}(K)} \sum_{ \substack{i_{1}, \ldots ,i_v=1 \\ i_{1},       \ldots ,i_v  \text{ are distinct} }  }^{ n } \prod_{(s,t)\in \mathcal{E}(K)}
      \widetilde{Y}_{i_s,i_t}  .
    \end{aligned}
  \end{equation}
  Without loss of generality, we assume that $(1,2)\in \mathcal{E}(K)$. Then
  \begin{equation}
     \label{eq:10003}
    \begin{aligned}
    \E \t K &=\frac{1}{\text{Aut}(K)} \sum_{ i_{1},i_{2}=1,i_{1}\neq i_{2} }^{ n }  \E
    \widetilde{Y}_{i_{1},i_{2}} \sum_{  \substack{i_{3}, \ldots ,i_v=1 \\ i_{1},       \ldots ,i_v  \text{ are distinct} }   }^{ n }\prod_{(s,t)\in \mathcal{E}(K)\backslash
    (1,2)}\t Y_{i_s, i_t}\\
    & = \frac{1}{\text{Aut}(K)} \sum_{ i_{1},i_{2}=1,i_{1}\neq i_{2} }^{ n }  \E
    \widetilde{Y}_{i_{1},i_{2}}' (S_{i_{1},i_{2}}-\E S_{i_{1},i_{2}}),
    \end{aligned}
  \end{equation}
  where $\widetilde{Y}_{i_{1},i_{2}}'$ has the conditional distribution
  of $\widetilde{Y}_{i_{1},i_{2}}$ give all the other edge indicators, $\widetilde{Y}_{i_{1},i_{2}}'$
  and $\widetilde{Y}_{i_{1},i_{2}}$ are conditionally independent, and $$S_{i_{1},i_{2}}=
    \sum_{ \substack{i_{3}, \ldots ,i_v=1 \\ i_{1},       \ldots ,i_v  \text{ are distinct} }  }^{ n }\prod_{(s,t)\in \mathcal{E}(K)\backslash
    (1,2)}\t Y_{i_s, i_t}.$$
Following the computation of the conditional expectation of $Y_{i_1,i_2}'$ and using Taylor's expansion as in \cref{eq:M0pV}, we have 
\begin{equation}
   \label{eq:10004}
  \begin{aligned}
 \E \t K = \frac{1}{\text{Aut}(K)} \sum_{ i_{1},i_{2}=1,i_{1}\neq i_{2} }^{ n }  \E (p+
 O(\partial_{i_{1},i_{2}} \bar{H})-\widetilde{p}) (S_{i_{1},i_{2}}-\E S_{i_{1},i_{2}}).
  \end{aligned}
\end{equation}
 We claim that for any $i_{1}$ and $i_{2}$. 
\begin{equation}
   \label{eq:10005}
 \Var(S_{i_{1},i_{2}})=O(n^{v-2}+ n^{2v-7}).
\end{equation}
Then, applying the facts that $\partial_{i_{1},i_{2}} \bar{H}=O_{r}(n^{1/2})$ from \cref{LH}, $p-\t p=O(1/n)$ from \cref{L2} and \cref{eq:10005} in \cref{eq:10004}, we complete the proof of \cref{eq:10001}. 

It
remains to prove \cref{eq:10005}. 
By symmetry, we assume without loss of generality that $i_{1}=1$ and $i_{2}=2$. 
According to the Hoeffding decomposition as in \cref{hof2},
%in Step 1 of the Proof of \cref{LH}, 
we have
\begin{equation}
   \label{eq:10006}
  \begin{aligned}
    S_{1,2}-\E S_{1,2}=&\sum_{ \substack{i_{3}, \ldots ,i_v=1 \\ 1, 2, i_3      \ldots ,i_v  \text{ are distinct} }  }^{ n }g_{\{(i_s, i_t): (s,t)\in K\backslash (1,2)\}}\\
    %+  \sum_{ e_2,e_3\in E } r_{e_2,e_3}
    %f_{1,2,K\backslash \{ (1,2), e_2, e_3 \}} \\
    &+ \dots+ O(n^{-1})\sum_{
    l=1,l\neq 1,2 }^{ n }  (O(n^{v-3})\widetilde{Y}_{1,l}+O(n^{v-3})\widetilde{Y}_{2,l}+
   O(n^{v-3}) \widetilde{Y}_{1,l}\widetilde{Y}_{2,l}).\\
  \end{aligned}
\end{equation}
 %where $f_{1,2,G}$ is defined by taking $A_I=A_{i_{1},i_{2}, \ldots ,i_{v(G)}}$ in \cite[(13)]{Sam2020} equal to $1$ for $(i_{1},i_{2})=(1,2)$ and zero otherwise. 
 Then by \cref{highorder}, we
 have 
 \begin{equation}
    \label{eq:10007}
   \begin{aligned}
     S_{1,2}-\E S_{1,2} = O_{r}(n^{(v-2)/2}+ n^{v-\frac{7}{2}}).
   \end{aligned}
 \end{equation}
 \end{proof}

\begin{proof}[Proof of \cref{lemmaB3}]
 %When $v\geq 5$ and $e\geq 4$, it is a immediate result of \cref{lemma:expectation of K} . If $v\geq 7$ and $e\geq 3$, then $K$ must have isolate vertex. If $v=6$ and $e\geq 3$, only the graph  consisting of three disjoint edges has no isolated vertex. In this case, the covariance of any two edges is $O(n^{-2})$, which implies $\IE \t K=O(n^4)$. If $v=5$ and $e\geq 3$, only the graph  consisting of two disjoint part(a edge and a two star), has no isolated vertex, in this case. The covariance of this edge with any edge of two star is $O(n^{-2})$, which implies $\IE \t K=O(n^3)$. 
The case $v\geq 5$ follows from \cref{lemma:expectation of K}. When $v=3$, $K$ is either a two star or a triangle. For every $i<j,k\neq i,j$, from \cref{eqmu} and \cref{GN}, we have 
\begin{align*}
    \IE \t Y_{ik}\t Y_{jk}=&O(1/n),\\
    \IE \t Y_{ik}\t Y_{jk}\t Y_{ij}=& O(1/n),
\end{align*}
which implies 
\ben{\label{cov1}
\IE \t K=O(n^2)\  \text{when}\  v=3.
}

 It remains to prove $\E \t K=O(n^2)$ when $v=4$. We only prove the case where $K=\tstar$ is the three star and
\be{
\t K=\widetilde{\tstar}=\sum_{i=1}^n\sum_{j<k<l:\atop j,k,l\ne i}\widetilde{Y}_{ij}\widetilde{Y}_{ik}\widetilde{Y}_{il},
}
and the other cases are similar. 
Our strategy is to use the exchangeable pair constructed above \cref{P2} and find an approximate linearity condition
\be{
\E(\t \tstar'-\t \tstar|G)=-\lambda \t \tstar +\text{Err.}
}
Then, $\E \t \tstar=\E(\text{Err.})/\lambda$ will give us an estimate.

%Using the same notation as for \cref{LH}, 
From the construction of the exchangeable pair, we have 
 \begin{align}\label{B1}
     \IE(\widetilde{\tstar}^{\prime}-\widetilde{\tstar}\mid G)=&\frac{1}{N}\IE\sum_{i<j}(\E[\widetilde{Y}_{ij}^{\prime}|G]-\widetilde{Y}_{ij})\sum_{k<l:\atop k,l\neq i,j}(\widetilde{Y}_{ik}\widetilde{Y}_{il}+\widetilde{Y}_{jk}\widetilde{Y}_{jl})\notag\\
   = & -\frac{3}{N}\widetilde{\tstar}+\frac{1}{N}\sum_{i<j}\sum_{k<l:\atop k,l\neq i,j}(\widetilde{Y}_{ik}\widetilde{Y}_{il}+\widetilde{Y}_{jk}\widetilde{Y}_{jl})\E[\widetilde{Y}_{ij}^{\prime}|G].
 \end{align}
Following the same derivation as in \cref{eq:M0pV} and using $\partial_{ij}\bar{H}=O_{r}(n^{-1/2})$ from \cref{LH}, we have
\be{
\E[\widetilde{Y}_{ij}^{\prime}|G]=p-\t p+p(1-p)\partial_{ij}\bar{H}+O(1)(\partial_{ij}\bar{H})^2+O(1)(\partial_{ij}\bar{H})^3+O_{r}(n^{-2}).
}
Plugging this into \cref{B1}, we have
\besn{\label{B2}
&\sum_{i<j}\sum_{k<l:\atop k,l\neq i,j}(\widetilde{Y}_{ik}\widetilde{Y}_{il}+\widetilde{Y}_{jk}\widetilde{Y}_{jl})\E[\widetilde{Y}_{ij}^{\prime}|G]\\
=&\sum_{i<j}\sum_{k<l:\atop k,l\neq i,j}(\widetilde{Y}_{ik}\widetilde{Y}_{il}+\widetilde{Y}_{jk}\widetilde{Y}_{jl})\bigg(p(1-p)\partial_{ij}\bar{H}+O(1)(\partial_{ij}\bar{H})^2+O(1)(\partial_{ij}\bar{H})^3 \bigg)+O_{r}(n^2),
}
where we used \cref{eqmu} and \cref{GN} to absorb the term involving $p-\t p=O(1/n)$ into $O_{r}(n^2)$.

Recall the definition of $\partial_{ij} \bar{H}$ in \cref{eq:pijbH} and its expansion in \cref{eq:partialABC}. Collecting terms according to various graphs they produce (whenever we encounter $\t Y_{ik}^2$, we use the identity \cref{star}), we obtain 
 \begin{align}\label{B3}
    &\sum_{i<j}\sum_{k<l:\atop k,l\neq i,j}(\widetilde{Y}_{ik}\widetilde{Y}_{il}+\widetilde{Y}_{jk}\widetilde{Y}_{jl})\bigg(p(1-p)\partial_{ij}\bar{H}+O(1)(\partial_{ij}\bar{H})^2+O(1)(\partial_{ij}\bar{H})^3 \bigg) \notag\\
    =&6\sum_{r=2}^m(\beta_rs_r)\t p^{e_l-1}(1-\t p)\widetilde{\tstar}+O(1)\t V+O(1)\t \triangle +O(n)\t E\notag\\
    &+\sum_{K: v(K)=4}\frac{O(1)}{n}\t K+\sum_{K: v(K)\geq 5}\frac{O(1)}{n^{v(K)-4}}\t K+O(n^2).
 \end{align}
 %where $ \t \sqcap=\sum_{i<j<k<l}\t Y_{ij}\t Y_{jk}\t Y_{kl}$. 
 In particular, the coefficient in front of $\t \tstar$ comes from the first term in the numerator of $\partial_{ij}\bar{H}_A$ in \cref{Ha}. Also, there is an upper bound for $v(K)$ in the last sum that is independent of $n$.
 %From the covariance $\operatorname{Cov}(Y_{ij},\t Y_{jk}\t Y_{kl})\leq \operatorname{Cov}(Y_{ij},\t Y_{jk})+\operatorname{Cov}(Y_{ij},\t Y_{kl})\leq Cn^{-1}$, 
 %From \cref{lemma:expectation of K}, we have $\IE  \t \sqcap\leq Cn^{5/2}$. 
 
 From exchangeability, we have $ \IE(\widetilde{\tstar}^{\prime}-\widetilde{\tstar})=0$. 
 From the discussion on the denominator of $\sigma_{V_n}^2$ in \cref{eq:sigv}, we know $1-2\sum_{l=2}^m\beta_{ln}s_l\t  p_n^{e_l-1}\t q_n$ is positive and bounded away from 0.
 Taking expectation on both sides of \cref{B1}, 
 we obtain
 \be{
 \E \t \tstar=O(1)\E \t V+O(1)\E \t \triangle+ \sum_{K: v(K)=4}\frac{O(1)}{n}\E \t K+\sum_{K: v(K)\geq 5}\frac{O(1)}{n^{v(K)-4}}\E \t K+O(n^2).
 }
Using the previously established results that $\E\t K=O(n^2)$ when $v(K)=3$ (cf. \cref{cov1}), $\E \t K=O(n^{5/2})$ when $v(K)=4$ and $\E \t K=O(n^{v-2})$ when $v\geq 5$ (cf. \cref{lemma:expectation of K}), we obtain the desired result $\E \t\tstar=O(n^2)$.
\end{proof}

\begin{proof}[Proof of Lemma \ref{LI}]
The argument is similar as for \cref{eq:10007}.
   Let
$\mathcal{V}(K)=\{1,2, \ldots ,v\}$.  By the definition of $\t K$, 
     \begin{align*}
      \widetilde{K}=\frac{1}{\text{Aut}(K)} \sum_{ \substack{i_{1}, \ldots ,i_v=1 \\ i_{1},       \ldots ,i_v  \text{ are distinct} } }^{ n } \prod_{(s,t)\in \mathcal{E}(K)}
      \widetilde{Y}_{i_s,i_t}. 
    \end{align*}
    Following Step 1 in the proof of \cref{LH}, we can decompose $\widetilde{K}-\IE
    \widetilde{K}$ as
    \begin{align*}
     \text{Aut}(K)  (\widetilde{K}-\IE \widetilde{K})= & \sum_{ \substack{i_{1}, \ldots ,i_v=1 \\ i_{1},       \ldots ,i_v  \text{ are distinct} }  }^{ n } \big(\prod_{(s,t)\in \mathcal{E}(K)}
      \widetilde{Y}_{i_s,i_t}  -\textit{mean}\big)\\
      = & f_{K}+O(n^{-1})\sum_{ s=2 }^{ e(K)-1 } \sum_{ e_{1}, \ldots ,e_{s}\in \mathcal{E}(K) }^{ }
      f_{K\backslash \{ e_{1}, \ldots ,e_{s} \}} ,
      %\left(\E \prod_{(u,v) \in \{ e_1, \ldots ,e_s \}} \widetilde{Y}_{u,v}+O(n^{-2})\right), 
      % =& \sum_{ I_{1}, \ldots ,i_5=1 \text{ and }  I_{1},
      % \ldots ,i_5  \text{ are distinct}  }^{ n } \big(1+O(n^{-1})\big)g_E
    \end{align*}
    where the $f$'s are of the form \cref{hoff} and $K\backslash \{ e_{1}, \ldots ,e_{s} \}$ is the graph derived by
    deleting the edges $\{ e_1 ,\ldots ,e_s \}$ (but keeping all the vertices) from $K$. 
    %$K\backslash \{ e_{1},\ldots ,e_{s} \}$ may not be connected and may have many isolate vertices.
    From \cref{highorder}, we obtain that if $K\backslash \{ e_{1}, \ldots ,e_{s} \}$ has $v_{\rm{iso}}(K\backslash \{ e_{1}, \ldots ,e_{s} \})$ isolated vertices, then
    \begin{align*}
      f_{K\backslash \{ e_{1}, \ldots ,e_{s} \}}=O_{r} ( n^{\frac{v(K)+v_{\rm{iso}}(K\backslash \{ e_{1}, \ldots ,e_{s} \})}{2}}).  
    \end{align*}
    Thus,     
 \begin{align*}
   %\label{eq:10009}
  \widetilde{K}-\E \widetilde{K}=  & O_{r}(n^{v/2})+O(n^{-1})\sum_{ s=2 }^{ e(K)-1 } \sum_{ e_{1}, \ldots ,e_{s}\in \mathcal{E}(K) }^{ }
       n^{\frac{v(K)+v_{\rm{iso}}(K\backslash \{ e_{1}, \ldots ,e_{s} \})}{2}}.
\end{align*}
    % In fact, the graph
    % generated by $\{ e_1,e_{2},
    % \ldots ,e_{s}\}$ will have at least $v_{\rm{iso}}(K\backslash \{ e_{1}, \ldots
  % ,e_{s} \})+1$ vertices. 
  Since $v_{\rm{iso}}(K\backslash \{ e_{1}, \ldots
  ,e_{s} \})\leq v(K)-2$ when $s\leq e(K)-1$ (at least one edge left), we obtain the desired estimate $\widetilde{K}-\E \widetilde{K}=O_{r}(n^{v/2}+n^{v-2})$.
  %Substitutuing \eqref{eq:10008} into \eqref{eq:10009} complete the proof.
\end{proof}

\section*{Acknowledgements}

We thank Yi-Kun Zhao for pointing us to the reference \cite{gotze_higher_2019}.
Fang X. was partially supported by Hong Kong RGC GRF 14305821, 14304822, 14303423, 14302124 and a CUHK direct grant. Liu S.H. was partially supported by National Nature Science Foundation of China (Grant No. 12301182).  Su Z. and Wang X. were partially supported by National Nature Science Foundation of China (Grant Nos. 12271475 and U23A2064).

\bibliographystyle{apalike}
%\bibliography{reference}

%\begin{thebibliography}{22}
%\providecommand{\natexlab}[1]{#1}
%\providecommand{\url}[1]{\texttt{#1}}
%\expandafter\ifx\csname urlstyle\endcsname\relax
%  \providecommand{\doi}[1]{doi: #1}\else
%  \providecommand{\doi}{doi: \begingroup \urlstyle{rm}\Url}\fi

%\bibitem{fang2024}
%Fang X, Liu S H, Shao Q M. Normal approximation for exponential random graphs[J]. arXiv preprint arXiv:2404.01666, 2024.

%\bibitem{rollin2015}Röllin A, Ross N. Local limit theorems via Landau–Kolmogorov inequalities[J]. 2015.

%\bibitem{reinert2019}Reinert G, Ross N. Approximating stationary distributions of fast mixing Glauber dynamics, with applications to exponential random graphs[J]. 2019.

%\bibitem{newman1980}Newman C M. Normal fluctuations and the FKG inequalities[J]. Communications in Mathematical Physics, 1980, 74(2): 119-128.
%\bibitem{ganguly2019} Ganguly S, Nam K. Sub-critical Exponential random graphs: Concentration of measure and some applications[J]. arXiv preprint arXiv:1909.11080, 2019.

%\bibitem{dey2023}
%Dey P S, Terlov G. Stein’s method for conditional central limit theorem[J]. The Annals of Probability, 2023, 51(2): 723-773.

%\bibitem{Sam2020} Sambale H, Sinulis A .Logarithmic Sobolev inequalities for finite spin systems and applications[J].Bernoulli, 2020(3): 1863-1890.

\appendix
\section{Computations for \cref{conj1}}\label{sec:multi}

\textbf{Step 1.}
We first sketch a computation leading to a multivariate conditional CLT for the joint distribution of centered triangle and two star counts given the number of edges. As by-products, 
we also get the mean and variance of conditional triangle counts.
We use the same notation as in the proof of \cref{thm:twostar}.

Recall that $\t \triangle =\sum_{i<j<k}\widetilde{Y}_{ij}\widetilde{Y}_{jk}\widetilde{E}_{kl}$ and let $\widetilde{\boldsymbol{W}}=(\t \triangle-\IE \t \triangle  ,\widetilde{V}-\IE \widetilde{V})^{T}.$ Using the same method as that in obtaining the linearity condition for two stars, we obtain (recall $\Delta (\t \triangle)=\t \triangle'-\t \triangle$)

\begin{align*}
 &\IE(\Delta (\t \triangle) \mathbbm{1}_{\Delta\widetilde{E}=1}\mid \t \triangle ,\widetilde{V},\widetilde{E})\\
 =&\frac{1}{N}\IE\left[\sum_{i<j}(\widetilde{q}-\td{Y}_{ij})Y_{ij}^\prime\sum_{k\neq i,j}(\td{Y}_{ik}\td{Y}_{jk})\mid\td{V},\td{E}\right]\\
    =&\frac{1}{N}\IE\left[\sum_{i<j}(\widetilde{q}-\td{Y}_{ij})\sum_{k\neq i,j}(\td{Y}_{ik}\td{Y}_{jk})\frac{e^{\partial_{ij}H(G)}}{1+e^{\partial_{ij}H(G)}}\mid \td{V},\td{E}\right]\\
    =&\frac{1}{N}\IE\bigg[\sum_{i<j}(\widetilde{q}-\td{Y}_{ij})\sum_{k\neq i,j}(\td{Y}_{ik}\td{Y}_{jk})\left(p+p(1-p)\partial_{ij} \bar{H}+O(\partial_{ij} \bar{H})^2\right)\mid \td{V},\td{E}\bigg].
\end{align*}
Considering $\ptdH$ as $\ptdH_A+(\ptdH-\ptdH_A)$ (defined in \cref{Ha}) yields 
\begin{align*}
     \IE(\Delta (\t \triangle) \mathbbm{1}_{\Delta\widetilde{E}=1}\mid \t \triangle ,\widetilde{V},\widetilde{E})=-\frac{3\widetilde{p}}{N}\t \triangle +\frac{\widetilde{p}\widetilde{q}}{N}\widetilde{V}+\frac{6N\t q^4\sum_{l=1}^m\beta_l t_l\t p^{e_l}}{N}+R^{\prime}_{1,+}.
\end{align*}
We also have
\begin{align*}
    \IE(\Delta (\t \triangle) \mathbbm{1}_{\Delta\widetilde{E}=-1}\mid \t \triangle ,\widetilde{V},\widetilde{E})= -\frac{3\widetilde{q}}{N}\t \triangle -\frac{\widetilde{p}\widetilde{q}}{N}\widetilde{V}+\frac{6N\t q^3\sum_{l=1}^m\beta_l t_l\t p^{e_l+1}}{N}+R^{\prime}_{1,-}.
\end{align*}
Combining these two equalities with \cref{P2} implies
%\begin{align*}
    %M_{1,+}(\widetilde{\boldsymbol{W}},\widetilde{E})+ M_{1,-}(\widetilde{\boldsymbol{W}},\widetilde{E}) = -\frac{1}{N} \left( \begin{pmatrix} 3 & 0 \\ 0 & \lambda \end{pmatrix} (\widetilde{\boldsymbol{W}}-\IE\widetilde{\boldsymbol{W}})+\begin{pmatrix} R_{1,+}+ R_{1,-} \\ R^{\prime}_{1,+}+R^{\prime}_{1,-} \end{pmatrix} \right),
%\end{align*}
\begin{align*}
    M_{1,+}(\widetilde{\boldsymbol{W}},\widetilde{E})+ M_{1,-}(\widetilde{\boldsymbol{W}},\widetilde{E}) = - \left( \begin{pmatrix} 3/N & 0 \\ 0 & \lambda \end{pmatrix} (\widetilde{\boldsymbol{W}}-\IE\widetilde{\boldsymbol{W}})+\begin{pmatrix} R^{\prime}_{1,+}+ R^{\prime}_{1,-} \\ \lambda R_{1,+}+\lambda R_{1,-} \end{pmatrix} \right),
\end{align*}
where
\begin{align*}
    \IE\widetilde{\boldsymbol{W}}=(2N\t q^3\sum_{l=1}^m\beta_l t_l\t p^{e_l}+O(n),\IE\widetilde{V})^T,\quad\text{and}\quad \left|\begin{pmatrix} N(R_{1,+}+ R_{1,-}) \\ R^{\prime}_{1,+}+R^{\prime}_{1,-} \end{pmatrix}\right|=O_{r}(n).
\end{align*}
Hence we get
\begin{align}\label{IET}
  \E \t \triangle= 2N\t q^3\sum_{l=1}^m\beta_l t_l\t p^{e_l}+O(n).
\end{align}
Thus we set 
\begin{align}\label{IET'}
 \mu_{\t\triangle}= 2N\t q^3\sum_{l=1}^m\beta_l t_l\t p^{e_l}.
\end{align}
For $M_{2,\pm}((\widetilde{\boldsymbol{W}},\widetilde{E}))$, we similarly get
%\begin{align}\label{covst}
%\IE \left( \left( \begin{array}{cc} 
%\Delta (\t \triangle^2) & \Delta (\t \triangle) \Delta \widetilde{V} \\ 
%\Delta (\t \triangle) \Delta \widetilde{V} & \Delta \widetilde{V}^2 
%\end{array} \right) \mathbbm{1}_{\Delta Y=\pm 1} \Big|\widetilde{\boldsymbol{W}},\widetilde{E}\right) = \frac{1}{N} \left( \left( \begin{array}{cc} 
%3 & 0 \\ 
%0 & \lambda 
%\end{array} \right) \sigma^2_{\boldsymbol{W}} + \Gamma_{2, \pm} \right),
%\end{align}
\begin{align}\label{covst}
\IE \left( \left( \begin{array}{cc} 
\Delta (\t \triangle^2) & \Delta (\t \triangle) \Delta \widetilde{V} \\ 
\Delta (\t \triangle) \Delta \widetilde{V} & \Delta \widetilde{V}^2 
\end{array} \right) \mathbbm{1}_{\Delta Y=\pm 1} \Big|\widetilde{\boldsymbol{W}},\widetilde{E}\right) =  \left( \left( \begin{array}{cc} 
3/N & 0 \\ 
0 & \lambda 
\end{array} \right) \sigma^2_{\boldsymbol{W}} + \Gamma_{2, \pm} \right),
\end{align}
where 
\begin{align}\label{sigmat}
\sigma^2_{\widetilde{\boldsymbol{W}}}=\left( \begin{array}{cc} 
\sigma_{\t \triangle}^2 & 0 \\ 
0 & \sigma_{V}^2 
\end{array} \right),\quad\sigma_{\t \triangle}^2=\frac{Nn\t p^3 \t q^3}{3}\quad\text{and}\quad  |\Gamma_{2,\pm}|=O_{r}(n^{5/2}).
\end{align}
From \eqref{covst} and \eqref{sigmat}, we have
that
\begin{align}\label{covtvtt}
    \operatorname{Cov}(\t V,\t \triangle)=O(n^{5/2}).
\end{align}
Conditioning on $E_n/N=\t p$, our computations above suggest 
\begin{align}\label{muclt}
   \sigma_{\widetilde{\boldsymbol{W}}} ^{-1}\widetilde{\boldsymbol{W}}\longrightarrow N(\boldsymbol{0},I_2) \quad \text{in distribution,}
\end{align}
where $N(\boldsymbol{0},I_2)$ is the two-dimensional standard normal distribution.
%\begin{proof}[Proof of the \cref{conj1}]

\medskip

\textbf{Step 2.}
In this step, we show that the conditional CLT for general subgraph counts reduces to that of a linear combination of $\t \triangle$ and $\t V$.

From \eqref{hof}, conditioning on $E_n/N=\t p$ and denoting $(n)_{(k)}:=n!/(n-k)!$, we have
\begin{align}\label{hoff:H}
  F_{n}-\IE F_{n} =&\sum_{K\subset H,e(K)\geq 2}\frac{(n-v(K))_{(v(H)-v(K))}\cdot \operatorname{Aut}(K) }{\operatorname{Aut}(H)}\big(\t p^{e(H)-e(K)}\big)(\t K-\IE \t K )\notag\\
   =&\left(\frac{(n-3)_{(v(H)-3)}}{\operatorname{Aut}(H)}\right)\big(2s\t p^{e(H)-2} (\t V-\t \E V)+6t\t p^{e(H)-3} (\t \triangle-\E\t \triangle)\big)\notag\\
   &+O_{r}(n^{v(H)-2+\varepsilon}),
   \end{align}
   where we used \cref{LI}, \cref{eq:cond1,eq:cond2} in the second equality for the terms in the summation with $v(K)\geq 4$.

Thus, combining \cref{eq:sigv,sigmat}, we have 
\begin{align}\label{hoff:H1}
  &\frac{1}{\sigma_F}\left(F_{n}-\IE F_{n} -\left(\frac{(n-3)_{(v(H)-3)}}{\operatorname{Aut}(H)}\right)\big(2s\t p^{e(H)-2} (\t V- \E\t V)+6t\t p^{e(H)-3} (\t \triangle-\E\t \triangle)\big)\right)\notag\\
   =&O_{r}(n^{-\frac{1}{2}+\varepsilon}).
   \end{align} 
Therefore, by the conditional joint CLT for two-star and triangle counts obtained in Step 1, it suffices to verify that 
\begin{align}
    \label{sigmaF}
    \sigma_F^2=&(1+o(1))\Var\bigg(\Big(\frac{(n-3)_{(v(H)-3)}}{\operatorname{Aut}(H)}\Big)\big(2s\t p^{e(H)-2} (\t V- \E\t V)\notag\\
   &\qquad\qquad\qquad\qquad\qquad\qquad\qquad+6t\t p^{e(H)-3} (\t \triangle-\E\t \triangle)\big)\mid \frac{E_n}{N}=\t p\bigg)
\end{align}
 and 
 \begin{equation}
     \label{muF}
     \E F_n-\mu_F=o(\sigma_F).
 \end{equation}
 By \eqref{covtvtt}, 
\begin{align*}
    &\text{R.H.S of  \cref{sigmaF}}\\=&\left(\frac{(n-3)_{(v(H)-3)}}{\operatorname{Aut}(H)}\right)^2\left(4s^2\t p^{2e(H)-4}\sigma^{2}_{\t V} +36t^2\t p^{2e(H)-6}\sigma^{2}_{\t \triangle }\right)\\
    &+O(n^{2v(H)-6})\operatorname{Cov}(\t V,\t \triangle\mid \frac{E_n}{N}=\t p)\\
    =&\left(\frac{(n-3)_{(v(H)-3)}}{\operatorname{Aut}(H)}\right)^2\left(4s^2\t p^{2e(H)-4}\sigma^2_{\t V} +36t^2\t p^{2e(H)-6}\sigma^2_{\t \triangle }\right)+O(n^{2v(H)-3.5+\varepsilon})
    \\=&(1+o(1))\sigma_F^2.
\end{align*}
Let 
\be{
G(K,H)= \frac{(n-v(K))_{(v(H)-v(K))}\cdot \operatorname{Aut}(K) }{\operatorname{Aut}(H)}\big(\t p^{e(H)-e(K)}\big).
}
By expanding the edge indicator $Y_{e}$ in $F_{n}$ as $\t Y_{e}+\t p$ and using \cref{lemmaB3}, we have
\begin{align*}
   &\IE [F_{n}\mid \frac{E_n}{N}=\t p]\\=&\frac{n_{(v(H))} }{\operatorname{Aut}(H)}\t p^{e(H)}+ \sum_{K\subset F,e(K)\geq 2} G(K,H)\IE \t K\\
   =&\frac{n_{(v(H))} }{\operatorname{Aut}(H)}\t p^{e(H)}+ \left(\frac{(n-3)_{(v(H)-3)}}{\operatorname{Aut}(H)}\right)\big(2s\t p^{e(H)-2}\IE [\t V\mid \frac{E_n}{N}=\t p]+6t\t p^{e(H)-3}\IE [\t \triangle\mid \frac{E_n}{N}=\t p]\big)\\&\quad+O(n^{v(H)-2+\varepsilon})\\
   =&\frac{n_{(v(H))} }{\operatorname{Aut}(H)}\t p^{e(H)}+ \left(\frac{(n-3)_{(v(H)-3)}}{\operatorname{Aut}(H)}\right)\big(2s\t p^{e(H)-2}\mu_{\t V}+6t\t p^{e(H)-3}\mu_{\t \triangle}\big)+O(n^{v(H)-2+\varepsilon}).
\end{align*}
Hence 
\begin{align*}
    \frac{{\E}F_n}{\sigma_{ F}}=\frac{\mu_{ F}}{\sigma_{ F}}+O(n^{-\frac{1}{2}+\varepsilon}),
\end{align*}
which completes the proof of \cref{muF}.

\section{Computing $c^*$}\label{append}
In this appendix, we give the explicit expression of $c^*$ mentioned in \cref{rem:cstar}.
\begin{lemma}\label{lemmacstar}
Denote the edge set of $H_l$, $1\leq l \leq m$, by $k_{l_1},\dots, k_{l_{e_l}}$.
    Denote by $s(H_l\backslash k_{l_j})$ and $t(H_l\backslash k_{l_j})$ the numbers of two stars and triangles in graph $H_l\backslash k_{l_j}$, respectively. We have $|\widetilde{p}-p-c^* n^{-1}|\leq C n^{-3/2}$ with
\begin{align*}
       c^*=\frac{p(1-p)}{1-\varphi^{\prime}(p)}&\bigg[\sum_{l=1}^m\beta_l\sum_{j=1}^{e_l}\left(4s(H_l\backslash k_{l_j})p^{e_l-3}\frac{q^2\sum_{r=2}^m\beta_rs_rp^{e_r}}{1-2\sum_{r=2}^m\beta_rs_rp^{e_r-1}q}\right.\\
       &\left.\quad\quad\quad\quad\quad\quad+12t(H_l\backslash k_{l_j})p^{e_l-4}q^3\sum_{r=2}^m\beta_rs_rp^{e_r}\right)\\
        &+(1-2p)\left( \frac{8(\sum_{l=2}^m \beta_ls_l p^{e_l-2})^2pq}{(1-2\sum_{l=2}^m\beta_ls_lp^{e_l-1}q)}+36\big(\sum_{l=2}^m \beta_lt_l p^{e_l-2}\big)^2q^2\right)\\
        &-\sum_{l=2}^m\beta_le_lp^{e_l-1}\left(v_l-2\right)\left(v_l-3\right)\bigg].
    \end{align*}
    
\end{lemma}
\begin{proof}[Proof of \cref{lemmacstar}]
From \cref{eq:h-h}, Taylor's expansion and \cref{r}, we have
\begin{align*}
    N\widetilde{p}-Np=&\sum_{s\in \mathcal{I}}\IE(\partial_{s}\bar{H}(Y))p(1-p)+\frac{1}{2}\sum_{s\in \mathcal{I}}\IE(\partial_{s}\bar{H}(Y))^2(p-3p^2+2p^3)\notag\\
    &+O(\sum_{s\in\mathcal{I}}\IE|\partial_{s}\bar{H}(Y)|^3)\\
    &=:I_{1}+I_{2}+I_{3}.
\end{align*}
Firstly, from \cref{eq:partialABC} and \cref{LH}, we have 
\begin{align}\label{OpbarH}
    \partial_{s}\bar{H}(Y)=O_{r}(n^{-1/2}).
\end{align} Thus,
\begin{align*}
    I_{3}\leq CN\cdot n^{-3/2}=O(n^{1/2}).
\end{align*} 
Secondly, by \cref{meandiffbar,OpbarH},
 we have
\begin{align*}
    I_{2}=&\frac{p(p-1)(2p-1)}{2}\sum_{s\in \mathcal{I}}\IE[(\partial_s H(Y)-\E \partial_s H(Y))+(\E \partial_s H(Y)-2\Phi_{\beta}(p))]^2\\
   =& \frac{p(p-1)(2p-1)}{2}\sum_{s\in \mathcal{I}}\operatorname{Var}\left(\partial_{s}H(Y)\right)+O(n^{1/2}).
\end{align*}
Thirdly, using 
Taylor’s expansion for $\Phi_{\beta}(x)$ around $x=p$, $\t p-p=O(1/n)$ and $\varphi'(p)=2p(1-p)\Phi_\beta'(p)$,
\begin{align*}
    I_{1}=& p(1-p)\sum_{s\in \mathcal{I}}\IE\partial_{s}\bar{H}(Y)\\
    =&Np(1-p)(\E\partial_s H(Y)-2\Phi_{\beta}(\widetilde{p}))+2Np(1-p)(\widetilde{p}-p)\Phi_{\beta}^{\prime}(p)+O(1)\\
    =&Np(1-p)(\E \partial_s H(Y)-2\Phi_{\beta}(\widetilde{p}))+N(\widetilde{p}-p)\varphi^{\prime}(p)+O(1).
\end{align*}
Now we get 
\begin{align}\label{ac1}
(1-\varphi^{\prime}(p))(\widetilde{p}-p)=&(\E \partial_s H(Y)-2\Phi_{\beta}(\widetilde{p}))p(1-p)+\frac{p(p-1)(2p-1)}{2}\operatorname{Var}\left(\partial_{s}H(x)\right)\notag\\
&+O(n^{-3/2}).
\end{align}
For $H_l$ consisted of the edges $k_1,\cdots,k_{e_l}$, we have
\begin{align*}
  \IE(\partial_{s}H(x)) =& 2\sum_{l=1}^m\frac{n(n-1)\cdots(n-v_l+3))}{n^{v_l-2}}\beta_l\sum_{j=1}^{e_l} \IE(Y_{k_1}\cdots Y_{k_{j-1}} Y_{k_{j+1}}\cdots Y_{k_{e_l}}),
\end{align*}
thus, 
\begin{align}\label{ac2}
   & \E \partial_s H(Y)-2\Phi_{\beta}(\widetilde{p})\notag\\=&2\sum_{l=1}^m\beta_l\sum_{j=1}^{e_l} \IE\left[Y_{k_1}\cdots Y_{k_{j-1}} Y_{k_{j+1}}\cdots Y_{k_{e_l}}-\widetilde{p}^{e_l-1}\right]\notag\\
    &+2\sum_{l=1}^m\beta_l\sum_{j=1}^{e_l} \left(\frac{n(n-1)\cdots(n-v_l+3))}{n^{v_l-2}}-1\right)\IE(Y_{k_1}\cdots Y_{k_{j-1}} Y_{k_{j+1}}\cdots Y_{k_{e_l}})\notag\\
=&2\sum_{l=1}^m\beta_l\sum_{j=1}^{e_l} \IE\left[Y_{k_1}\cdots Y_{k_{j-1}} Y_{k_{j+1}}\cdots Y_{k_{e_l}}-\widetilde{p}^{e_l-1}\right]\notag\\
    &+2\sum_{l=1}^m\beta_l\sum_{j=1}^{e_l} \left(\frac{n(n-1)\cdots(n-v_l+3))}{n^{v_l-2}}-1\right)\left(\IE(Y_{k_1}\cdots Y_{k_{j-1}} Y_{k_{j+1}}\cdots Y_{k_{e_l}})-\t p^{e_l-1}\right)\notag\\
    &+2\sum_{l=1}^m\beta_l\sum_{j=1}^{e_l} \left(\frac{n(n-1)\cdots(n-v_l+3))}{n^{v_l-2}}-1\right)\t p^{e_l-1}\notag\\
=&2\sum_{l=1}^m(\beta_l+O(n^{-1}))\sum_{j=1}^{e_l} \IE\left[(Y_{k_1}\cdots Y_{k_{j-1}} Y_{k_{j+1}}\cdots Y_{k_{e_l}})-\widetilde{p}^{e_l-1}\right]\notag\\
    &-\sum_{l=1}^m\beta_le_l\t p^{e^l-1}\left(v_l-2\right)\left(v_l-3\right)n^{-1}+O(n^{-2}).
\end{align}
From \cref{lemmaB3}, we know that only if a subset of ${k_1}\cdots {k_{j-1}}, {k_{j+1}}\cdots {k_{e_l}}$ forms a two star or a triangle, the contribution to 
\begin{align*}
  &\IE(Y_{k_1}\cdots Y_{k_{j-1}} Y_{k_{j+1}}\cdots Y_{k_{e_l}})-\t p^{e_l-1}\\
  =&\IE\left[(\widetilde{Y}_{k_1}+\widetilde{p})\cdots (\widetilde{Y}_{k_{j-1}}+\widetilde{p})(\widetilde{Y}_{k_{j+1}}+\widetilde{p})\cdots (\widetilde{Y}_{k_{e_l}}+\widetilde{p})-\widetilde{p}^{e_l-1}\right]
\end{align*}
is of order $O(n^{-1})$, and the other cases have order $O(n^{-2})$. Therefore, from $|p-\t p|=O(n^{-1}),$ $\IE \t V-\mu_{\t V}=O(n)$, $\IE \t \triangle-\mu_{\t \triangle}=O(n)$, \cref{muv,IET}
\begin{align*}
     &\E \partial_s H(Y)-2\Phi_{\beta}(\widetilde{p})\\
     =&2\sum_{l=1}^m\beta_l\sum_{j=1}^{e_l}\left(s(H_l\backslash k_{j})\t p^{e_l-3}(N^{-1}(n-2)^{-1}\IE \widetilde{V})+t(H_l\backslash k_{j})\t p^{e_l-4}(3N^{-1}(n-2)^{-1}\IE \t \triangle )\right)\\
     &-\sum_{l=1}^m\beta_le_l\t p^{e^l-1}\left(v_l-2\right)\left(v_l-3\right)n^{-1}+O(n^{-2})\\
     =&n^{-1}\bigg[4\sum_{l=1}^m\beta_l\sum_{j=1}^{e_l}s(H_l\backslash k_{l_j})p^{e_l-3}\frac{q^2\sum_{r=2}^m\beta_rs_rp^{e_r}}{(1-2q\sum_{r=2}^m\beta_rs_rp^{e_r-1})}\\
     &+\sum_{l=1}^m\beta_l\sum_{j=1}^{e_l}12t(H_l\backslash k_{l_j})p^{e_l-4}q^3\sum_{r=2}^m\beta_rs_rp^{e_r}-\sum_{l=2}^m\beta_le_lp^{e_l-1}\left(v_l-2\right)\left(v_l-3\right)\bigg]\\
     &+O(n^{-2}).
\end{align*}
For $\operatorname{Var}\left(\partial_{s}H(x)\right)$ and $s=(i,j)$, we get from \cref{eq:partialABC}, \cref{Ha} and \cref{LH1} that
\begin{align}\label{ac3}
    \operatorname{Var}(\partial_s H(x))=&\IE\Big[\sum_{r\neq i,j}\sum_{l=2}^m\frac{2\beta_ls_l\widetilde{p}^{e_l-2}(\widetilde{Y}_{ir}+\widetilde{Y}_{jr})}{n}+\sum_{r\neq i,j}\sum_{l=2}^m\frac{6\beta_lt_l\widetilde{p}^{e_l-3}(\widetilde{Y}_{ir}\widetilde{Y}_{rj}-\textit{mean})}{n}\Big]^2\notag\\
    %&\quad+\sum_{v(K\cup s)\geq 3}\frac{C_{\beta,K}f_K}{n^{v(K\cup s)-2}}\Big]^2\notag\\
    &\quad+O(n^{-3/2})\notag\\
    =:&\IE(X_1+X_2)^2+O(n^{-3/2}).
\end{align}
From \cref{lemmaB3}, we have 
\begin{align*}
    &\IE X_1^2=\frac{8(\sum_{l=2}^m \beta_ls_l p^{e_l-2})^2}{n}\left(pq+\frac{\E \t V}{N}\right)+O(n^{-2})=\frac{8(\sum_{l=2}^m \beta_ls_l p^{e_l-2})^2pq}{(1-2\sum_{l=2}^m\beta_ls_lp^{e_l-1}q)n}+O(n^{-2}),\\
    &\IE X_2^2=\frac{36(\sum_{l=2}^m \beta_lt_l p^{e_l-2})^2q^2}{n}+O(n^{-2}),\quad \operatorname{Cov}(X_1,X_2)=O(n^{-3/2}).
\end{align*}
Combining \eqref{IEV}, \eqref{IET} with \eqref{ac1}-\eqref{ac3} , we complete the proof.
\end{proof}

\section{Proof of \cref{appc} using the Poincar\'e inequality}\label{appendc}

From \cite[Theorem~2.1]{ganguly2019} and bounding their transition probabilities $P(x, x^e)$ by 1, we have the following Poincar\'e inequality for the ERGM in the subcritical region:
\begin{equation}\label{eq:PoincareERGM}
    \|f(Y)-\E f(Y)\|_2^2\leq \sigma^2 \E |\mathfrak{h} f(Y)|^2,
\end{equation}
where $Y=\{Y_e, e\in \mathcal{I}\}$, $\mathcal{I}=\{(i,j): 1\leq i<j\leq n\}$, are the edge indicators of the ERGM, $f: \mathcal{I}\to \IR$ is any real-valued function, 
\begin{equation}\label{40001}
|\mathfrak{h} f|=(\sum_{e\in \mathcal{I}}( \nabla_e f)^2)^{1/2},\quad   \mathfrak{h}_e f(x)=f(x_{e+})-f(x_{e-}),
\end{equation}
and $\sigma^2$ is a constant depending only on the compact subset $B$ of the subcritical region.
In \cite{gotze_higher_2019}, it was shown that Poincar\'e inequality leads to higher-order concentration inequalities. Their result was stated for the usual gradient $\nabla$ in $\IR^N$. We modify their proof to obtain the following discrete analog.

To formulate our result, we introduce higher order differences $\mathfrak{h}_{i_1 \ldots i_d}$ for any $d \in \mathbb{N}$ by setting $\mathfrak{h}_{i_1 \ldots i_d} f=\mathfrak{h}_{i_1}\left(\mathfrak{h}_{i_2 \ldots i_d} f\right)$. In particular, we obtain tensors of $d$ th order differences $\mathfrak{h}^{(d)} f$ with coordinates $\mathfrak{h}_{i_1 \ldots i_d} f$. Regarding $\mathfrak{h}^{(d)} f$ as a vector indexed by $\mathcal{I}^d$, we may define $|\mathfrak{h}^{(d)} f|_{\rm{HS}}$ as its Euclidean norm. 
%We will write $\|f\|_p$ for the $L^p(\mu)$ norm of $f$ and $\left\|\mathfrak{h}^{(d)} f\right\|_p:=\left\|\left|\mathfrak{h}^{(d)} f\right|\right\|_p$.

\begin{lemma}\label{lem:GS19} 
From \cref{eq:PoincareERGM}, we have, for any function $g: \mathcal{I}\to \IR$,
 \begin{align*}
     \|g(Y)-\E g(Y)\|_r\leq& \sum_{k=1}^{d-1} (2r\sigma)^k\||\mathfrak{h}^{(k)} g(Y)|_{\rm{HS}}\|_2+(2r\sigma)^d\||\mathfrak{h}^{(d)} g(Y)|_{\rm{HS}}\|_r.
 \end{align*}
\end{lemma}

\begin{remark}
Using a modified logarithmic Sobolev inequality, \cite[Eq.(32)]{Sam2020} obtained the following inequality in Dobrushin's uniqueness region:
\be{
\|g(Y)-\E g(Y)\|_r\leq \sum_{k=1}^{d-1} (r\sigma^2)^{k/2}\||\mathfrak{h}^{(k)} g(Y)|_{\rm{HS}}\|_2+(r\sigma^2)^{d/2}\||\mathfrak{h}^{d} g(Y)|_{\rm{HS}}\|_r.
}
Our result is weaker in terms of the dependence on $r$. 
This is not surprising because we used the Poincar\'e inequality which is weaker than the modified logarithmic Sobolev inequality.
However, it enables use to prove \cref{appc} in the subcritical region.
Following the proof of \cite[Theorem 3.7]{Sam2020}, we can obtain
\begin{align*}
    \|f_{d,A}\|_r\leq (2r\sigma)^d\|A\|_2,
\end{align*}
where $f_{d,A}$ is defined in \cref{hoff}. Therefore, \cref{appc} is proved.
\end{remark}

It remains to prove \cref{lem:GS19}.

\begin{proof}[Proof of \cref{lem:GS19}]
%\begin{align}\label{POI}
%    \Var (f(Y))\leq C\IE \sum_{i\in \mathcal{I}}(\mathfrak{d}_i f(Y))^2,
%\end{align}
%where 
%\begin{align*}
%    \big(\mathfrak{d}_i f(x)\big)^2=\frac{1}{2}\big[q_Y(x^{(i,1)}\mid x)\big(f(x^{(i,1)})-f(x)\big)^2+q_Y(x^{(i,0)}\mid x)\big(f(x^{(i,0)})-f(x)\big)^2\big]
%\end{align*}
%Then from \cref{POI}, we have 
%\begin{equation}
%    \|f(Y)-\E f(Y)\|_2^2\leq \sigma^2 \E |\nabla f(Y)|^2.
%\end{equation}
Let $\{X,Y,X',Y'\}$ be i.i.d., where $Y=\{Y_e, e\in \mathcal{I}\}$ follows the ERGM model. 
Let $X^e$ be the configuration $X$ with edge $e$ flipped (i.e., open$\leftrightarrow$ closed). Let $Y^e$ be defined similarly for $Y$.
For any function $u(\cdot, \cdot): \mathcal{I}\times \mathcal{I}\to \IR$, we have
\begin{align}\label{eq:varianceu}
   & \Var \big(u(X,Y)\big)\nonumber\\= & \frac{1}{2}\E \big(u(X,Y)-u(X',Y')\big)^2 \nonumber\\
    \leq& \E\big(u(X,Y)-u(X,Y')\big)^2+\E\big(u(X,Y')-u(X',Y')\big)^2\nonumber\\
    =&2\E \Var\big(u(X,Y)\mid X\big)+2\E \Var\big(u(X,Y)\mid Y\big)\nonumber\\
    \leq& 2\E\big[\sigma^2\sum_{e\in \mathcal{I}}\big(u(X,Y)-u(X^e,Y)\big)^2\big]+2\E\big[\sigma^2\sum_{e\in \mathcal{I}}\big(u(X,Y)-u(X,Y^e)\big)^2\big],
\end{align}
where we used the Poincar\'e inequality \cref{eq:PoincareERGM} in the last inequality.

%\red{NO NEED?} Let $\nabla_e g(x)=g(x)-g(x^e)$,
%$\nabla g(x)=(\nabla_1 g(x),\dots,\nabla_N g(x))^T,$
%$|\nabla g(x)|_{\rm{HS}}=\Big(\sum_{e\in \mathcal{I}}\big(\nabla_e g(x)\big)^2\Big)^{1/2}.$ Then 
%\begin{align*}
%    \Var \big(u(x,y)\big)\leq 2\sigma^2 \E |\nabla_x u(X,Y)|_{\rm{HS}}+2\sigma^2\E |\nabla_y u(X,Y)|_{\rm{HS}}.
%    \end{align*}

    Letting $u(x,y)=|g(x)-g(y)|^{r/2}\operatorname{sign}\big(g(x)-g(y)\big)$, we have $\E u(X, Y)=0$,
    \begin{align*}
        |u(X,Y)-u(X^e,Y)|=\frac{r}{2}\big[|g(X)-g(Y)|^{r/2-1}+|g(X^e)-g(Y)|^{r/2-1}\big]\cdot |g(X)-g(X^e)|
    \end{align*}
     and, from \cref{eq:varianceu}, 
    \begin{align*}
        &\frac{1}{\sigma^2}\E|g(X)-g(Y)|^r=\frac{1}{\sigma^2}\Var(u(X, Y))\\
        \leq& 2\E\big[\sum_{e\in \mathcal{I}}\big(u(X,Y)-u(X^e,Y)\big)^2\big]+2\E\big[\sum_{e\in \mathcal{I}}\big(u(X,Y)-u(X,Y^e)\big)^2\big]\\
        \leq& \frac{r^2}{2}\sum_{e\in \mathcal{I}} \E\Big[\big[|g(X)-g(Y)|^{\frac{r}{2}-1}+|g(X^e)-g(Y)|^{\frac{r}{2}-1}\big]^2\cdot |g(X)-g(X^e)|^2\Big]\\
        & +\frac{r^2}{2}\sum_{e\in \mathcal{I}} \E\Big[\big[|g(X)-g(Y)|^{\frac{r}{2}-1}+|g(X)-g(Y^e)|^{\frac{r}{2}-1}\big]^2\cdot |g(Y)-g(Y^e)|^2\Big]\\
        \leq  &2r^2 \E \big[|g(X)-g(Y)|^{r-2}\sum_{e\in \mathcal{I}}|g(X)-g(X^e)|^2\big]\\
        &+2r^2 \E \big[|g(X)-g(Y)|^{r-2}\sum_{e\in \mathcal{I}}|g(Y)-g(Y^e)|^2\big]\\
        \leq & 2r^2\Big(\E|g(X)-g(Y)|^r\Big)^{(r-2)/r}\\
        &\times\Big(2^{r/2-1}\E\big(\sum_{e\in \mathcal{I}}(\mathfrak{h}_e g(X))^2\big)^{r/2}+2^{r/2-1}\E\big(\sum_{e\in \mathcal{I}}(\mathfrak{h}_e g(Y))^2\big)^{r/2}\Big)^{2/r}\\
        =&4r^2\big(\E|g(X)-g(Y)|^r\big)^{(r-2)/r}\Big(\E\big(\sum_{e\in \mathcal{I}}(\mathfrak{h}_e g(X))^2\big)^{r/2}\Big)^{2/r},
    \end{align*}
which implies $\|g(X)-g(Y)\|_r\leq 2r\sigma \||\mathfrak{h} g(X)|_{\rm{HS}}\|_r$. By Jensen inequality, $$\|g(X)-g(Y)\|_r\geq \|g-\E g(X)\|_r\geq \|g(X)\|_r-\|g(X)\|_2.$$ Now we obtain
\begin{align}\label{eq:sametoken}
    \|g(X)\|_r\leq \|g(X)\|_2+4r\sigma  \||\mathfrak{h} g(X)|_{\rm{HS}}\|_r.
\end{align}
Applying \cref{eq:sametoken} to $|\mathfrak{h} g(X)|_{\rm{HS}}$, we obtain
 $\||\mathfrak{h} g(X)|_{\rm{HS}}\|_r\leq \||\mathfrak{h} g(X)|_{\rm{HS}}\|_2+4r\sigma \||\mathfrak{h}|\mathfrak{h} g(X)|_{\rm{HS}}|_{\rm{HS}}\|_r$. From the triangle inequality for the norm $|\cdot|_{\rm{HS}}$, it can be verified that $|\mathfrak{h}|\mathfrak{h} g(x)|_{\rm{HS}}|_{\rm{HS}}\leq |\mathfrak{h}^{(2)} g(x)|_{\rm{HS}}.$ Thus,
 \begin{align*}
     \||\mathfrak{h} g(X)|_{\rm{HS}}\|_r\leq \||\mathfrak{h} g(X)|_{\rm{HS}}\|_2+4r\sigma  |\mathfrak{h}^{(2)} g(x)|_{\rm{HS}}.
 \end{align*}
Repeating the above argument, we finally obtain 
 \begin{align*}
     \|g(X)-\E g(X)\|_r\leq& \sum_{k=1}^{d-1} (2r\sigma)^k\||\mathfrak{h}^{(k)} g(X)|_{\rm{HS}}\|_2+(2r\sigma)^d\||\mathfrak{h}^{(d)} g(X)|_{\rm{HS}}\|_r.
 \end{align*}
for each $d\geq 1$. 
%Once we get the above inequality similar to \cite[Eq. 32]{Sam2020}:
%$$\|g(X)-\E g(X)\|_r\leq \sum_{k=1}^{d-1} (r^{1/2}\sigma)^k\||\nabla^k g(X)|_{\rm{HS}}\|_2+(2r\sigma)^d\||\nabla^d g(X)|_{\rm{HS}}\|_r.$$
\end{proof}

\end{document}